\documentclass[12pt]{amsart}

\newtheorem{theorem}{Theorem}[section]
\newtheorem{lemma}[theorem]{Lemma}
\newtheorem{proposition}[theorem]{Proposition}
\newtheorem{corollary}[theorem]{Corollary}

 \theoremstyle{definition}
\newtheorem{definition}[theorem]{Definition}

\theoremstyle{remark}
\newtheorem{remark}[theorem]{Remark}

\numberwithin{equation}{section}

\begin{document}

\title[Conformal deformation of Einstein tensor]
%{On the structure of a class of nonlinear elliptic operators and applications to a priori estimates}
%{On some structures for a class of nonlinear elliptic operators and applications to fully nonlinear PDEs}
 %{Partial uniform ellipticity for some operators and applications to fully nonlinear PDEs}
%{The partial uniform ellipticity for some operators and applications}
%{On the partial uniform ellipticity and its applications}
%{The partial uniform ellipticity for some symmetric functions and applications}
%{On partial uniform ellipticity for some operators and applications}
%{The partial uniform ellipticity,  symmetric functions, and conformal deformation of the Einstein tensor}
%{Partial uniform ellipticity,  symmetric functions, and the conformal deformation of Einstein tensor}
{On the partial uniform ellipticity and complete conformal metrics with prescribed curvature functions 
on manifolds with boundary
}
%{On the partial uniform ellipticity and prescribing  curvature equations  of Einstein tensor on  manifolds with boundary}

 \author{Rirong Yuan}
\address{School of Mathematics, South China University of Technology, Guangzhou 510641, China}
%\email{yuanrr@scut.edu.cn}
\email{yuanrr@scut.edu.cn; rirongyuan@stu.xmu.edu.cn}

\thanks{%The author is partially supported by NSFC (Grant No. 11801587).
% Research of the third named author appreciates the support and hospitality of the Department, and also wants
% to thank Graduate School of Xiamen University for sponsoring his visit to OSU through the Short-term Research Program.
%Research of the author was supported in part by NSFC grant 11801587.
%Research partially supported by NSFC grant 11801587.
 The author was supported by %the National Natural Science Foundation of China, Grant No.
  NSFC Grant 11801587.
}
\date{}

\begin{abstract}
 We consider the problem of finding complete conformal metrics with prescribed curvature functions of the Einstein tensor 
 and of more general modified Schouten tensors.
To achieve this, we  
 reveal an algebraic 
  structure of a wide class of fully nonlinear %elliptic
   equations.
Our method is appropriate and delicate
as shown by a topological obstruction.  
%The asymptotic behavior of complete conformal metrics near boundary is also formulated.
Finally, we discuss Hessian equations and Weingarten equations by confirming a key assumption.

%{\em Mathematical Subject Classification (2010):}
%{\em Mathematics Subject Classification (2010):}
%  Primary 53A30;  Secondary 35J15

%{\em Keywords:} Partial uniform ellipticity, prescribed curvature equation, Einstein tensor, fully nonlinear elliptic equation

\end{abstract}

\maketitle

%\tableofcontents

\section{Introduction}

\medskip

  In this paper, we assume that $(M,g)$ is a compact Riemannian manifold  of dimension $n\geq3$,
   with smooth boundary $\partial M$,  $\bar M=M\cup\partial M$, and with Levi-Civita connection $\nabla$. 
   %($g$ is smooth in $\bar M$).
    Let $\mathrm{Ric}_g$ and 
 $\mathrm{R}_g$ denote the Ricci and scalar curvature of $g$, respectively. 
  The modified Schouten tensor of  $(M,g)$ is simply defined as
   $$S_g^\tau=\frac{1}{n-2}(\mathrm{Ric}_g-\frac{\tau}{2(n-1)}\mathrm{R_g}g),$$
 %The Schouten tensor  corresponds to $\tau=1$.
 which includes the Schouten tensor  and Einstein tensor as  special cases.

  In conformal geometry, it is interesting to find which function could arise as a curvature function of a conformal metric.
%\vspace{1mm}
 For the scalar curvature, according to the Yamabe problem proved by combining the work  
  of Trudinger, Aubin, Schoen \cite{Aubin1976,Schoen1984,Trudinger1968},
 %with the work of Trudinger  (see  \cite{Lee-Parker1987}),
  any closed Riemannian manifold of dimension $n\geq 3$ admits a conformal metric with constant scalar curvature. A fully nonlinear version of Yamabe problem, for the Schouten tensor  $S_g^1$, on closed Riemannian manifolds was proposed in \cite{Viaclovsky2000}, since then 
 %the 
 %conformal deformation of 
 % Schouten tensor  $S_g^1$  %on compact manifolds with or without boundary
it has drawn much attention in
  \cite{Branson2006Gover,ChangGurskyYang2002,Ge2006Wang,GuanLinWang04,Guan2003Wang-CrelleJ,ABLi2003YYLi,Gursky2004Viaclovsky,Gursky2003Viaclovsky-JDG,Gursky2007Viaclovsky,ShengTrudingerWang2007,Viaclovsky2002CAG,Wang2006}.
  % (When $\partial M=\emptyset$ and $f=(\sigma_k)^{\frac{1}{k}}$, 
   The $\sigma_k$-curvature equations for %conformal deformation of 
   modified Schouten tensor $S_g^\tau$, with proper restriction to $\tau$, was studied in \cite{Gursky2003Viaclovsky,Sheng2006Zhang}.
   % for $\tau<1$ and $\lambda(S_g^\tau)\in \Gamma_k$, and by \cite{Sheng2006Zhang} for
   %$\tau>n-1$ and $\lambda(S_g^\tau)\in \Gamma_k$, or $\tau\geq n-1$ and $\lambda(-S_g^\tau)\in \Gamma_k$.
 % (We are referred to \cite{Escobar1992Ann,Escobar1992JDG,Marques2005,Brendle2014Chen,SChen2007,SChen2009,JinLiLi2007} for  Yamabe problem and its fully nonlinear version on manifolds with boundary).
  %\cite{Branson2006Gover,ChangGurskyYang2002-JAM,ChangGurskyYang2002,SChen2005,SChen2007,SChen2009,Ge2006Wang,Guan2007AJM,GuanLinWang04,Guan2003Wang,Guan2003Wang-CrelleJ,JinLiLi2007,ABLi2003YYLi,Gursky2004Viaclovsky,Gursky2007Viaclovsky,ShengTrudingerWang2007,Viaclovsky2002CAG,Wang2006}.
%\vspace{1mm}
When the manifold has boundary,  the Yamabe problem with certain boundary properties was studied by \cite{Cherrier1984,Escobar1992JDG,Escobar1992Ann}, and further complemented by \cite{Escobar1996CVPDE,Escobar1996IUMJ,Han1999Li,Han2000Li,Brendle2014Chen,Marques2005,Marques2007} (see \cite{SChen2007,SChen2009,JinLiLi2007} for a fully nonlinear version), in which the conformal metrics are not 
%necessary 
complete.

%\vspace{1mm}
However, the case is very different when the metric is complete.
%We are concerned with complete conformal metric on manifolds with boundary. 
A deep result %concerning prescribed scalar curvature
% is due to 
of Aviles-McOwen \cite{Aviles1988McOwen}  (see Theorem \ref{thm-AM}),  significantly extending a famous theorem of Loewner-Nirenberg \cite{Loewner1974Nirenberg}, states that any compact Riemannian manifold of dimension $n\geq3$ with smooth boundary admits a complete conformal  metric with negative constant scalar curvature.    
  Also, it would be reasonable to study the fully nonlinear version. %of their results. 
 For  $-\mathrm{Ric}_g$, %$-S_g^\tau$ with $\tau=0$, 
 Guan \cite{Guan2008IMRN} and Gursky-Streets-Warren  \cite{Gursky-Streets-Warren2011}  studied certain fully nonlinear prescribed curvature problem and proved the existence of  complete conformal metrics;
 %with prescribed curvature functions %on compact Riemannian manifolds with boundary
 % (when $\partial M=\emptyset$ it was studied in \cite{Gursky2003Viaclovsky}); 
  moreover,   Gursky-Streets-Warren
  obtained  asymptotic behavior for complete conformal metric with $\sigma_k(-\mathrm{Ric}_g)=1$, extending that of  %Loewner-Nirenberg 
  \cite{Loewner1974Nirenberg} from $\sigma_1$ to $\sigma_k$. 
For the modified Schouten tensor 
 %$f=(\frac{\sigma_k}{\sigma_l})^{\frac{1}{k-l}}$, $0\leq l<k\leq n$, and
 $S_g^\tau$ with $\tau>n-1$, it was also considered by  Li-Sheng  \cite{Li2011Sheng} who proved some existence results.  
% (When $\partial M=\emptyset$ and $f=(\sigma_k)^{\frac{1}{k}}$, the conformal deformation for $S_g^\tau$ was studied by \cite{Gursky2003Viaclovsky} for $\tau<1$ and $\lambda(S_g^\tau)\in \Gamma_k$, and by \cite{Sheng2006Zhang} for $\tau\geq n-1$ and $\lambda(-S_g^\tau)\in \Gamma_k$).

  %settles the case $\tau=n-1$ left open
 
% \vspace{1mm}
 In some  problems from differential geometry and PDEs, 
 it seems interesting to 
 study prescribed curvature problem for the Einstein tensor $$G_g=\mathrm{Ric}_g-\frac{1}{2}\mathrm{R}_g g,$$
which corresponds exactly to \textit{critical} case  $\tau=n-1$ from the viewpoint of PDEs.
As described %in Remark \ref{remark1-crucial} below, 
later, there is a topological obstruction to constructing complete conformal metrics with nonnegative Einstein tensor 
  on certain 3-manifolds, which indicates
% It follows that, 
 that, in this \textit{critical} case, the prescribed curvature problem  is much more complicated.
 %The prescribed curvature problem for the Einstein tensor $$G_g=\mathrm{Ric}_g-\frac{1}{2}\mathrm{R}_g g$$ is much more complicated, as described in Remark \ref{remark1-crucial} below, there is a topological obstruction to the existence of  complete conformal metrics with positive Einstein tensor on certain 3-manifolds. From the viewpoint of PDEs, the curvature equation corresponds exactly to \textit{critical} case  $\tau=n-1$.

%\vspace{1mm}
%In this paper, we are concerned with
The  purpose of this paper % is devoted to studying this problem.
is to settle the \textit{critical} case left open: %We are going to 
finding on $M$ 
 a complete metric ${g}_\infty$, that is conformal to $g$,   with prescribed curvature %function %$\psi$
  \begin{equation}
 \label{conformal-equ0}
\begin{aligned}
\,& f(\lambda_{{g}_{\infty}}(G_{{g}_{\infty}}))=\psi \mbox{ in } M,
\end{aligned}
\end{equation}
where $\lambda_{{g}_\infty}(G_{{g}_\infty})$ denote  eigenvalues of $G_{{g}_\infty}$ with respect to ${g}_\infty$. 
As in \cite{CNS3}, $f$ is a smooth symmetric function of $n$ real variables defined in an open symmetric  convex cone 
$\Gamma\subset\mathbb{R}^n$
with vertex at the origin, $\partial \Gamma\neq \emptyset$, $\Gamma_n\subseteq\Gamma\subset  \Gamma_1,$ 
%$\Gamma\neq\Gamma_1$, 
where 
$\Gamma_k$
 %$\Gamma\supseteq\Gamma_n:=\{\lambda\in \mathbb{R}^n: \mbox{ each } \lambda_i>0\}$;
 is the $k$-th G{\aa}rding's cone defined by 
 $$ \Gamma_k=\left\{\lambda\in \mathbb{R}^n: \sigma_j(\lambda)>0, \mbox{  } \forall 1\leq j\leq k \right\},$$
 where $\sigma_k=\sum_{1\leq i_1<\cdots<i_k\leq n}\lambda_{i_1}\cdots\lambda_{i_k}$.
  % is the $k$-th elementary symmetric function.
 % ($\sigma_0\equiv1$ for simplicity). 
 In addition, %in order to ensure \eqref{mainequ0} to be an elliptic and concave equation, %as in \cite{CNS3},  
 $f$  satisfies % fundamental conditions:
\begin{equation}
\label{concave}
 f \mbox{ is  concave in } \Gamma,
\end{equation}
%and $f$ is a homogeneous  function of degree one, i.e.,
\begin{equation}
\label{homogeneous-1}
\begin{aligned}
f(t\lambda)=tf(\lambda) % \mbox{  } f(\lambda)>0,
 \mbox{ for any } \lambda\in\Gamma, \mbox{  } t>0,
 \end{aligned}
 \end{equation}
\begin{equation}
\label{homogeneous-1-buchong2}
\begin{aligned}
f(\vec{\bf 1})=1, \mbox{  }
  f|_{\Gamma}>0, \mbox{ and }  f|_{\partial \Gamma}\equiv0,
% \mbox{ where } \vec{\bf 1}:=  (1,\cdots, 1)\in \mathbb{R}^n,
 \end{aligned}
 \end{equation}
where  $\vec{\bf 1}:=  (1,\cdots, 1)\in \mathbb{R}^n$. 
Conditions \eqref{concave}, \eqref{homogeneous-1} and \eqref{homogeneous-1-buchong2} yield
\begin{equation}
\label{key1-main}
\begin{aligned}
\,&   \sum_{i=1}^n \lambda_i \geq nf(\lambda),
\,& \forall \lambda\in\Gamma.
\end{aligned}
\end{equation}  
%which is important for studying asymptotic behavior of complete conformal metrics.
  
 % \vspace{1mm}
We assume that there is a  metric $\underline{g}= {e^{2\underline{u}}g}$ with $\underline{u}\in C^2(\bar M)$ such that 
 \begin{equation}
 \label{admissible-metric1}
 \begin{aligned}
 \lambda_{\underline{g}}(G_{\underline{g}})\in \Gamma \mbox{ in } \bar M.
 \end{aligned}
 \end{equation}
 We call such a metric an \textit{admissible conformal} metric for simplicity. 
  
  % \vspace{1mm}
 We prove the following existence results.
%The existence theorems can be stated as follows.
  \begin{theorem}
 \label{thm0-conformal}
  Assume  %\eqref{elliptic},
   \eqref{concave}, \eqref{homogeneous-1}, \eqref{homogeneous-1-buchong2}, 
  \eqref{admissible-metric1} hold.
 Suppose in addition that  $\Gamma\neq \Gamma_n$. %and  there exists an $C^2$-\textit{admissible conformal} metric.
Then for each smooth positive function $\psi$ on $\bar M$,  there is on $M$ a smooth admissible complete metric 
 ${g}_\infty=e^{2u_{\infty}}g$ satisfying
 \eqref{conformal-equ0}.
 
 \end{theorem}

 In particular, 
 \textit{admissible conformal} metrics   exist on certain manifolds  according to Proposition \ref{lemma5-main} and Lemma \ref{remark1-nonzero} below.  
  \begin{theorem}
\label{thm1-conformal}
%For $n\geq 3$.
 Let $(M,g)=(X\times \Omega,g)$ be a Riemannian manifold, where 
  $\Omega$ is a smooth  bounded domain in $\mathbb{R}^k$  and $X$ is a closed Riemannian manifold of dimension $n-k$
   ($1\leq k\leq n$ and $g$ is not necessary the product metric).
Given a smooth positive function 
 $\psi$ on $\bar M$,
there exists a smooth admissible complete conformal Riemannian metric 
 satisfying \eqref{conformal-equ0},  
provided $\Gamma\neq \Gamma_n$ and $f$ satisfies %\eqref{elliptic},
 \eqref{concave}, \eqref{homogeneous-1} and \eqref{homogeneous-1-buchong2}.

\end{theorem}

In the following theorem we obtain the asymptotic behavior and uniqueness  of complete conformal metrics satisfying \eqref{conformal-equ0}, in which  $\psi$ is assumed to a positive constant when restricted to $\partial M$ (rather than $M$), to be compared with previously mentioned results in \cite{Loewner1974Nirenberg,Gursky-Streets-Warren2011};
this extensively generalizes their results. 
Our proof is different and relies on \eqref{key1-main} and $|\nabla \mathrm{d}||_{\partial M}=1$, where $\mathrm{d}(x)$
denotes the distance from $x$ to $\partial M$ with respect to $g$, i.e. $\mathrm{d}(x)=\mathrm{dist}_g(x,\partial M)$.
\begin{theorem}
\label{thm1-asymptotic}
Let $\psi>0$, $\psi|_{\partial M}\equiv1$ and we assume the other assumptions in Theorem \ref{thm0-conformal} hold. Then there exists a unique smooth admissible complete metric $g_\infty=e^{2u_\infty}g$ satisfying  \eqref{conformal-equ0}.  Moreover,
\begin{equation}
\label{asymptotic-rate0}
\begin{aligned}
\lim_{x\rightarrow\partial M} (u_\infty(x)+\log \mathrm{d}(x))=\frac{1}{2}\log\frac{(n-1)(n-2)}{2}. \nonumber
\end{aligned}
\end{equation} 
%where $\mathrm{d}(x)$ denotes the distance from $x$ to $\partial M$ with respect to $g$, i.e. $\mathrm{d}(x)=\mathrm{dist}_g(x,\partial M)$.
 \end{theorem}

\begin{remark}
%A surprising fact is that
It is amazing that in Theorems  \ref{thm0-conformal}, \ref{thm1-conformal} and \ref{thm1-asymptotic}  we do not impose the following condition corresponding the ellipticity of the equation
\begin{equation}
\label{elliptic}
\begin{aligned}
 f_{i}(\lambda):= %f_{\lambda_i}(\lambda)=
 \frac{\partial f}{\partial \lambda_{i}}(\lambda)> 0  \mbox{ in } \Gamma,\  1\leq i\leq n,
 \end{aligned}
\end{equation}
to be compared with fundamental work of Caffarelli-Nirenberg-Spruck \cite{CNS3} in which \eqref{elliptic} was imposed as a crucial condition in the theory of fully nonlinear elliptic equations. %See \cite{Guan2007AJM,Guan12a,LiYY1990,LiYY1991,ShengUrbasWang-Duke,Trudinger90,Trudinger95,Gabor,Urbas2002,yuan2017,yuan2019} for more references.
Surprisingly, we can prove %in Proposition \ref{weak-elliptic} below 
that a weaker condition than \eqref{elliptic} 
 is satisfied %in general context. %More precisely, 
 by $f$ obeying \eqref{concave} and 
  \begin{equation}
\label{addistruc}
\begin{aligned}
\mbox{For each $\sigma<\sup_{\Gamma}f$ and } \lambda\in \Gamma, \mbox{   } \lim_{t\rightarrow +\infty}f(t\lambda)>\sigma.
\end{aligned}
\end{equation}
\begin{proposition}
 \label{weak-elliptic}
If $f$ satisfies \eqref{concave} and \eqref{addistruc}, then 
 \begin{equation}
\label{elliptic-weak}
\begin{aligned}
f_{i}(\lambda)\geq 0, \mbox{  } \forall 1\leq i\leq n,  \mbox{ and } \sum_{i=1}^n f_i(\lambda)>0   \mbox{ in } \Gamma.
 \end{aligned}
\end{equation}
 \end{proposition}
\end{remark}

  \begin{remark}
 \label{remark1-crucial}
It is a remarkable fact that the condition $\Gamma\neq\Gamma_n$ imposed in 
Theorems  \ref{thm0-conformal}, \ref{thm1-conformal} and \ref{thm1-asymptotic} is crucial and cannot be further dropped.
% in general setting.
%It  is not a technical assumption.

%\vspace{1mm}
 Here is a topological obstruction.
% In dimension 3, 
 On a $3$-manifold $(M,g)$, the Einstein tensor connects closely 
 with the sectional curvature:
% For a $3$-manifold $(M,g)$, 
Fix $x\in M$,  
 let $\Sigma\subset T_xM$ be a tangent $2$-plane, %(of dimension 2),
    $K_g(\Sigma)$ be the sectional curvature of $(M,g)$ with respect to  
  $\Sigma$, then (see  \cite{Gursky-Streets-Warren2010})
 \begin{equation} \label{GSW1} \begin{aligned}
 G_g(\vec{\bf n},\vec{\bf n})=-K_g(\Sigma), \nonumber
 \end{aligned}  \end{equation}
 where $\vec{\bf n}\in T_xM$ is a unit normal vector to $\Sigma$.  
Therefore, if Theorem \ref{thm0-conformal} holds for  $\Gamma=\Gamma_n$ (so do both Theorems   \ref{thm1-conformal} and  \ref{thm1-asymptotic}),  then for any conformal class $[g]$
on $\mathbb{S}^2\times (0,1)$ (not necessary the standard one, to be compared with that in \cite{Gursky-Streets-Warren2010}),  
 there would exist in $[g]$ a complete %conformal 
 metric with negative sectional curvature according to Theorem
  \ref{thm1-conformal},
which contradicts to the Cartan-Hadamard theorem.

 This topological obstruction also motivated Gursky-Streets-Warren to prove that any 3-manifold with boundary admits a complete conformal metric satisfying a pinching condition.

 \end{remark}

%When $f$ is a homogeneous of degree one, equation \eqref{conformal-equ0} is specifically reduced to 
In order to prove Theorem \ref{thm0-conformal} it suffices to solve the Dirichlet problem 
  \begin{equation}
  \label{conformal-equation-degree1}
 \begin{aligned}
 f(\lambda (\Delta u g-\nabla^2 u +\frac{1}{n-2}G_g+\frac{n-3}{2}|\nabla u|^2 g+du\otimes du))
 =\frac{\psi e^{2u}}{n-2},
 \end{aligned}
 \end{equation}
 %with %infinity boundary value
  \begin{equation}
 \label{infty-boundary}
\begin{aligned}
%\lim_{x\rightarrow\partial M} u(x)=+\infty,
u|_{\partial M} =+\infty,
  \end{aligned}
\end{equation}
 according to %Under the conformal change , the Einstein tensor obeys
  the  formula % the Einstein tensor of $\tilde{g}$ is given by
   under the conformal transform %change %of Einstein tensor
    $\tilde{g}=e^{2u}g$
 (see e.g. \cite{Besse1987}),
\begin{equation}
 \label{conformal-Einstein-tensor}
\begin{aligned}
\frac{1}{n-2}G_{\tilde{g}}=\frac{1}{n-2}G_g+\Delta u g-\nabla^2 u +\frac{n-3}{2}|\nabla u|^2 g+du\otimes du,
%\frac{2}{n-2}G_{\tilde{g}}=\frac{2}{n-2}G_g+\Delta u g-\nabla^2 u +\frac{n-3}{4}|\nabla u|^2 g+\frac{1}{2}du\otimes du,   \nonumber
\mbox{ if } n\geq 3,
\end{aligned}
\end{equation}
where and hereafter $\nabla u$ and $\nabla^2u$ are gradient and Hessian of $u$ with respect to $g$, 
 $|\nabla u|^2=g(\nabla u,\nabla u)$, 
  $\Delta$ is the Laplacian operator with $\Delta u=\mathrm{tr}_g(\nabla^2 u)$, 
 and $\lambda(A)$ $(=\lambda_g(A))$ denote the eigenvalues of $A$ with respect to %background metric 
 $g$.

   The key issue %ingredient 
  is to derive interior estimates for equation \eqref{conformal-equation-degree1}. 
However, 
 it is rather hard to achieve.
The topological obstruction 
 as described in Remark \ref{remark1-crucial} 
 indicates that when $\Gamma=\Gamma_n$, at least for $n=3$, the estimates for equation \eqref{conformal-equation-degree1}  have to fail,  in which the equation is only of $(n-1)$-uniform ellipticity, as denoted later, according to  
 Proposition \ref{prop31} proved below,  to be compared with those in 
\cite{Guan2008IMRN,Gursky-Streets-Warren2011,Gursky-Streets-Warren2010,Li2011Sheng} where the corresponding curvature equations are clearly of fully uniform ellipticity,   
  thereby proving existence results.
It follows  that
 the difficulty and
  obstruction arise primarily 
 from the lack of fully uniform ellipticity of the curvature equation. % \eqref{conformal-equation-degree1}.
 It is noteworthy that this is different from that of the $k$-Yamabe problem for Schouten tensor studied previously in \cite{ChangGurskyYang2002-JAM} for  $k=2$, $n=4$ and in \cite{Guan2003Wang} for general $k\leq n$ (see also \cite{SChen2005,Guan2004Wang} for other related work).

%\vspace{1mm}
We start our strategy with partial uniform ellipticity. 
%Our strategy starts with partial uniform ellipticity. 
Throughout this article, for simplicity, we call the function
 $f$ is of \textit{$k$-uniform ellipticity}, if $f$ satisfies \eqref{elliptic-weak} and 
there exists a uniform positive constant $\theta$, %(does not depend specifically on $\lambda$), 
such that for any $\lambda\in \Gamma$ with the order $\lambda_1\leq \cdots\leq \lambda_n$, %there holds
\begin{equation}
\label{partial-uniform2}
\begin{aligned}
f_{i}(\lambda)\geq \theta\sum_{j=1}^n f_j(\lambda), \mbox{  } \forall 1\leq i\leq k.
\end{aligned}
\end{equation}
In particular, \textit{$n$-uniform ellipticity} is also called \textit{fully uniform ellipticity}.
Accordingly, we have a similar notion of partial uniform ellipticity for a second order elliptic equation, 
if its linearized operator %(leading symbol) 
satisfies a similar condition.

% \vspace{1mm}
 %We start our strategy with partial uniform ellipticity. 
It was shown in \cite{Lin1994Trudinger} that $\sigma_k$ in $\Gamma_k$ is of $(n+1-k)$-uniform ellipticity, while their proof relies  essentially on the specific properties of $\sigma_k$ which %does 
can not adapt to more general functions.
The first object of our strategy is to settle the general context left open by Lin-Trudinger.
%Our approach is independent of any specific structure of elementary symmetric functions.
%Our proof only utilizes
  We observe that 
  a characterization (see Lemma \ref{lemma3.4}),  extending that initially proposed in  \cite{yuan2019},
   for $f$ satisfying %\eqref{elliptic},
    \eqref{concave} and \eqref{addistruc}
 leads naturally to bridging partial uniform ellipticity of $f$ and
  the count of negative component of a vector contained in $\Gamma$.

\begin{definition}
\label{yuan-kappa}
For the cone $\Gamma$ as stated above, we define
\begin{equation}\begin{aligned}
\kappa_{\Gamma}=\max \left\{l: (-\alpha_1,\cdots,-\alpha_l,\alpha_{l+1},\cdots, \alpha_n)\in \Gamma, \mbox{ where } \alpha_j>0, \mbox{ } \forall j=1,\cdots, n \right\}. \nonumber
\end{aligned}\end{equation}
\end{definition}

Clearly, $\kappa_\Gamma$
 is well defined; in addition, one can see that
\begin{enumerate}
\item $\kappa_\Gamma$ is an integer with $0\leq \kappa_\Gamma\leq n-1$. 
\item $\kappa_{\Gamma}=n-k \mbox{ if }\Gamma=\Gamma_k$.
\item $\kappa_\Gamma\geq 1$ if and only if $\Gamma\neq \Gamma_n$.
\item $\kappa_\Gamma=n-1 \mbox{ if and only if $\Gamma$ is of type 2  in the sense of \cite{CNS3}}.$
\end{enumerate}

\begin{theorem}
\label{yuan-k+1}
%Let $f$ and $\Gamma$ be as above and $\kappa_\Gamma$ be as defined in Definition \ref{yuan-kappa}.
Suppose $f$ satisfies %\eqref{elliptic},
 \eqref{concave} and \eqref{addistruc}.
Then \eqref{elliptic-weak} holds and there is a universal positive constant $\vartheta_{\Gamma}$ depending only on %$n$ and 
$\Gamma$, such that
    for each $ \lambda\in \Gamma$ with the order $\lambda_1 \leq \cdots \leq\lambda_n$,
   \begin{equation}
   \begin{aligned}
   f_{{i}}(\lambda) \geq   \vartheta_{\Gamma} \sum_{j=1}^{n}f_j(\lambda), \mbox{ for all }  1\leq i\leq \kappa_\Gamma+1. \nonumber
    %\mbox{  } \cdots, \mbox{  } f_{{{1+\kappa_{\Gamma}}}}(\lambda) \geq  \vartheta_{\Gamma} \sum_{j=1}^{n}f_j(\lambda).
   \end{aligned}
   \end{equation}
  In particular, $f$ is of fully uniform ellipticity in the case $\kappa_\Gamma=n-1$ (i.e., $\Gamma$ is of type 2).
\end{theorem}

  %  The following theorem %reveals a generic structure  %of fully nonlinear equations 
This theorem states that $f$ is exactly 
of \textit{$(\kappa_\Gamma+1)$-uniform ellipticity},
 which is in effect sharp; %as shown below. %in the following corollary.  %Corollary \ref{thm-sharp} below.
that is
 \begin{corollary}
    \label{thm-sharp}
   The $(\kappa_\Gamma+1)$-uniform ellipticity asserted in Theorem \ref{yuan-k+1} 
% is sharp and % That is $(\kappa_\Gamma+1)$-uniform ellipticity 
 cannot be further improved.
 \end{corollary}
  
   %\vspace{1mm}
  A somewhat surprising fact to us is that the measure of 
partial uniform ellipticity %($\vartheta_{\Gamma}$, $\kappa_\Gamma$)
depends only on the cone  $\Gamma$  instead  specifically on %specific function
the function $f$. %to be compared with those in \cite{Lin1994Trudinger} for $f=(\sigma_k)^{\frac{1}{k}}$. 
Namely, for $\Gamma$ fixed, the constants $\vartheta_{\Gamma}$, $\kappa_\Gamma$ as claimed in Theorem \ref{yuan-k+1} are both universal, and each $f$ as assumed there
%satisfying \eqref{elliptic}, \eqref{concave} and \eqref{addistruc} in $\Gamma$  
is of $(\kappa_\Gamma+1)$-uniform ellipticity with same constant $\vartheta_{\Gamma}$. %and depends only on $\Gamma$.
This is new even for homogenous functions of \eqref{homogeneous-1}.
 As a consequence, we %deduce the following %result.
confirm in Corollary \ref{prop1-operator}  the following inequality in general context  % \eqref{key1-yuan}, 
\begin{equation}
\label{key1-yuan}
\begin{aligned}
f_i(\lambda)\geq \vartheta_{\Gamma} \sum_{j=1}^n f_j(\lambda)  \mbox{ if } \lambda_i\leq 0.
\end{aligned}
\end{equation}
  This inequality was imposed extensively as a vital assumption by many mathematicians \cite{LiYY1991,Trudinger90,ShengUrbasWang-Duke,Urbas2002,Guan2007AJM,SChen2007,SChen2009} to study certain geometric PDEs from classical differential geometry and conformal geometry. %including Weingarten equations, 

% \vspace{1mm}
Theorem \ref{yuan-k+1} suggests  an effective and novel way to  %propose right approach to
 understand the structure of linearized operator of
  a wide class of fully nonlinear %elliptic
   equations.
 Surprisingly, even without assuming \eqref{elliptic}, we can employ it to show in Proposition \ref{prop31} %Section \ref{section4}
  that equation \eqref{conformal-equation-degree1} satisfying \eqref{concave} and \eqref{addistruc}
   is %indeed 
 in effect of fully uniform ellipticity provided $\Gamma\neq\Gamma_n$. 
 % Once we proved \eqref{conformal-equation-degree1} is %indeed 
 % of fully uniform ellipticity,
 As a result, we can apply it to derive interior estimates then to solve the prescribed curvature problem %\eqref{conformal-equ0}
  for Einstein tensor.  The method is then standard, which is analogous to that used in  \cite{Guan2008IMRN,GQY2018,Gursky-Streets-Warren2010,Li2011Sheng}. %where the equations are of fully uniform ellipticity. 
% derive  interior estimates  for solutions of  \eqref{conformal-equation-degree1}.

%\vspace{1mm}
We prove a more general result than Theorems \ref{thm0-conformal} and \ref{thm1-asymptotic}. For 
the %modified Schouten tensors 
$S_g^\tau$ with
\begin{equation}
 \label{assumption-tau}
\begin{aligned}
\tau>1+(n-2)(1-\kappa_\Gamma\vartheta_{\Gamma})
\end{aligned}
\end{equation}
which is rather delicate and appropriate %for the parameter $\tau$ 
as shown in Remarks  \ref{remark1-crucial}  and \ref{remark-important1},
we can find on $M$ a complete conformal metric ${g}_\infty$ by solving the curvature equation 
 \begin{equation}
 \label{conformal-equ0-general} 
\begin{aligned}
f(\lambda_{{g}_\infty}(S_{{g}_\infty}^\tau))=\frac{\psi}{n-2}.
\end{aligned}
\end{equation}

\begin{theorem}
\label{thm0-general}
Suppose 
\eqref{concave}, \eqref{homogeneous-1},  \eqref{homogeneous-1-buchong2} hold and that 
 $\psi$  is a smooth positive function on $\bar M$. 
Assume in addition that the parameter $\tau$  satisfies \eqref{assumption-tau} and 
there is an $C^2$-conformal metric with
\begin{equation}
\label{admissible-metric23}
\begin{aligned}
\lambda(S_{e^{2\underline{u}}g}^\tau)\in \Gamma \mbox{ in } \bar M \mbox{ for some } \underline{u}\in C^2(\bar M). 
\end{aligned}
\end{equation}
Then there is on $M$ a smooth complete conformal metric 
 ${g}_\infty$  
 satisfying 
 \eqref{conformal-equ0-general} and  
 \begin{equation}
\label{conformal-equ0-general-admissible} 
\begin{aligned}
\lambda_{{g}_\infty}(S^\tau_{{g}_\infty})\in \Gamma \mbox{ in } M.
\end{aligned}
\end{equation}
 
 In addition, if $\psi|_{\partial M}=1$ and $\tau\geq2$, then %there is a unique 
 the smooth complete  metric $g_\infty=e^{2u_\infty}g$ satisfying \eqref{conformal-equ0-general} and 
 \eqref{conformal-equ0-general-admissible} 
 is unique.
 Furthermore,
 \begin{equation}
\label{asymptotic-rate0-general}
\begin{aligned}
\lim_{x\rightarrow\partial M} (u_\infty(x)+\log \mathrm{d}(x))=\frac{1}{2}\log\frac{n\tau+2-2n}{2}. \nonumber
\end{aligned}
\end{equation} 
\end{theorem}

In particular, for products as in Theorem \ref{thm1-conformal}, we have the existence result.
\begin{theorem}
  \label{coro73}

Let $(M,g)=(X\times\Omega,g)$ be as in Theorem \ref{thm1-conformal}.
  In addition to  
   \eqref{concave}, \eqref{homogeneous-1},  \eqref{homogeneous-1-buchong2},  
  \eqref{assumption-tau},
we assume    
$\tau\geq2$. Then for any $0<\psi\in C^\infty(\bar M)$,  
there is a smooth complete conformal   metric ${g}_\infty$ 
  satisfying \eqref{conformal-equ0-general} and  \eqref{conformal-equ0-general-admissible}.
 
\end{theorem}

    Throughout this article, $\vartheta_{\Gamma}$ always stands for the constant claimed in 
    Theorem \ref{yuan-k+1},  $\kappa_\Gamma$
denotes the constant defined in Definition \ref{yuan-kappa}.

\vspace{1mm}
The paper is organized as follows. 
In Section \ref{section2} we investigate the partial uniform ellipticity. 
In Section \ref{section4} we 
 apply  the partial uniform ellipticity to verify that a class of fully nonlinear elliptic equations 
 are of fully uniform ellipticity. 
In Section \ref{interiori-estimate} we derive interior estimates for a class of Hessian type equations of fully uniform ellipticity.
The proof relies upon the fully uniform ellipticity that we verified in previous section.
In Section \ref{proofofmainresult} we prove Theorems  \ref{thm0-conformal}, \ref{thm1-conformal} and \ref{thm1-asymptotic}.
In Section \ref{general-existence} Theorems \ref{thm0-general} and \ref{coro73} are further obtained.
Finally, we briefly discuss Hessian equations and Weingarten equations  in Section \ref{section3}.

\section{The partial uniform ellipticity}
\label{section2}

% Let's denote $Df(\lambda):=(f_1(\lambda),\cdots,f_n(\lambda))$.
Throughout this paper, we denote    
   \begin{equation}
 \begin{aligned}
%\,& \sup_{\partial \Gamma}f :=  \sup_{\lambda_{0}\in \partial \Gamma } \limsup_{\lambda\rightarrow \lambda_{0}}f(\lambda),\\
% \vec{\bf 1}:= \,& (1,\cdots, 1)\in \mathbb{R}^n, \mbox{   }
 % Df(\lambda):= \,& (f_1(\lambda),\cdots,f_n(\lambda)), \\
 \Gamma^\sigma:= \,& \{\lambda\in\Gamma: f(\lambda)>\sigma\},   \mbox{  }
  \partial\Gamma^\sigma:= \{\lambda\in\Gamma: f(\lambda)=\sigma\}, \\ \nonumber
% \end{aligned}  \end{equation}
% \begin{equation} \label{Gamma2} \begin{aligned}
  \Gamma^{\infty}_{\mathbb{R}^k}:= \,& \left\{\lambda'\in\mathbb{R}^k: (\lambda',c,\cdots,c)\in\Gamma \mbox{ for some } c>0\right\}. \nonumber
  \end{aligned} \end{equation}
% Clearly, $\Gamma^\infty_{\mathbb{R}^{\kappa_\Gamma}}=\mathbb{R}^{\kappa_\Gamma}$ if $\kappa_\Gamma\geq 1$. 

Our approach is based on the results we proved in this section and Section \ref{section4}.
This section is devoted to %developing 
the investigation of partial uniform ellipticity.
% \vspace{1mm}
The following lemma is the starting point. %of our partial uniform ellipticity. 
The first two statements of this lemma was initially proposed by the author in \cite[Lemma 3.4]{yuan2019} for $f$ satisfying  
 \eqref{concave}, \eqref{elliptic} and \eqref{addistruc}. In this paper we extend it to more general case by removing \eqref{elliptic}.
%among which the equivalence of the first two was asserted in %Lemma 3.4 of
%  \cite{yuan2019}.
% It is a characterization of $f$ satisfying %\eqref{elliptic}, 
% \eqref{concave} and \eqref{addistruc}.
% which plays an important role in the discussion of partial uniform ellipticity of nonlinear elliptic operators. 
 \begin{lemma} 
 \label{lemma3.4} 
 For $f$ satisfying %\eqref{elliptic} and 
 \eqref{concave}, the following %statements 
 are equivalent each other.
 \begin{enumerate}
 \item $f$ further satisfies \eqref{addistruc}.
\item For each $\lambda$, $\mu\in \Gamma,$ $\sum_{i=1}^n f_i(\lambda)\mu_i>0.$
\item $f(\lambda+\mu)>f(\lambda)$ for any $\lambda$, $\mu\in\Gamma$.
\item $f(\lambda+c\vec{\bf 1})>f(c\vec{\bf 1})$, $\forall\lambda\in\Gamma$, $\forall c>0$.
\end{enumerate}
\end{lemma}

\begin{proof}

%$\mathrm{(1)}\Rightarrow \mathrm{(2)}$: 
%$\mathrm{1}\Rightarrow \mathrm{2}$: 
%Let  $Df(\lambda)= (f_1(\lambda),\cdots,f_n(\lambda))$. Condition \eqref{addistruc}  yields for any $\lambda$, $\mu \in \Gamma$, there is $T\geq1$ such that for each $t>T$, there holds $t\mu\in\Gamma^{f(\lambda)}$. Together with the convexity of level set $\partial \Gamma^{f(\lambda)}$, %further one gets $Df(\lambda)\cdot (t\mu-\lambda)>0$. Thus, $Df(\lambda)\cdot\lambda>0$ (if one takes $\mu=\lambda$) and $Df(\lambda)\cdot\mu>0$.
 In the proof, we use the formula saying that for any $\lambda$, $\mu\in\Gamma$,
   \begin{equation}
   \label{concavity1}
   \begin{aligned}
f(\lambda)\geq f(\mu)+\sum_{i=1}^n f_i(\lambda)(\lambda_i-\mu_i),
\end{aligned}
\end{equation}
which follows from the concavity of $f$.

 $\mathrm{(1)}\Rightarrow \mathrm{(2)}$: 
%$\mathrm{1}\Rightarrow \mathrm{2}$: 
Fix $\lambda\in \Gamma$. The condition \eqref{addistruc} implies that for any  
$\mu \in \Gamma$, there is $T\geq1$ (may depend on $\mu$) such that for each $t>T$,
 $f(t\mu)>f(\lambda)$. Together with \eqref{concavity1},
 one gets $\sum_{i=1}^n f_i(\lambda) (t\mu_i-\lambda_i)>0$. Thus, $\sum_{i=1}^nf_i(\lambda)\lambda_i>0$ (if one takes $\mu=\lambda$)
  then $\sum_{i=1}^n f_i(\lambda)\mu_i>0$.
  
$\mathrm{(2)}\Rightarrow \mathrm{(3)}$:
%$\mathrm{2}\Rightarrow \mathrm{3}$: 
 The proof uses \eqref{concavity1}.
%$f(\mu)\geq f(\lambda)+\sum_{i=1}^n f_i(\mu)(\lambda_i-\mu_i), \mbox{   } \forall \lambda,  \mbox{  } \mu\in\Gamma.$
% \begin{equation}\label{concavity1}\begin{aligned}
%f(\mu)\geq f(\lambda)+\sum_{i=1}^n f_i(\mu)(\mu_i-\lambda_i), \mbox{   } \forall \lambda,  \mbox{  } \mu\in\Gamma.
%\end{aligned}\end{equation}
%which follows from \eqref{concavity1}.

%$\mathrm{(3)}\Rightarrow \mathrm{(4)}$: It is trivial.

$\mathrm{(4)}\Rightarrow \mathrm{(1)}$: 
%$\mathrm{4}\Rightarrow \mathrm{1}$: 
Given $\sup_{\partial\Gamma}f<\sigma<\sup_\Gamma f$.  Let  
 $c_\sigma$ be the positive constant with
   \begin{equation}
 \label{kappa1-sigma}
 \begin{aligned}
 f(c_\sigma \vec{\bf 1})=\sigma.
 \end{aligned}
  \end{equation}
   For any $\lambda\in\Gamma$, 
    $t\lambda-c_\sigma\vec{\bf 1}\in\Gamma$ for $t> \frac{\sqrt{n}c_\sigma}{\mbox{dist}(\lambda,\partial\Gamma)}$.
  %Thus $f(t\lambda)\geq f(c_\sigma\vec{\bf 1})+f_i(t\lambda)(t\lambda_i-c_\sigma) > f(c_\sigma\vec{\bf 1})=\sigma$.
  Thus $f(t\lambda)> \sigma$ for such $t$.
\end{proof}

As an interesting result, we obtain Proposition \ref{weak-elliptic}.
%show that \eqref{elliptic-weak} holds for $f$ satisfying \eqref{concave} and \eqref{addistruc}.
% \begin{proposition} \label{weak-elliptic}
%If $f$ satisfies \eqref{concave} and \eqref{addistruc}, then  \eqref{elliptic-weak} holds automatically.
% \end{proposition}
\begin{proof}
[Proof of Proposition \ref{weak-elliptic}]
Fix $\lambda\in\Gamma$.
According to Lemma \ref{lemma3.4}  $\sum_{i=1}^n f_i(\lambda)\mu_i>0$ for any $\mu\in \Gamma$.
%moreover,  $\sum_{i=1}^n f_i(\lambda)\mu_i\geq0$ for  $\mu\in\overline{\Gamma}$.
  By setting $\mu=\vec{\bf 1}$,   $\sum_{i=1}^n f_i(\lambda)>0$. 
Note that $\Gamma_n\subseteq\Gamma$,  %$\overline{\Gamma}_n\subseteq\overline{\Gamma}$
 then \[f_i(\lambda)\geq 0 \mbox{ for each} i.\]
 \end{proof}
 
 \begin{corollary}
 Suppose that $f$ satisfies %\eqref{elliptic},
  \eqref{concave} and \eqref{addistruc}.
 Then for any $\lambda\in\Gamma$, $\sum_{i=1}^n f_i(\lambda)\lambda_i>0$ %and $\sum_{i=1}^n f_i(\lambda)>0$.
  i.e., $f(t\lambda)$ is a strictly increasing function in $t\in\mathbb{R}^+$.
 \end{corollary}

Fix $\lambda\in\Gamma$. 
%For $i\neq j$, let $\mu_i=\lambda_j$, $\mu_j=\lambda_i$,  $\mu_k=\lambda_k$ for $k\neq i, j$, then
% $f(\mu)=f(\lambda)$ according to  the symmetry of $f$. Together with \eqref{concavity1}, 
It follows from \eqref{concavity1} and the symmetry of $f$ that 
 \[f_i(\lambda)\geq f_j(\lambda) \mbox{ for } \lambda_i\leq\lambda_j.\]
Therefore,   Theorem \ref{yuan-k+1} for the case $\Gamma=\Gamma_n$ $(\mbox{i.e. } \kappa_\Gamma=0)$ immediately follows. For general $\Gamma$, 
it is a consequence of the following proposition.
\begin{proposition}
\label{yuanrr-2}
Assume $\Gamma\neq \Gamma_n$ and  $f$ satisfies %\eqref{elliptic},
 \eqref{concave} and \eqref{addistruc}. 
For the $\kappa_\Gamma$ as defined in Definition \ref{yuan-kappa}, let
  $\alpha_1, \cdots, \alpha_n$ be $n$ strictly positive constants such that 
     $$(-\alpha_1,\cdots,-\alpha_{\kappa_\Gamma}, \alpha_{\kappa_\Gamma+1},\cdots, \alpha_n)\in \Gamma.$$
     In addition,  assume $\alpha_1\geq\cdots\geq \alpha_{\kappa_\Gamma}$.
   Then for each $ \lambda\in \Gamma$ with order $\lambda_1 \leq \cdots \leq\lambda_n$,
    \begin{equation}\label{theta1}
\begin{aligned}
 f_{\kappa_\Gamma+1}(\lambda)\geq\frac{\alpha_1}{(\sum_{i=\kappa_\Gamma+1}^n \alpha_i-\sum_{i=2}^{\kappa_\Gamma}\alpha_i)}f_1(\lambda).
\end{aligned}
\end{equation}
Furthermore, $f_1(\lambda)\geq \frac{1}{n}\sum_{i=1}^n f_i(\lambda).$
\end{proposition}

\begin{proof}
%[Proof of Theorem \ref{yuan-k+1}]
Fix $\lambda\in\Gamma$
%For $i\neq j$, let $\mu_i=\lambda_j$, $\mu_j=\lambda_i$,  $\mu_k=\lambda_k$ for $k\neq i, j$, then $f(\mu)=f(\lambda)$ according to  the symmetry of $f$. Together with \eqref{concavity1}, it follows that if $\lambda_i<\lambda_j$ then $f_i(\lambda)\geq f_j(\lambda)$.
%By \eqref{concave},   
with the order $\lambda_1\leq \lambda_2\leq \cdots \leq\lambda_n$. One then has
 $ f_1(\lambda)\geq f_2(\lambda)\geq \cdots \geq f_n(\lambda)$ and
 $f_1(\lambda)\geq \frac{1}{n}\sum_{i=1}^n f_i(\lambda). $
 By Lemma \ref{lemma3.4},
 \begin{equation}
 \label{good1-yuan}
\begin{aligned}
-\sum_{i=1}^{\kappa_\Gamma} \alpha_i f_i(\lambda)+\sum_{i=\kappa_\Gamma+1}^n \alpha_i f_i(\lambda)>0,
\end{aligned}
\end{equation}
which simply yields $f_{\kappa_\Gamma+1}(\lambda)>   \frac{\alpha_1}{\sum_{i=\kappa_\Gamma+1}^n \alpha_i}f_1(\lambda)$.
In addition, 
one can derive \eqref{theta1} by using iteration and \eqref{good1-yuan}.
%This completes the proof.
 
\end{proof}

The following corollary confirms condition \eqref{key1-yuan} in very general context.
\begin{corollary}
\label{prop1-operator}
%Let $\vartheta_{\Gamma}$ be as in Theorem \ref{yuan-k+1}.  
The condition \eqref{key1-yuan} is satisfied if $f$ satisfies 
%\eqref{elliptic}, 
\eqref{concave} and \eqref{addistruc}. %then
 %there exists a universally positive constant $\vartheta_{\Gamma}$ (same as in Theorem \ref{yuan-k+1}) depending only on $n$ and $\Gamma$, such that 
  % for each $\lambda\in \Gamma$ we have 
%\begin{equation} \label{key1-yuan} \begin{aligned}
%f_i(\lambda)\geq \vartheta_{\Gamma} \sum_{j=1}^n f_j(\lambda)  \mbox{ if } \lambda_i\leq 0.
%\end{aligned}\end{equation}
%Here $\vartheta_{\Gamma}$ is the constant as in Theorem \ref{yuan-k+1}.
 Moreover,  \eqref{addistruc} 
%in Proposition \ref{prop1-operator} 
can be removed if $n=2$.
\end{corollary}

\begin{proof}
%[Proof of Corollary \ref{prop1-operator}]
 % The proof is based on Theorem \ref{yuan-k+1}.
%For simplicity $\kappa=\kappa_\Gamma$ for $\kappa_\Gamma$ defined in Definition \ref{yuan-kappa}.
Fix $\lambda=(\lambda_1,\cdots,\lambda_n)\in \Gamma.$ Without loss of generality, we assume
$\lambda_n\geq \cdots\geq \lambda_1.$
 %Then the concavity yields $f_1(\lambda)\geq \cdots\geq f_n(\lambda).$
By the definition of $\kappa_\Gamma$,  $\lambda_{\kappa_\Gamma+1}>0.$ Therefore, %one deduces %if $\lambda_i\leq 0$ then
$$i\leq \kappa_\Gamma \mbox{ if } \lambda_i\leq0.$$
%Proposition \ref{prop1-operator} 
Therefore,  \eqref{key1-yuan}  follows from Theorem \ref{yuan-k+1}.
\end{proof}
 
 Corollary \ref{thm-sharp} simply follows from Proposition \ref{weak-elliptic} and the following proposition.
 \begin{proposition}
\label{thm-k+1}

In addition to %\eqref{elliptic} and 
\eqref{concave} and \eqref{elliptic-weak}, we assume that
 $f$ is $(k+1)$-uniform ellipticity for some $1\leq k\leq n-1$.  Then $\kappa_\Gamma\geq k$, and so
  $\Gamma^{\infty}_{\mathbb{R}^k}=\mathbb{R}^k.$
  
\end{proposition}

\begin{proof}
%[Proof of Theorem \ref{thm-k+1}]

For $\epsilon, R>0$, set $\lambda_{\epsilon,R}=(\epsilon,\cdots,\epsilon, R, \cdots, R)$, $(\epsilon,\cdots,\epsilon)\in \mathbb{R}^k$, $(R, \cdots, R)\in\mathbb{R}^{n-k}$.
From \eqref{elliptic-weak}, $f_i(\lambda_{\epsilon,R})\geq0$, $1\leq i\leq n$, $\sum_{i=1}^n f_i(\lambda_{\epsilon,R})>0$.
 Fix $\sup_{\partial\Gamma} f<\sigma<\sup_{ \Gamma}f$,  let $a=1+c_\sigma$, $0<\epsilon<a$,
 here $c_\sigma$ obeys \eqref{kappa1-sigma}. Thus,
 \begin{equation}
\begin{aligned}
f(\lambda_{\epsilon,R})\geq
\,& f(a\vec{\bf 1})+\epsilon\sum_{i=1}^{k} f_i(\lambda_{\epsilon,R})+R\sum_{i=k+1}^n  f_i(\lambda_{\epsilon,R})
-a\sum_{i=1}^n f_i(\lambda_{\epsilon,R}) %\mbox{ (according to \eqref{concave}})
\\
>\,& f(a\vec{\bf 1})+(R\theta -a)\sum_{i=1}^{n} f_i(\lambda_{\epsilon,R}) \mbox{ (using $(k+1)$-uniform ellipticity)}
\\
= \,& f(a\vec{\bf 1}) \mbox{ (by setting } R=\frac{a}{\theta})
\\>\,&  \sigma. \nonumber
\end{aligned}
\end{equation}
Here $\theta$ is as in \eqref{partial-uniform2}.  
 So $\lambda_{\epsilon,R}=(\epsilon,\cdots,\epsilon, R,\cdots, R)\in \Gamma^{f(a\vec{\bf 1})}$.
Notice $R=\frac{a}{\theta}$  uniformly depending not on $\epsilon$. 
Let $\epsilon\rightarrow 0^+$, we know $(0,\cdots,0,R,\cdots,R)\in  \overline{\Gamma^{f(a\vec{\bf 1})}}
\subset\Gamma^\sigma\subset\Gamma.$
 By the openess of $\Gamma$, %and the definition of $\kappa_\Gamma$,
  $\kappa_\Gamma\geq k$, equivalently  %and so
$\Gamma^{\infty}_{\mathbb{R}^k}=\mathbb{R}^{k}$.
\end{proof}

%\begin{corollary}  \label{thm-sharp}
% The $(\kappa_\Gamma+1)$-uniform ellipticity asserted in Theorem \ref{yuan-k+1} 
% is sharp and % That is $(\kappa_\Gamma+1)$-uniform ellipticity 
% cannot be further improved.
% \end{corollary}
 %Theorem \ref{thm-k+1}. %it can not be improved to be $(\kappa_\Gamma+2)$-uniform ellipticity.
%The following is a characterization of $(k+1)$-uniform ellipticity. 

 %Proposition \ref{thm-k+1} immediately follows  Theorem \ref{thm-sharp}.
 
% \begin{proof}
% [Proof of Theorem \ref{thm-sharp}]
% If $f$ is $(\kappa_\Gamma+2)$-uniform ellipticity then $\kappa_\Gamma\geq \kappa_\Gamma+1$
% equivalently to $\Gamma^\infty_{\mathbb{R}^{\kappa_\Gamma+1}}=\mathbb{R}^{\kappa_\Gamma+1}$
% according to Proposition \ref{thm-k+1}. This is a contradiction. %which contradicts to the definition of $\kappa_\Gamma$. 

% \end{proof}

%  Also, we have the following as  corollaries of Theorem \ref{yuan-k+1} and Proposition \ref{thm-k+1}.

 \begin{corollary}
\label{coro-type2}
In addition to  \eqref{concave} and \eqref{elliptic-weak},
we assume that $f$ is of fully uniform ellipticity. Then the corresponding cone $\Gamma$ is of type 2.
%In addition, if $f$ further satisfies \eqref{addistruc} then  these two are equivalent.
\end{corollary}

\begin{corollary}
\label{thm-type2}
%Let $f$ and $\Gamma$ be as in Introduction.
If $\Gamma$ supposes $f$ satisfying \eqref{concave} %\eqref{elliptic} 
 and \eqref{addistruc},
 then the following statements are equivalent each other.
\begin{enumerate}
\item $\Gamma$ is of type  2.
%\item $\Gamma^{\infty}_{\mathbb{R}^{n-1}}=\mathbb{R}^{n-1}$.
%\item $\kappa_\Gamma=n-1.$
\item $f$ is of fully uniform ellipticity.
%There exists a universally positive constant $\theta$ (does not depend specifically on $\lambda$) such that for each $\lambda\in\Gamma$ and  $i$, $f_i(\lambda)\geq \theta \sum_{j=1}^n f_j(\lambda).$
\end{enumerate}
\end{corollary}

The following proposition helps one to compute $\kappa_\Gamma$ defined above. %defined as in Definition \ref{yuan-kappa}.
 Let 
\begin{equation}\begin{aligned}
\widetilde{\kappa}_{\Gamma}=\max \left\{l: (0,\cdots,0,\alpha_{l+1},\cdots, \alpha_n)\in \Gamma, \mbox{ for } \alpha_j>0, \mbox{ }   l+1\leq j\leq n \right\}. \nonumber
\end{aligned}\end{equation}
\begin{proposition}
\label{yuan-kappa-2}
   $\kappa_\Gamma=\widetilde{\kappa}_\Gamma$.
\end{proposition}
\begin{proof}
The case $\Gamma=\Gamma_n$ is true since $\kappa_{\Gamma_n}=0$ and $\widetilde{\kappa}_{\Gamma_n}=0$.
Next, we assume $\Gamma\neq\Gamma_n$.
Assume  $(-\alpha_1,\cdots,-\alpha_{\kappa_\Gamma}, \alpha_{\kappa_\Gamma+1},\cdots,\alpha_n)\in \Gamma$ for
some $\alpha_i>0$. Then $(0,\cdots,0, \alpha_{\kappa_\Gamma+1},\cdots,\alpha_n)\in \Gamma$, 
which gives $\widetilde{\kappa}_\Gamma\geq \kappa_\Gamma$.
 Conversely, assume there are positive constants $\alpha_{\widetilde{\kappa}_\Gamma+1},\cdots,\alpha_n $ such that
$(0,\cdots,0, \alpha_{\widetilde{\kappa}_\Gamma+1},\cdots,\alpha_n)\in \Gamma$. By the openness of $\Gamma$, 
$(-\epsilon,\cdots,-\epsilon, \alpha_{\widetilde{\kappa}_\Gamma+1}-\epsilon,\cdots,\alpha_n-\epsilon)\in \Gamma$, $\alpha_i-\epsilon>0$ for some $0<\epsilon\ll1$, thus $ \kappa_\Gamma\geq \widetilde{\kappa}_\Gamma$.
\end{proof}

\begin{remark}
The results on partial uniform ellipticity and interior estimates obtained in this paper 
allow $n\geq 2$ and \eqref{addistruc} (more general than \eqref{homogeneous-1}); while the theorems concerning conformal metrics hold for $n\geq3$. 

%Clearly, $\vartheta_{\Gamma}\leq \frac{1}{n}$ and $\kappa_\Gamma\leq n-1$.$\vartheta_{\Gamma}=\frac{1}{n}$ and $\kappa_\Gamma=n-1$ cannot occur simultaneity (otherwise $f_i(\lambda)=\frac{1}{n}\sum_{j=1}^n f_j(\lambda)$,  $\forall \lambda\in\Gamma$, $\forall i=1,\cdots, n$).
\end{remark}

%\section{Assumption \eqref{assumption-4} implies fully uniform ellipticity}
\section{The fully uniform ellipticity of certain curvature equations}
\label{section4}

 Throughout this section and Section \ref{interiori-estimate},  $\psi(x,z,p)$ and $A(x,z,p)$ (as well as $W(x,z,p)$) denote
 respectively  smooth function and smooth symmetric $(0,2)$-type tensor 
 of variables $(x,z,p)$,  $(x,p)\in T^*\bar M$, $z\in \mathbb{R}$.  Moreover,  
 %Here the right-hand side satisfies, for each $w\in C^2(\bar M)$,
% $\psi(x,z,p)$ satisfies, for each $(x,z,p)\in T^*\bar M\times \mathbb{R}$,
 \begin{equation}
 \label{nondegenerate1}
 \begin{aligned}
\mbox{  } \sup_{\partial \Gamma}f<   \psi(x,z, p) 
 <\sup_\Gamma f, \mbox{  } \forall (x,p)\in T^*\bar M, \mbox{  } \forall z\in \mathbb{R},
% \mbox{ in } \bar M,
\end{aligned}
\end{equation}
%for all $x\in \bar M$, 
 %$ C^2$-\textit{admissible} function $w$ as defined in \eqref{admissible-function}, 
where 
 $\sup_{\partial \Gamma}f =  \sup_{\lambda_{0}\in \partial \Gamma } \limsup_{\lambda\rightarrow \lambda_{0}}f(\lambda).$
% \begin{equation} \label{vect1}\begin{aligned}
 %\,& \vec{\bf 1}:=  (1,\cdots, 1)\in \mathbb{R}^n,
%\,&  \sup_{\partial \Gamma}f =  \sup_{\lambda_{0}\in \partial \Gamma } \limsup_{\lambda\rightarrow \lambda_{0}}f(\lambda),
% \,& \sup_\Gamma f=\lim_{t\rightarrow+\infty}f(t\vec{\bf 1}). \nonumber
 %\end{aligned}\end{equation}
% of variables $(x,z,p)$,  $(x,p)\in T^*\bar M$, $z\in \mathbb{R}$.

The purpose of this section is to verify 
a  class of equations is of fully uniform ellipticity. 
The crucial ingredient in proof is the partial uniform ellipticity.

%We  prove the equation \eqref{n-1-equation1} is indeed fully uniform ellipticity if  \eqref{assumption-4}  holds. 
% To achieve this we verify that  \eqref{n-1-equation1} is fully uniform ellipticity if \eqref{assumption-4} holds.

 We start with equation 
\eqref{conformal-equation-degree1}. %as a special case of \eqref{n-1-equation1}.  
It can be rewritten as of the form
 %takes the form %$\widetilde{f}(\lambda(\nabla^2u+A))=\frac{\psi}{n-2}e^{2u}, $
$$\widetilde{f}(\lambda(\nabla^2u+A(x,\nabla u)))=\frac{\psi}{n-2}e^{2u},$$
%\begin{equation}\begin{aligned}
%\widetilde{f}(\lambda(\nabla^2u+A))=\frac{\psi}{n-2}e^{2u}, \nonumber
% \end{aligned} \end{equation}
% where $A=\frac{|\nabla u|^2}{2}g-\frac{R_g}{2(n-1)}g-\frac{G_g}{n-2} -du\otimes du$,  
 with  
  $\widetilde{\Gamma}= \{\lambda:  \mu=(\mu_1,\cdots,\mu_n)\in\Gamma \mbox{ for }   \mu_i=\sum_{j\neq i}\lambda_{j}\}.$
 For such $\widetilde{\Gamma}$, 
 \begin{equation}\label{cone-buchong2}\begin{aligned}
 \kappa_{\widetilde{\Gamma}}=n-1 \mbox{ if } \Gamma\neq \Gamma_n; \mbox{ while } \kappa_{\widetilde{\Gamma}_n}=n-2 \mbox{ if } \Gamma=\Gamma_n.
 \end{aligned}\end{equation}
% First, we can check that if $\varrho=1$ then for the corresponding cone   $\widetilde{\Gamma}= \{\lambda:  \mu=(\mu_1,\cdots,\mu_n)\in\Gamma, \mbox{ for }   \mu_i=\sum_{j\neq i}\lambda_{j}\}$, we have %given by \eqref{map1} below, 
%\begin{equation}\begin{aligned}
%\kappa_{\widetilde{\Gamma}}=n-1 \mbox{ if } \Gamma\neq \Gamma_n; \mbox{ while } \kappa_{\widetilde{\Gamma_n}}=n-2 \mbox{ if } \Gamma=\Gamma_n. \nonumber
%\end{aligned}\end{equation}
From Theorem \ref{yuan-k+1} and Corollary \ref{thm-sharp},
we obtain
\begin{proposition}
\label{prop31}
When $\Gamma\neq\Gamma_n$, equation \eqref{conformal-equation-degree1} is of fully uniform ellipticity for the solutions satisfying $\lambda(G_{\tilde{g}})\in \Gamma$ in $\bar M$; while it is of $(n-1)$-uniform ellipticity for such solutions if $\Gamma=\Gamma_n$.
\end{proposition}

\begin{remark}
Let $\widetilde{\Gamma}$ be as above, and $\widetilde{f}(\lambda)=f(\mu)$.
Assume $f$ satisfies \eqref{concave} and \eqref{addistruc} in $\Gamma$. Then $\widetilde{f}$ is of fully uniform ellipticity in $\widetilde{\Gamma}$, provided $\Gamma\neq\Gamma_n$.
\end{remark}
%\begin{remark}
% \label{remark1-crucial-buchong1} 

%The topological obstruction in Remark \ref{remark1-crucial} %together with  Section \ref{interiori-estimate},
%to deriving the interior estimates for  \eqref{conformal-equation-degree1} with $\Gamma=\Gamma_n$
% indicates  that, when $\Gamma=\Gamma_n$, the estimates for  \eqref{conformal-equation-degree1}  have to fail, at least in dimension 3.
%From this fact, we see that fully uniform ellipticity is crucial for estimates, as the equation \eqref{conformal-equation-degree1}  with $\Gamma=\Gamma_n$ is only $(n-1)$-uniform ellipticity according to   Theorem \ref{yuan-k+1} and Corollary \ref{thm-sharp}. %and \eqref{cone-buchong2}.
%\end{remark} 

%\vspace{1mm}
In fact we can prove a more general result.
\begin{proposition}
\label{keykey}
The equation
 \begin{equation}
\label{n-1-equation1}
\begin{aligned}
f(\lambda(\Delta u g-\varrho\nabla^2 u+W(x,u,\nabla u)))=\psi(x,u,\nabla u) %\nonumber
\end{aligned}
\end{equation}
is of fully uniform ellipticity for the solutions with $\lambda(\Delta u g-\varrho\nabla^2 u+W(x,u,\nabla u))\in \Gamma$ in $\bar M$, provided that  $\varrho$ in the equation %as in \eqref{n-1-equation1} 
is a constant satisfying %an appropriate arrangement
\begin{equation}
\label{assumption-4}
\begin{aligned}
%\Gamma\neq \Gamma_n  (\mbox{i.e. }  \kappa_\Gamma\geq 1) \mbox{ and }
  \varrho<\frac{1}{1-\kappa_\Gamma \vartheta_{\Gamma}} \mbox{ and } \varrho\neq 0.
  %\mbox{  where $\vartheta_{\Gamma}$ is as in Theorem \ref{yuan-k+1}.}
 \end{aligned}
\end{equation}
 %Here  $W(x,z,p)$ is a smooth symmetric $(0,2)$-type tensor, $\psi(x,z,p)$ is a smooth function satisfying 

 \end{proposition}

%\begin{proposition} \label{keykey}
%Equation \eqref{n-1-equation1} %(equivalently, equation \eqref{n-1-equation1}) 
 % is of fully uniform ellipticity, provided \eqref{assumption-4} holds.
%\end{proposition}

\begin{proof}
For $u$ we denote 
 \begin{equation}
\begin{aligned}
 \,& U[u]=\Delta u g-\varrho\nabla^2 u+W(x,u,\nabla u), \,& \mathfrak{g}[u]=\nabla^2u+A(x,u,\nabla u), \nonumber
 \end{aligned}
\end{equation}
where $A(x,u,\nabla u)=\frac{\mathrm{tr}_{g}W(x,u,\nabla u)}{\varrho(n-\varrho)}g-\frac{W(x,u,\nabla u)}{\varrho}.$
(Notice  \eqref{assumption-4} yields $\varrho<n$, $\varrho\neq 0$, as $\vartheta_{\Gamma}\leq \frac{1}{n}$, $\kappa_\Gamma\leq n-1$).
 %For simplicity ******we denote $U=U[u] \mbox{ and } {\underline{U}}=U[\underline{u}]$ for $u$ and $\underline{u}$, respectively.
We denote  eigenvalues by 
 \begin{equation}
\begin{aligned}
 \,& \lambda(U[u])=(\mu_1,\cdots,\mu_n),  \,& \lambda(\mathfrak{g}[u])=(\lambda_1,\cdots,\lambda_n). \nonumber
 \end{aligned}
\end{equation}
% \begin{equation}\label{eigenvalues-denote1}\begin{aligned} \lambda(\mathfrak{g}[u])=(\mu_1,\cdots,\mu_n) \mbox{ and } \lambda(U[u])=(\lambda_1,\cdots,\lambda_n). \nonumber \end{aligned}\end{equation}
Then one can check $\mu_i=\sum_{j=1}^n \lambda_j -\varrho\lambda_i$.
Let $P': \Gamma \longrightarrow P'(\Gamma)=:\tilde{\Gamma}$ be a map 
\begin{equation}
\label{map1}
\begin{aligned}
(\mu_1,\cdots,\mu_n)  \longrightarrow (\lambda_1,\cdots,\lambda_n)=(\mu_1,\cdots,\mu_n) Q^{-1}, %\nonumber
\end{aligned}
\end{equation}
where $Q=(q_{ij})$, $q_{ij}=1-\varrho\delta_{ij}$. %($Q$ is symmetric).
 Here $Q^{-1}$ is well defined, since
 $\mathrm{det}Q=(-1)^{n-1}\varrho^{n-1} (n-\varrho)\neq 0$.
 % since $\kappa_{\Gamma}\leq n-1$ and $\vartheta_{\Gamma}\leq \frac{1}{n}$).
So $\tilde{\Gamma}$ is also an open symmetric convex cone of $\mathbb{R}^n$.
We define %$\tilde{f}: \widetilde{\Gamma}\rightarrow \mathbb{R}$ by $\tilde{f}(\lambda)=f(\mu).$
\begin{equation}\begin{aligned} f(\mu)=\tilde{f}(\lambda). \nonumber \end{aligned}\end{equation}
% (See Remark \ref{remark-10} for more discussion).
The equation \eqref{n-1-equation1} then takes the form 
\begin{equation}
\label{n-1-equation2}
\begin{aligned}
\tilde{f}(\lambda (\mathfrak{g}[u]))=\psi(x,u,\nabla u).
\end{aligned}
\end{equation}
 %That is, equation \eqref{n-1-equation1} is of the form \eqref{hessianequ2-riemann} below.
One can verify that 
if $f$ satisfies %\eqref{elliptic}, 
\eqref{concave}  and \eqref{addistruc} in ${\Gamma}$, then so does 
$\tilde{f}$ in $\tilde{\Gamma}$.

 {\bf Case 1}: $\varrho<0$.  
A straightforward computation shows
$$\frac{\partial \tilde{f}}{\partial\lambda_i}
=\sum_{j=1}^n\frac{\partial f}{\partial\mu_j}\frac{\partial\mu_j}{\partial\lambda_i}
=\sum_{j=1}^n\frac{\partial f}{\partial\mu_j} -\varrho\frac{\partial f}{\partial\mu_i}
\geq \sum_{j=1}^n\frac{\partial f}{\partial\mu_j} =
\frac{1}{n-\varrho} \sum_{j=1}^n\frac{\partial \tilde{f}}{\partial\lambda_j}.$$

{\bf Case 2}: $0<\varrho<\frac{1}{1-\kappa_\Gamma\vartheta_{\Gamma}}$. 
By Theorem \ref{yuan-k+1},  $(1- \kappa_\Gamma\vartheta_{\Gamma} )\sum_{j=1}^n \frac{\partial f}{\partial\mu_j}\geq \frac{\partial f}{\partial\mu_i}$
%i.e.,
%\begin{equation}\label{sumfi-3}\begin{aligned}
%\sum_{j=1}^n \frac{\partial f}{\partial\mu_j} %=\sum_{i\neq j}\frac{\partial f}{\partial\mu_i}+\frac{\partial f}{\partial\mu_j}
% \geq  (1-\varrho(1-\kappa_\Gamma\vartheta_{\Gamma}))\sum_{j=1}^n \frac{\partial f}{\partial\mu_j} +\varrho\frac{\partial f}{\partial\mu_i} \nonumber
%\end{aligned}
%\end{equation}
for all $1\leq i\leq n$. Thus
\begin{equation}
\label{computation-1}
\begin{aligned}
\frac{\partial \tilde{f}}{\partial\lambda_i}
%=\sum_{j=1}^n\frac{\partial f}{\partial\mu_j}\frac{\partial\mu_j}{\partial\lambda_i}
=\sum_{j=1}^n\frac{\partial f}{\partial\mu_j}-\varrho \frac{\partial f}{\partial\mu_i}
\geq\,&
(1-\varrho(1-\kappa_\Gamma\vartheta_{\Gamma}))\sum_{j=1}^n \frac{\partial f}{\partial\mu_j}
\\=\,&
\frac{1-\varrho(1-\kappa_\Gamma\vartheta_{\Gamma})}{n-\varrho} \sum_{j=1}^n\frac{\partial \tilde{f}}{\partial\lambda_j}. \nonumber
\end{aligned}
\end{equation}

\end{proof}

\begin{remark}
It would be worthwhile to note that when $\Gamma\neq\Gamma_n$, besides the \textit{critical} case $\varrho=1$,  \eqref{assumption-4} surprisingly allows \textit{supercritical} case $1<\varrho<\frac{1}{1-\kappa_\Gamma\vartheta_\Gamma}$. 
 (Note that equation \eqref{n-1-equation1}  fails to be elliptic for general $\varrho>1$).
\end{remark}

%\vspace{1mm}
\section{Interior and boundary estimates for certain Hessian type equations}
\label{interiori-estimate}
%The construction of type 2 cones leads to 

%In this section we %focus on the Hessian type equations
 
To achieve the goal, it requires to derive interior and boundary estimates for equations
 \begin{equation}
\label{hessianequ2-riemann}
\begin{aligned}
f(\lambda(\nabla^2 u+A(x,u,\nabla u)))=\psi(x,u,\nabla u)
\end{aligned}
\end{equation}
 of fully uniform ellipticity satisfying
 \begin{equation}
\label{fully-uniform2}
\begin{aligned}
f_{i}(\lambda)\geq \theta\sum_{j=1}^n f_j(\lambda)>0 \mbox{ in } \Gamma,
%\mbox{ for all }  i \mbox{ and for some positive constant } \theta.
\end{aligned}
\end{equation}
for all $i=1,\cdots, n$ and for some positive constant $ \theta$.
%on compact Riemannian manifolds %$(M,g)$ 
%possibly with boundary. 

We focus on solutions in the class of \textit{admissible} functions $ w\in C^2(\bar M)$ satisfying  
 \begin{equation} \label{admissible-function} \begin{aligned} 
 %\lambda(\Delta w g-\varrho\nabla^2 w+W(x,w,\nabla w))
 \lambda(\nabla^2 w+A(x,w,\nabla w))\in \Gamma \mbox{ in } \bar M.
  \end{aligned}\end{equation}
As a result, %\eqref{elliptic} 
\eqref{fully-uniform2} and \eqref{concave} correspond respectively to the fully uniform ellipticity 
and concavity of equation \eqref{hessianequ2-riemann}
% on
 for \textit{admissible} solutions. %This also allows one to use Evans-Krylov theorem if one could derive $C^2$-estimates.
%We assume the right-hand side  satisfies the %nondegenerate and
% \textit{admissible} condition \eqref{nondegenerate1}.
Throughout this section, we denote
 $\mathfrak{g}[u]=\nabla^2 u+A(x,u,\nabla u)$, $\psi[u]=\psi(x,u,\nabla u)$  for $u$.
%  $\mathfrak{g}_{\xi\eta}=\mathfrak{g}(\xi,\eta)$,  $A_{\xi\eta}(x,z,p)=A(x,z,p)(\xi,\eta)$ for $\xi, \eta\in T_x^*\bar M$.
Also,   %$F(\nabla^2 u+A(x,u,\nabla u))=f(\lambda(\nabla^2 u+A(x,u,\nabla u)))$.
\[F(\mathfrak{g}[u])=f(\lambda(\mathfrak{g}[u])).\]
 Furthermore,  for the \textit{admissible} solution $u$ to \eqref{hessianequ2-riemann},
%$\mathfrak{g}=\mathfrak{g}[u]$, $F^{ij}=\frac{\partial F}{\partial a_{ij}}(\mathfrak{g})$.
\begin{equation}
\label{notation-1}
\begin{aligned}
 \mathfrak{g}=\mathfrak{g}[u], \mbox{  } \lambda=\lambda(\mathfrak{g}), \mbox{  }
 F^{ij}=\frac{\partial F}{\partial a_{ij}}(\mathfrak{g}).  
\end{aligned}
\end{equation}
%Let $\lambda=\lambda(\mathfrak{g})$. 
Then the matrix $\{F^{ij}\}$ has eigenvalues $f_1,\cdots, f_n$. Moreover,
\begin{equation}
\begin{aligned}
\,& \sum_{i=1}^n f_i=F^{ij}g_{ij},  
\,& \sum_{i=1}^n f_i\lambda_i =F^{ij}\mathfrak{g}_{ij}. \nonumber
\end{aligned}
\end{equation}

The interior estimates for second derivatives can be stated as follows.
 \begin{theorem}
\label{interior-2nd-2}
Let $u\in C^4(\bar B_{r})$ be an admissible solution to equation \eqref{hessianequ2-riemann} in $B_r\subset M$.
Assume %\eqref{elliptic}, 
\eqref{concave}, \eqref{addistruc}, \eqref{nondegenerate1} and \eqref{fully-uniform2} hold.
Then  
\begin{equation}
\begin{aligned}
\sup_{B_{\frac{r}{2}}}|\nabla^2 u| \leq \frac{C}{r^2}, \nonumber
\end{aligned}
\end{equation}
where $C$ depends on  $|u|_{C^1(B_r)}$, $\theta^{-1}$ and other known data.
\end{theorem}

%\vspace{0.5mm}
 We will prove the
  interior gradient estimate under  asymptotic assumptions:
\begin{equation}
\label{key-assimption1}
\begin{aligned}
\,& |D_pA|\leq \gamma(x,z) |p|, \mbox{  } g(D_p A,p)\leq \gamma(x,z) |p|^2 g, \\
\,&
D_zA+\frac{1}{|p|^2}g(\nabla' A,p)\leq \beta(x,z,|p|) |p|^2g, \\
%\end{aligned}\end{equation}
%\begin{equation} \label{key-assimption2} \begin{aligned}
\,& |D_p \psi|\leq \gamma(x,z) |p|, \mbox{  } -g(D_p \psi,p)\leq \gamma(x,z) |p|^2, \\
\,&
-D_z\psi-\frac{1}{|p|^2}g(\nabla' \psi,p)\leq \beta(x,z,|p|) |p|^2,
\end{aligned}
\end{equation}
where
%$$A(x,z,p):=\frac{\mathrm{tr}_{g}W(x,z,p)}{\varrho(n-\varrho)}g-\frac{W(x,z,p)}{\varrho},$$
%\begin{equation}\label{A-definition1}\begin{aligned}
%A(x,z,p):=\frac{\mathrm{tr}_{g}W(x,z,p)}{\varrho(n-\varrho)}g-\frac{W(x,z,p)}{\varrho}, \nonumber
%\end{aligned}\end{equation}
  $\beta=\beta(x,z,r)$ and $\gamma=\gamma(x,z)$ are positive continuous functions with
 \begin{equation}\begin{aligned}
\,& \lim_{r\rightarrow+\infty} \beta(x,z,r)=0, \,& (x,z,r)\in \bar M\times \mathbb{R}\times\mathbb{R}^+. \nonumber
 \end{aligned}
\end{equation} 
%Throughout this paper,  $\psi(x,z,p)$ and $A(x,z,p)$ (as well as $W(x,z,p)$) denote respectively  smooth function and smooth symmetric $(0,2)$-type tensor of variables $(x,z,p)$,  $(x,p)\in T^*\bar M$, $z\in \mathbb{R}$.
 For simplicity  $\nabla' A$ and $\nabla' \psi$  denote 
the partial covariant derivative of  $A$ and $\psi$, respectively, when viewed as depending on $x\in M$,
and the meanings of $D_zA$, $D_z\psi$, $D_pA$ and $D_p\psi$ are obvious. 

\begin{theorem}
\label{thm-gradient2}
Let $B_r\subset M$ as above and 
$u\in C^3(\bar B_{r})$ be an admissible solution to  \eqref{hessianequ2-riemann} in $B_r$.
In addition to %\eqref{elliptic},
 \eqref{concave} and \eqref{nondegenerate1}, 
we assume  \eqref{fully-uniform2} and \eqref{key-assimption1} hold.
Suppose %for $F^{ij}$ denoted as in \eqref{notation-1} below,
 that there is a positive constant $\kappa_0$ depending not on $\nabla u$ such that
\begin{equation}
\label{sumfi-assumption1}
\begin{aligned}
F^{ij}g_{ij}\geq \kappa_0.
\end{aligned}
\end{equation}
 Then there is a positive constant $C$ depending on $|u|_{C^0(B_r)}$,   $\theta^{-1}$ and other known data such that
$$\sup_{B_{\frac{r}{2}}}|\nabla u|\leq \frac{C}{r}.$$
 
\end{theorem}

\begin{remark}
According to Lemma \ref{lemma3.4} and \eqref{concavity1}, if % \eqref{elliptic}, 
\eqref{concave} and \eqref{addistruc} hold then
\begin{equation}
\label{sumfi}
\begin{aligned}
\sum_{i=1}^n f_i(\lambda)>\frac{f(R\vec{\bf 1})-f(\lambda)}{R} \mbox{ for } R>0.
%\mbox{ where } \vec{\bf 1}=  (1,\cdots, 1)\in \mathbb{R}^n,
\end{aligned}
\end{equation}
 Moreover, if $f$ is homogeneous of degree one then  $\sum_{i=1}^n f_i(\lambda)\geq f(\vec{\bf 1})>0.$
Thus, assumption \eqref{sumfi-assumption1} clearly holds if $f$ is a homogeneous function  of degree one, or  
 $f$ satisfies \eqref{addistruc} and $\psi[u]=\psi(x,u)$. 
 \end{remark}

% \vspace{1mm}
%Together with the fully uniform ellipticity of \eqref{n-1-equation1} proved in Sections \ref{section4}, as a consequence, 
Consequently, we obtain   interior estimates for the equation \eqref{n-1-equation1}.
%\begin{theorem}\label{thm1-c2}
% Let $B_r\subset M$ be a geodesic ball of radius $r$.Assume \eqref{elliptic}, \eqref{concave}, \eqref{addistruc}, \eqref{assumption-4} and \eqref{nondegenerate1} hold.Then any admissible solution $u\in C^4(\bar B_{r})$ to   \eqref{n-1-equation1} in $B_r$  satisfies
%\begin{equation}\begin{aligned}
%\sup_{B_{\frac{r}{2}}}|\nabla^2 u| \leq \frac{C}{r^2}, \nonumber
%\end{aligned}\end{equation} where $C$ depends on $(1-\varrho(1-\kappa_\Gamma\vartheta_{\Gamma}))^{-1}$, $|u|_{C^1(B_r)}$ and other known data.
%\end{theorem}

\begin{theorem}
\label{thm1-c1}
Let $B_r\subset M$ be a geodesic ball of radius $r$.
Let $u\in C^4(\bar B_{r})$ be an admissible solution to \eqref{n-1-equation1} in $B_r$.
%Suppose $A(x,z,p)$ and $\psi(x,z,p)$ satisfy the asymptotic condition \eqref{key-assimption1}.
In addition to %\eqref{elliptic}, 
\eqref{concave}, \eqref{addistruc}, \eqref{assumption-4}, \eqref{nondegenerate1} and \eqref{key-assimption1}, we
assume furthermore that either $\psi[u]=\psi(x,u)$ or $f$ is a homogeneous function of degree one. Then 
$$\sup_{B_{\frac{r}{2}}}( |\nabla^2 u| +|\nabla u|^2)\leq \frac{C}{r^2},$$
where $C$ % is a positive constant depending 
depends on $(1-\varrho(1-\kappa_\Gamma\vartheta_{\Gamma}))^{-1}$, $|u|_{C^0(B_r)}$ and other known data.
\end{theorem}

%Similar interior estimates in real and complex variables are  previously obtained in \cite{Guan2008IMRN,Gursky-Streets-Warren2011,Li2011Sheng,GQY2018}.
 %The proof is analogous to that used in \cite{Guan2008IMRN,GQY2018}.

The boundary estimates for second derivatives are also obtained. 
 \begin{theorem}
 \label{thm2-bdy}
Suppose  %\eqref{elliptic}, 
\eqref{concave}, \eqref{addistruc}, \eqref{nondegenerate1} and \eqref{fully-uniform2} hold.
 Let $u\in C^3(M)\cap C^2(\bar M)$ be an admissible solution to Dirichlet problem of equation
  \eqref{hessianequ2-riemann} with 
 $u=\varphi$ on $\partial M$ for $\varphi\in C^3(\bar M)$. Then $$\sup_{\partial M}|\nabla^2 u|\leq C$$ holds for 
 a uniformly positive constant depending on $\theta^{-1}$, $|u|_{C^1(\bar M)}$, $|\varphi|_{C^3(\bar M)}$ and other known data. 
 
 \end{theorem}

\subsection{Preliminaries}

% The notation and computation are standard.

On a Riemannian manifold $(M,g)$, one defines the curvature tensor by
$$R(X,Y)Z=-\nabla_X\nabla_YZ+\nabla_Y\nabla_X Z+\nabla_{[X,Y]}Z.$$

Let $e_1,...,e_n$ be a local frame on $M$. 
Under Levi-Civita connection $\nabla$ of $(M,g)$, $\nabla_{e_i}e_j=\Gamma_{ij}^k e_k$, and $\Gamma_{ij}^k$ denote the Christoffel symbols and $\Gamma_{ij}^k=\Gamma_{ji}^k$. The curvature coefficients are given by
$R_{ijkl}=g(e_i,R(e_k,e_l)e_j)$. 
For simplicity we shall write 
$$\nabla_i=\nabla_{e_i}, \nabla_{ij}=\nabla_i(\nabla_j)-\Gamma_{ij}^k\nabla_k, 
\nabla_{ijk}=\nabla_i(\nabla_{jk})-\Gamma_{ij}^l\nabla_{lk}-\Gamma^l_{ik}\nabla_{jl}, \mbox{ etc}.$$
We also denote $g_{ij }= g(e_i,e_j), ({g^{ij} })= ({g_{ij}})^{-1}$,
$A_{ij}(x,u,\nabla u)=A(x,u,\nabla u)(e_i,e_j)$,
and $\mathfrak{g}_{ij}=\nabla_{ij}u+A_{ij}(x,u,\nabla u)$.

The linearized operator of equation \eqref{hessianequ2-riemann} is given by
\begin{equation}
\begin{aligned}
\mathcal{L}v=F^{ij}\nabla_{ij}v+(F^{ij}A_{ij,p_l}-\psi_{p_l})\nabla_l v \mbox{ for } v\in C^2(\bar M). \nonumber
\end{aligned}
\end{equation}

\noindent
{\bf Useful computation.} Similar computation can be found in \cite{Guan2015Jiao}.
 We have
\begin{equation}
\begin{aligned}
\nabla_k\mathfrak{g}_{ij}=
%\nabla_{kij}u+\nabla'_kA_{ij}(x,u,\nabla u)+A_{ij,z}(x,u,\nabla u)\nabla_k u+A_{ij,p_l}(x,u,\nabla u)\nabla_{kl}u.
\nabla_{kij}u+\nabla'_kA_{ij} +A_{ij,z} \nabla_k u+A_{ij,p_l}\nabla_{kl}u. \nonumber
\end{aligned}
\end{equation}
By differentiating equation \eqref{hessianequ2-riemann},
\begin{equation}
\label{differ-equa1}
\begin{aligned}
F^{ij}\nabla_k\mathfrak{g}_{ij}=\nabla'_k\psi+\psi_z\nabla_k u+\psi_{p_l}\nabla_{kl}u,
\end{aligned}
\end{equation}
\begin{equation}
\label{differ-buchong3}
\begin{aligned}
F^{ii}\nabla_{11}\mathfrak{g}_{ii}=\,&
\nabla'_{11}\psi
-F^{ij,kl}\nabla_1\mathfrak{g}_{ij} \nabla_1\mathfrak{g}_{kl}
 + 2\nabla'_1\psi_z \nabla_1 u
  +\psi_z\nabla_{11}u
 \\\,&
  +2\nabla'_1\psi_{p_l}\nabla_{1l}u 
+2\psi_{zp_l}\nabla_{1l}u\nabla_1u
+\psi_{zz}|\nabla_1u|^2
\\\,&
+\psi_{p_lp_m}\nabla_{1l}u\nabla_{1m}u
+\psi_{p_l}\nabla_{11l}u.  %\nonumber
\end{aligned}
\end{equation}
%See for example \cite{Guan2015Jiao} for details.
%where $\nabla'_k A_{ij}$ and $\nabla'_k\psi$  denote the partial covariant derivative of  $A$ and $\psi$, respectively, when viewed as depending on $x\in M$, and the meanings of $A_{ij,z}$, $\psi_z$, $A_{ij,p_l}$ and $\psi_{p_l}$ are clearly given. 
Using  \eqref{differ-equa1}, $\nabla_{ij}u=\nabla_{ji}u$ and $\nabla_{ijk}u=\nabla_{kij}u-R_{ljik}\nabla_l u$,
 \begin{equation}
 \label{ineq1-c1}
\begin{aligned}
 \,&F^{ij}\nabla_{ijk}u+(F^{ij}A_{ij,p_l}-\psi_{p_l})\nabla_{lk}u   \\
 =\,&
 \nabla'_k \psi+\psi_z \nabla_k u-F^{ij}R_{ljik}\nabla_l u
 -F^{ij}(\nabla'_kA_{ij}+A_{ij,z}\nabla_k u).
\end{aligned}
\end{equation}

Let $w=|\nabla u|^2$.
 The straightforward computation yields
 \begin{equation}
 \label{wi1}
\begin{aligned}
\,& \nabla_i w=2\nabla_k u \nabla_{ik}u,
\,& \nabla_{ij}w=2\nabla_{ik}u\nabla_{jk}u+2\nabla_k u\nabla_{ijk}u,
\end{aligned}
\end{equation}
\begin{equation}
\label{Lw1}
\begin{aligned}
\mathcal{L}w=\,&
%2F^{ij}\nabla_{ik}u \nabla_{jk}u+2\left\{F^{ij}\nabla_{ijk}u+(F^{ij}A_{ij,p_l}-\psi_{p_l})\nabla_{lk}u \right\}\nabla_k u \\
%=\,&
2F^{ij}\nabla_{ik}u \nabla_{jk}u
-2F^{ij}R_{ljik}\nabla_k u \nabla_l u
+2\nabla'_k\psi \nabla_k u
\\\,&
+2\psi_z|\nabla u|^2
-2F^{ij}(\nabla'_k A_{ij}\nabla_k u+A_{ij,z}|\nabla u|^2).
\end{aligned}
\end{equation}
%one has $|\nabla_i w|\leq 4|\nabla u|^2 |\nabla_{ik}u|^2$, 

\subsection{Interior gradient estimate}
%\begin{proof}
%[Proof of interior gradient estimate]
 %$\eta\in C_0^\infty(M)$, $\eta\geq 0$. 
  % \vspace{1mm}
 Since the equation \eqref{hessianequ2-riemann} is of fully uniform ellipticity, the proof of interior estimates is standard.
 % and is  somewhat analogous to that used in the papers  (see e.g. \cite{Guan2008IMRN,GQY2018,Gursky-Streets-Warren2010,Li2011Sheng}).
    For convenience, we present the details below. 
    The method is analogous to that used in \cite{Guan2008IMRN,GQY2018,Gursky-Streets-Warren2010,Li2011Sheng}
    where the equations are of fully uniform ellipticity. 
   % The method is analogous to that used in the papers cited above (see e.g. \cite{Guan2008IMRN,Guan2003Wang,Gursky-Streets-Warren2010}). 
   
As above we denote $w=|\nabla u|^2$. As in \cite{Guan2008IMRN},
 we consider the quantity $$Q:=\eta w e^\phi$$ where $\phi$ is determined later. 
 Following \cite{Guan2003Wang} %one takes $\eta$ to be
 let $\eta$ be a smooth function 
with compact support in $B_r \subset M$ and 
\begin{equation}
\label{eta1}
\begin{aligned}
0\leq \eta\leq 1, \mbox{  } \eta|_{B_{\frac{r}{2}}}\equiv1, 
 \mbox{  } |\nabla\eta|\leq  \frac{C\sqrt{\eta}}{r},
  \mbox{  } |\nabla^2\eta|\leq \frac{C}{r^2}.
\end{aligned}
\end{equation}
 The quantity $Q$ must attain its maximum at an interior point $x_0\in M$. 
 We may assume $|\nabla u|(x_0)\geq 1$.
 By maximum principle at $x_0$
 \begin{equation}
 \label{mp1}
\begin{aligned}
\,& \frac{\nabla_i\eta}{\eta}+\frac{\nabla_i w}{w}+\nabla_i\phi=0, 
\,& \mathcal{L}(\log\eta+\log w+\phi)\leq 0.
\end{aligned}
\end{equation}
 Around $x_0$ we choose a local orthonormal frame $e_1,\cdots, e_n$;
  for simplicity,  we further assume $e_1,\cdots,e_n$ have been chosen such that at $x_0$,  $\Gamma_{ij}^k=0$,
 %$\nabla u=(\nabla_1u)e_1$.
 $\mathfrak{g}_{ij}$ is diagonal (so is $F^{ij}$).
 
 \subsubsection*{Step 1. Computation and estimation for $\mathcal{L}(\log w)$}
 By \eqref{wi1},
\begin{equation}
\label{wiwi1}
\begin{aligned}
F^{ij}\nabla_iw \nabla_j w \leq 4|\nabla u|^2 F^{ij}\nabla_{ik}u\nabla_{jk}u. %\nonumber
\end{aligned}
\end{equation}
 Using Cauchy-Schwarz inequality, one derives
%and %as in \cite{GQY2018},
 \begin{equation}
 \label{wiwi2}
\begin{aligned}
F^{ij} \frac{\nabla_i w\nabla_j w}{w^2}\leq (1+\epsilon)
\left(\frac{1}{\epsilon}F^{ij}\frac{\nabla_i \eta\nabla_j \eta}{ \eta^2}
+F^{ij}\nabla_i \phi \nabla_j \phi \right), \nonumber
\end{aligned}
\end{equation}
which, together with \eqref{wiwi1}, yields 
  \begin{equation}
 \label{wiwi3}
\begin{aligned}
F^{ij} \frac{\nabla_i w\nabla_j w}{w^2}
\leq (1-\epsilon^2)
\left(\frac{1}{\epsilon}F^{ij}\frac{\nabla_i \eta\nabla_j \eta}{ \eta^2}
+F^{ij}\nabla_i \phi \nabla_j \phi \right)+\frac{4\epsilon}{w}F^{ij}\nabla_{ik}u \nabla_{jk}u. \nonumber
\end{aligned}
\end{equation}
 Set $0<\epsilon\leq \frac{1}{4}$. We now obtain 
 \begin{equation}
 \label{Llogw}
\begin{aligned}
\mathcal{L}(\log w) \geq \,&
\frac{1}{w}F^{ij}\nabla_{ik}u \nabla_{jk}u
+2(\psi_z+\frac{1}{w}\nabla'_k \psi \nabla_k u)
 \\
\,&
-2F^{ij}(A_{ij,z}+\frac{1}{w}\nabla'_k A_{ij}\nabla_k u)  
-\frac{1-\epsilon^2}{\epsilon}F^{ij}\frac{\nabla_i \eta\nabla_j \eta}{ \eta^2}
 \\
\,&
-(1-\epsilon^2)F^{ij}\nabla_i\phi \nabla_j\phi -C_0\sum F^{ii}.
\end{aligned}
\end{equation}
 
 \subsubsection*{Step 2. Construction and computation of $\phi$}
 As in \cite{Guan2008IMRN}  (see also \cite{GQY2018}), let $\phi=v^{-N}$ where $v=u-\inf_{B_r} u+2$ ($v\geq 2$ in $B_r$)  and $N\geq 1$ is an integer that is chosen later. By direct computation,
\begin{equation}
\label{phii}
\begin{aligned}
\,& \nabla_i \phi=-Nv^{-N-1}\nabla_i u, 
%\mbox{  } 
\,& \nabla_{ij}\phi
=N(N+1)v^{-N-2}\nabla_i u\nabla_j u-Nv^{-N-1}\nabla_{ij}u  \nonumber
\end{aligned}
\end{equation}
which % couple with fully uniform ellipticity \eqref{fully-uniform2}, 
implies 
\begin{equation}
\label{Lphi}
\begin{aligned}
\mathcal{L}\phi
= \,& N(N+1)v^{-N-2}F^{ij}\nabla_i u \nabla_j u-Nv^{-N-1}F^{ij}\nabla_{ij}u 
\\\,& -Nv^{-N-1}(F^{ij}A_{ij,p_l}-\psi_{p_l})\nabla_l u. 
%\geq \,& \theta wN(N+1) v^{-N-2} \sum F^{ii}-Nv^{-N-1}F^{ij}\nabla_{ij}u \\\,& -Nv^{-N-1}(F^{ij}A_{ij,p_l}-\psi_{p_l})\nabla_l u.
\end{aligned}
\end{equation}
% Moreover,  $F^{ij}\nabla_i\phi \nabla_j \phi=N^2 v^{-2N-2}F^{ij}\nabla_i u \nabla_j u$.
 
% \subsubsection*{Step 3. Coumputation for $\mathcal{L}(\log\eta)$}
% \begin{equation} \label{Llogeta} \begin{aligned}
%\mathcal{L}(\log\eta)=\frac{F^{ij}\nabla_{ij}\eta}{\eta}-\frac{F^{ij}\nabla_i \eta \nabla_j\eta}{\eta^2}+(F^{i\bar j}A_{ij,p_l}-\psi_{p_l})\frac{\nabla_l \eta}{\eta}.
%\end{aligned}\end{equation}
 
 \subsubsection*{Step 3. Completion of the proof}
 By \eqref{fully-uniform2},
 \begin{equation}
 \label{fully-application1}
\begin{aligned}
\,& F^{ij}\nabla_{ik}u \nabla_{jk}u \geq \theta |\nabla_{ik} u|^2 \sum F^{ii}, 
%\mbox{ } 
\,& F^{ij}\nabla_i u \nabla_j u \geq \theta w\sum F^{ii}.
\end{aligned}
\end{equation}
By Cauchy-Schwarz inequality again, one derives
\begin{equation}
\label{cauchy-1}
\begin{aligned}
\frac{1}{8}N^2 v^{-N-2}\theta w+\frac{\theta}{2w}|\nabla_{ik}u|^2 \geq \frac{1}{2}N\theta |\nabla_{ik}u| v^{-\frac{N}{2}-1}. \nonumber
\end{aligned}
\end{equation}
We choose $N\gg1$ so that 
\begin{equation}
\begin{aligned}
 \,& \frac{1}{4}N\theta v^{-\frac{N}{2}-1}\geq Nv^{-N-1}, 
\,& N(N+1)v^{-N-2}-N^2v^{-2N-2}\geq \frac{N^2}{2}v^{-N-2}.  \nonumber
\end{aligned}
\end{equation}
Suppose furthermore that $\frac{1}{8}Nv^{-N-2}\theta w\geq C_0$ where $C_0$ is as in \eqref{Llogw} (otherwise we are done). 
 Together with \eqref{mp1}, \eqref{Llogw}, \eqref{Lphi} and \eqref{fully-application1},
 we derive
  \begin{equation}
 \label{key1}
\begin{aligned}
0 \geq \,&
%\frac{1}{w}F^{ij}\nabla_{ik}u \nabla_{jk}u +2(\psi_z+\frac{1}{w}\nabla'_k \psi \nabla_k u)-2F^{ij}(A_{ij,z}+\frac{1}{w}\nabla'_k A_{ij}\nabla_k u)   \\ \,& +\{N(N+1)v^{-N-2}-(1-\epsilon^2)N^2v^{-2N-2}\}F^{ij}\nabla_i u\nabla_j u -Nv^{-N-1}F^{ij}\nabla_{ij}u\\ \,&+(F^{ij}A_{ij,p_l}-\psi_{p_l})\frac{\nabla_l \eta}{\eta} + \frac{F^{ij}\nabla_{ij}\eta}{\eta}-\frac{1+\epsilon-\epsilon^2}{\epsilon}F^{ij}\frac{\nabla_i \eta\nabla_j \eta}{ \eta^2} \\ \,& -Nv^{-N-1}(F^{ij}A_{ij,p_l}-\psi_{p_l})\nabla_l u -C_0\sum F^{ii}
% \\ \geq\,&
 \frac{1}{4}\theta N^2 w v^{-N-2}\sum F^{ii}
% +\frac{\theta}{2w}|\nabla_{ik}u|^2\sum F^{ii}
 +\frac{1}{4}N\theta |\nabla_{ik} u|v^{-\frac{N}{2}-1} \sum F^{ii}
 +2(\psi_z+\frac{1}{w}\nabla'_k \psi \nabla_k u)
 \\\,&
-2F^{ij}(A_{ij,z}+\frac{1}{w}\nabla'_k A_{ij}\nabla_k u)  
-Nv^{-N-1}(F^{ij}A_{ij,p_l}-\psi_{p_l})\nabla_l u
\\ \,&
+(F^{ij}A_{ij,p_l}-\psi_{p_l})\frac{\nabla_l \eta}{\eta} + \frac{F^{ij}\nabla_{ij}\eta}{\eta}
-\frac{1+\epsilon-\epsilon^2}{\epsilon}F^{ij}\frac{\nabla_i \eta\nabla_j \eta}{ \eta^2}. \nonumber
\end{aligned}
\end{equation}
Using \eqref{eta1} and the asymptotic assumption \eqref{key-assimption1}, we obtain
\begin{equation}
\begin{aligned}
0\geq  \,&
 \frac{1}{4}\theta N^2 w v^{-N-2}\sum F^{ii}
% +\frac{\theta}{2w}|\nabla_{ik}u|^2\sum F^{ii}
 +\frac{1}{4}N\theta |\nabla_{ik} u|v^{-\frac{N}{2}-1} \sum F^{ii}  -\frac{C}{r^2\eta}\sum F^{ii}
 \\\,&
 -\left(\beta(x,u,|\nabla u|) w+CNv^{-N-1}w+\frac{C\sqrt{w}}{r\sqrt{\eta}}\right) \left(1+\sum F^{ii}\right), \nonumber
\end{aligned}
\end{equation}
which gives $\eta w\leq \frac{C}{r^2}$. Then we complete the proof.
 
% \end{proof}

 \subsection{Interior estimates for second derivatives}

%Let $\zeta$ be a cutoff function as given by \eqref{eta1}.
%\begin{equation}\label{cutfunction-2nd}\begin{aligned}
%\zeta\in C_0^\infty(B_r),\mbox{  }  0\leq \zeta\leq 1, \mbox{  } \zeta|_{B_{\frac{r}{2}}}\equiv 1, \mbox{  } |\nabla\zeta|\leq C_r,
% \mbox{  } |\nabla^2 \zeta|\leq C_r.
%\end{aligned}\end{equation}
Let's consider the function
$$P(x)=\max_{\xi\in T_x\bar M, |\xi|=1} \eta\mathfrak{g}(\xi,\xi)e^{\varphi},
$$
where $\eta$ is the cutoff function as given by \eqref{eta1} and $\varphi$ is a function to be chosen later.
One knows $P$ achieves maximum at an interior point $x_0\in B_r$ and for $\xi\in T_{x_0}\bar M$. 
Around $x_0$ we choose a smooth orthonormal local frame $e_1,\cdots, e_n$ such that $e_1(x_0)=\xi$, $\Gamma_{ij}^k(x_0)=0$ and
$\{\mathfrak{g}_{ij}(x_0)\}$ is diagonal (so is $\{F^{ij}(x_0)\}$). We may assume $\mathfrak{g}_{11}(x_0)\geq 1$.
At $x_0$ one has
\begin{equation}
\label{mp2nd-1}
\begin{aligned}
\,& \nabla_i \varphi+\frac{\nabla_i \mathfrak{g}_{11}}{\mathfrak{g}_{11}}+\frac{\nabla_i \eta}{\eta}, 
\,&
\mathcal{L}(\varphi+\log\eta+\log\mathfrak{g}_{11}) \leq 0.
\end{aligned}
\end{equation}
  
\subsubsection*{Step 1. Estimation for $\mathcal{L}(\mathfrak{g}_{11})$}

%Differentiating the equation, one has

It follows from \eqref{differ-buchong3}  that
\begin{equation}
\label{key2nd-1}
\begin{aligned}
F^{ii}\nabla_{11}\mathfrak{g}_{ii}\geq \psi_{p_l}\nabla_l \mathfrak{g}_{11}-C_0\mathfrak{g}_{11}^2.
\end{aligned}
\end{equation}
The straightforward computation shows
\begin{equation}
\label{key2nd-2}
\begin{aligned}
F^{ii}(\nabla_{ii}\mathfrak{g}_{11}-\nabla_{11}\mathfrak{g}_{ii}) \geq %\,&
%F^{ii}\nabla_l\mathfrak{g}_{ii} A_{11,p_l}-F^{ii}A_{ii,p_l}\nabla_l\mathfrak{g}_{11}-C_1\mathfrak{g}_{11}^2 \sum_{i=1}^n F^{ii} \\
%\geq\,&
 -F^{ii}A_{ii,p_l}\nabla_l\mathfrak{g}_{11}
-C_2\mathfrak{g}_{11}^2 \sum_{i=1}^n F^{ii}-C_2\mathfrak{g}_{11}.
\end{aligned}
\end{equation}
Here we use \eqref{differ-equa1}. %$F^{ii}\nabla_l\mathfrak{g}_{ii} =\nabla'_l\psi+\psi_z \nabla_l u+\psi_{p_k}\nabla_{lk}u$.
Combing \eqref{key2nd-1} and \eqref{key2nd-2}, 
\begin{equation}
\label{key2nd-3}
\begin{aligned}
\mathcal{L}(\mathfrak{g}_{11})=%\,&
F^{ii}\nabla_{ii}\mathfrak{g}_{11}+(F^{ii}A_{ii,p_l}-\psi_{p_l})\nabla_l\mathfrak{g}_{11}  \nonumber
%\\\geq \,&
\geq
-C_3\mathfrak{g}_{11}^2(1+\sum_{i=1}^n F^{ii}).
\end{aligned}
\end{equation}
Using \eqref{mp2nd-1} and Cauchy-Schwarz inequality, 
\begin{equation}
\begin{aligned}
F^{ii}\frac{|\nabla_i\mathfrak{g}_{11}|^2}{\mathfrak{g}_{11}^2}
\leq \frac{3}{2}F^{ii}|\nabla_i\varphi|^2+3F^{ii}\frac{|\nabla_i \eta|^2}{\eta^2}.  \nonumber
\end{aligned}
\end{equation}
Hence,
\begin{equation}
\label{key2nd-4}
\begin{aligned}
\mathcal{L}(\log\mathfrak{g}_{11}) \geq -C_3\mathfrak{g}_{11}(1+\sum_{i=1}^n F^{ii}) -\frac{3}{2}F^{ii}|\nabla_i\varphi|^2
-3F^{ii}\frac{|\nabla_i \eta|^2}{\eta^2}.  \nonumber
\end{aligned}
\end{equation}

\subsubsection*{Step 2. Construction and estimate for $\mathcal{L}\varphi$}

Following \cite{Guan2008IMRN}  we set $$\varphi=\varphi(w)=(1-\frac{w}{2N})^{-\frac{1}{2}} \mbox{ where $w=|\nabla u|^2$ and $N=\sup_{\{\eta>0\}}|\nabla u|^2$}. $$
One can check $\varphi'=\frac{\varphi^3}{4N},$ $\varphi''=\frac{3\varphi^5}{16N^2}$ and $1\leq \varphi\leq \sqrt{2}$.
By \eqref{Lw1}, \eqref{fully-uniform2} and Cauchy-Schwarz inequality, we obtain
\begin{equation}
\begin{aligned}
\mathcal{L}w \geq 
%F^{ii}\mathfrak{g}_{ii}^2-C_4\mathfrak{g}_{11}\sum_{i=1}^n F^{ii}
%\geq (\theta \mathfrak{g}_{11}^2 -C_4\mathfrak{g}_{11})\sum_{i=1}^n F^{ii}.  \nonumber
\frac{1}{2}F^{ii}\mathfrak{g}_{ii}^2-C_4 \sum_{i=1}^n F^{ii}
\geq (\frac{\theta }{2}\mathfrak{g}_{11}^2 -C_4)\sum_{i=1}^n F^{ii}.  \nonumber
\end{aligned}
\end{equation}
Consequently, 
\begin{equation}
\label{key2nd-5}
\begin{aligned}
\mathcal{L}\varphi
%=\,&\varphi'' F^{ii}|\nabla_i w|^2+\varphi'\mathcal{L}w \\
% =\frac{3\varphi^5}{16N^2}F^{ii}|\nabla_i w|^2+\frac{\varphi^3}{4N}\mathcal{L}w \\
\geq %\,&
\frac{3\varphi^5}{16N^2}F^{ii}|\nabla_i w|^2+
\frac{\varphi^3}{4N}  (\frac{\theta }{2}\mathfrak{g}_{11}^2 -C_4)\sum_{i=1}^n F^{ii}.  \nonumber
\end{aligned}
\end{equation}

\subsubsection*{Step 3. Completion of the proof}
From \eqref{eta1},
\begin{equation}
\begin{aligned}
\,&
F^{ii}\frac{|\nabla_i\eta|^2}{\eta^2} \leq \frac{C}{r^2 \eta}\sum_{i=1}^n F^{ii},
\,&
\mathcal{L}(\log\eta)\geq -\frac{C}{r^2 \eta}\sum_{i=1}^n F^{ii} -\frac{C}{r\sqrt{\eta}}(1+\sum_{i=1}^n F^{ii}).  \nonumber
\end{aligned}
\end{equation}
Next,
\begin{equation}
\begin{aligned}
\frac{3}{2}F^{ii}|\nabla_i\varphi|^2=\frac{3\varphi^6}{32N^2}F^{ii}|\nabla_i w|^2
%\leq \frac{3\sqrt{2}\varphi^5}{32N^2}F^{ii}|\nabla_i w|^2
 \leq \frac{3\varphi^5}{16N^2}F^{ii}|\nabla_i w|^2.  \nonumber
\end{aligned}
\end{equation}
Finally,  at $x_0$, %we achieve our goal,
\begin{equation}
\begin{aligned}
0\geq \,& \frac{3}{4N} \left(\frac{\theta}{2} \mathfrak{g}_{11}^2-\frac{4NC_3}{3}\mathfrak{g}_{11}-C_4\right)\sum_{i=1}^n F^{ii}-
C_3\mathfrak{g}_{11} \\\,&
-\frac{C}{r^2 \eta}\sum_{i=1}^n F^{ii} -\frac{C}{r\sqrt{\eta}}(1+\sum_{i=1}^n F^{ii}). \nonumber
%-\frac{C_r}{\eta^2}(1+\sum_{i=1}^n F^{ii}).
\end{aligned}
\end{equation}
Combining with \eqref{sumfi} we have $\eta \mathfrak{g}_{11}\leq \frac{C}{r^2}$ at $x_0$. 
%\begin{equation}\begin{aligned}
%\eta \mathfrak{g}_{11}\leq \frac{C}{r^2} \mbox{ at } x_0.  \nonumber
%\end{aligned}\end{equation}
Therefore the proof is complete.

\subsection{Boundary estimates}

For a boundary point $x_0\in\partial M$ we choose a smooth orthonormal local frame $e_1, \cdots, e_n$ around $x_0$ such that
$e_n$ is unit outer normal vector field when restricted to $\partial M$.  
As we denoted above, %of Theorem \ref{thm1-asymptotic},
 $\mathrm{d}(x)$ is the distance from $x$ to $\partial M$.
Denote by % $ \mathrm{d}(x)$ and
 $\rho(x)$ the distance from $x$ to $x_0$ with respect to $g$,   
\begin{equation}
\label{distance1-def}
\begin{aligned}
%\,& \mathrm{d}(x):=\mathrm{dist}_g(x,\partial M),
%\end{aligned}
%\end{equation}  
%\begin{equation}
%\label{distance2-def}
%\begin{aligned}
 \rho(x):=\mathrm{dist}_g(x,x_0).
\end{aligned}
\end{equation} 
Then $\mathrm{d}$ is smooth near the boundary (whenever $\partial M$ is smooth), i.e., 
$\mathrm{d}$ is smooth and $|\nabla \mathrm{d}|\geq \frac{1}{2}$ in $M_\delta$ for small $\delta>0$, where
\begin{equation}\begin{aligned}
%\,&\Omega_\delta:=\{x\in M: 0<\mathrm{d}<\delta\}, 
\,& M_\delta:=\{x\in M: \rho(x)<\delta\}. \nonumber
\end{aligned}
\end{equation}

 \subsubsection*{Case 1. Pure tangential derivatives}   %Since on boundary $u=\varphi$,
 From the boundary value condition,  on $\partial M$
\begin{equation}
\label{morry-1}
\begin{aligned}
\,& \nabla_\alpha u=\nabla_\alpha \varphi, 
\,& \nabla_{\alpha\beta}u=\nabla_{\alpha\beta}\varphi+\nabla_{n}(u-\varphi)\mathrm{II}(e_\alpha,e_\beta)
% \mbox{ on } \partial M  \nonumber
\end{aligned}
\end{equation}
 for $1\leq \alpha, \beta<n$, where $\mathrm{II}$ is the fundamental form of $\partial M$. 
 This gives the bound of
 second estimates for pure tangential derivatives
\begin{equation}
\label{ineq2-bdy}
\begin{aligned}
|\mathfrak{g}_{\alpha\beta}(x_0)|\leq C.
\end{aligned}
\end{equation}

 \subsubsection*{Case 2. Mixed derivatives}
 For $1\leq\alpha<n$, the local barrier function $\Psi$ on $M_\delta$ ($\delta$ is small) is given by
\begin{equation}
\begin{aligned}
\Psi=\pm \nabla_\alpha (u-\varphi) +A_1 (\frac{N \mathrm{d}^2}{2}
-t \mathrm{d})-A_2\rho^2+A_3 \sum_{k=1}^{n-1} |\nabla_k (u-\varphi)|^2,  \nonumber
\end{aligned}
\end{equation}
where $A_1$, $A_2$, $A_3$, $N$, $t$ are all positive constants to be determined, and $$N\delta-2t\leq0.$$
%The local barriers of this type originate from \cite{Guan1993Boundary}.  
The  local barrier is similar to that used in \cite{Guan12a}, in which the subsolution is a key ingredient.  

 Next, we collect and derive some useful inequalities used here. Clearly, 
 \begin{equation}
\label{Ld1}
\begin{aligned}
\,&\mathcal{L}(\mathrm{d})\leq C_d(1+\sum F^{ii}),
\,&
\mathcal{L}(\rho^2)\leq C_\rho (1+\sum F^{ii}).  \nonumber
\end{aligned}
\end{equation}

To deal with local barrier functions, we need the following formula 
 \begin{equation}
\label{ineq1-bdy}
\begin{aligned}
\nabla_{ij}(\nabla_k u)=\nabla_{ijk}+\Gamma_{ik}^l\nabla_{jl}u+\Gamma_{jk}^l\nabla_{il}u+\nabla_{\nabla_{ij}e_k}u, \nonumber
\end{aligned}
\end{equation}
see e.g. \cite{Guan12a}. Combining with \eqref{ineq1-c1} one derives
\begin{equation}
\label{bdy-1}
\begin{aligned}
|\mathcal{L}(\nabla_k (u-\varphi))|\leq C_1(1+\sum_{i=1}^n f_i+\sum_{i=1}^n f_i |\lambda_i|).  \nonumber
\end{aligned}
\end{equation}
  As in \eqref{Lw1}, we can prove
\begin{equation}
\label{bdy-2}
\begin{aligned}
\mathcal{L}(|\nabla_k (u-\varphi)|^2) \geq F^{ij}\mathfrak{g}_{ik}\mathfrak{g}_{jk} 
-C_1(1+\sum_{i=1}^n f_i+\sum_{i=1}^n f_i |\lambda_i|).  \nonumber
\end{aligned}
\end{equation}

By Proposition 2.19 of \cite{Guan12a}, there exists an index $1\leq r\leq n$ 
\begin{equation}
\begin{aligned}
\sum_{k=1}^{n-1} F^{ij}\mathfrak{g}_{ik}\mathfrak{g}_{jk} \geq \frac{1}{2} \sum_{i\neq r} f_i\lambda_i^2.  \nonumber
\end{aligned}
\end{equation}

Since $f$ satisfies \eqref{fully-uniform2},
% i.e., $f$ is of fully uniform ellipticity $$f_i\geq \theta \sum_{j=1}^n f_j \mbox{ for all } i,$$ 
 as in \cite{Guan14}, 
%one can prove
 there are $c_0$, $C_0$ depending on $\theta$ and other known data so that
\begin{equation}
\label{1c-0C-0}
\begin{aligned}
\sum_{i\neq r} f_i\lambda_i^2 \geq c_0|\lambda|^2\sum_{i=1}^n f_i-C_0\sum_{i=1}^n f_i.   %\nonumber
\end{aligned}
\end{equation}
%By condition \eqref{fully-uniform2}, 
%More precisely,  if $\lambda_r\leq 0$ then $|\lambda_r|=-\lambda_r<\sum_{i\neq r}\lambda_i$ and $\sum_{i\neq r} \lambda_i^2 \geq \frac{|\lambda|^2}{n}.$  For $\lambda\in\partial\Gamma^\sigma$, if $\lambda_r>0$ then
%\begin{equation}\begin{aligned}
%$(\theta\sum_{i=1}^n f_i )^2 \lambda_r^2\leq (f_r\lambda_r)^2\leq (-\sum_{i\neq r} f_i\lambda_i +c_\sigma\sum_{i=1}^n f_i)^2$
%\end{aligned}\end{equation}
%which implies $$\sum_{i\neq r}\lambda_i^2 \geq \frac{\theta^2 |\lambda^2|}{2(n-1)+\theta^2}-\frac{2c_\sigma^2}{2(n-1)+\theta^2},$$here as denoted above $f(c_\sigma\vec{\bf1})=\sigma.$  
%Then we get \eqref{1c-0C-0}. % (See also  \cite{Guan14}).
%On the other hand,
%\begin{equation}\begin{aligned}\mathcal{L}(\frac{N \mathrm{d}^2}{2}-t \mathrm{d})
%=\,&NF^{ij}\nabla_i \mathrm{d}\nabla_j \mathrm{d}+(N \mathrm{d}-t)\mathcal{L} \mathrm{d} \\
%\geq  \frac{N\theta}{4}\sum_{i=1}^n f_i-C_d|N\mathrm{d}-t|(1+\sum F^{ii}), \mbox{  } N>0. \nonumber
%\end{aligned}\end{equation}
By \eqref{sumfi} there is $\kappa>0$ so that $\sum_{i=1}^n f_i = \sum F^{ii} \geq \kappa>0.$
%\begin{equation} \label{sumfi-23} \begin{aligned}
%\sum_{i=1}^n f_i = \sum F^{ii} \geq \kappa>0. \nonumber
%\end{aligned} \end{equation}
Thus,
if $\delta $ and $t$ are chosen small enough so that
 $C_d \max\{|N\delta-t|, t\}\leq  \frac{N \kappa \theta}{8(1+\kappa)},$
 then 
 \begin{equation}
\begin{aligned}
\mathcal{L}(\frac{N \mathrm{d}^2}{2}-t \mathrm{d})
\geq \,& \frac{N \kappa \theta}{4(1+\kappa)}(1+\sum_{i=1}^n f_i)-C_d|N\mathrm{d}-t|(1+\sum F^{ii}) \\
\geq \,&
 \frac{N \kappa \theta}{8(1+\kappa)}(1+\sum_{i=1}^n f_i).  \nonumber
\end{aligned}
\end{equation}

Putting those inequalities above together, if $A_1\gg A_2$, $A_1\gg A_3>1$ then 
\begin{equation}
\begin{aligned}
\mathcal{L}(\Psi)
%\geq \,& \frac{A_1N\theta \kappa}{8(1+\kappa)}(1+\sum_{i=1}^n f_i)+\frac{A_3c_0}{2}|\lambda|^2 \sum_{i=1}^n f_i-C_1(1+A_3)\sum_{i=1}^n f_{i}|\lambda_i| \\ \,&-\frac{A_3C_0}{2}\sum_{i=1}^n f_i-(A_2C_\rho+C_1(1+A_3))(1+\sum_{i=1}^n f_i)\\
\geq\,&
\frac{A_1N\kappa\theta}{8(1+\kappa)}(1+\sum_{i=1}^n f_i )
-(\frac{A_3C_0}{2}+\frac{C_1^2(1+A_3)^2}{2A_3c_0})\sum_{i=1}^n f_i
\\ \,&
-(A_2C_\rho+ C_1(1+A_3))(1+\sum_{i=1}^n f_i)>0 \mbox{ in } M_\delta.  \nonumber
\end{aligned}
\end{equation}
On the other hand $\Psi|_{\partial M}=-A_2\rho^2\leq 0$; while on $\partial M_\delta\setminus \partial  M$, $\rho=\delta$, so $\Psi\leq 0$ if  $N\delta-2t\leq 0$, $A_2\gg1$ and $A_2\gg 1$.

Note that $\Psi(x_0)=0$. By Hopf lemma one derives %the second estimates for mixed derivatives
\begin{equation}
\label{mix-1}
\begin{aligned}
|\mathfrak{g}_{\alpha n}(x_0)|\leq C \mbox{ for } 1\leq \alpha\leq n-1.
\end{aligned}
\end{equation}

 \subsubsection*{Case 3. Double normal derivatives}
 Fix $x_0\in\partial M$. As above $c_{\psi[u](x_0)}$ denotes the positive constant such that
 $$f(c_{\psi[u](x_0)}\vec{\bf 1})=\psi[u](x_0).$$
 Since $\mathrm{tr}_g (\mathfrak{g})>0$,  \eqref{ineq2-bdy} and \eqref{mix-1}, we have $\mathfrak{g}_{nn}\geq -C$. 
 In what follows we assume $\mathfrak{g}_{nn}(x_0)\geq c_{\psi[u](x_0)}$ (otherwise we are done).
 Since  \eqref{fully-uniform2}, $F^{ij}\geq \theta \sum_{k=1}^n f_k \delta_{ij}$.  (Here $g_{ij}(x_0)=\delta_{ij}$). 
 %Thus  \begin{equation} \label{dou-normal1} F^{ii}\geq \theta \sum_{i=1}^n f_i. \end{equation}
 By the concavity of equation,   \eqref{ineq2-bdy} and \eqref{mix-1}
\begin{equation}
\begin{aligned}
0= F(\mathfrak{g})-F(c_{\psi[u](x_0)} g) 
\geq \,&  
F^{ij}(\mathfrak{g})(\mathfrak{g}_{ij}-c_{\psi[u](x_0)}\delta_{ij}) 
%\\ \geq \,&-\sum_{i<n \mbox{ or } j<n}|F^{ij}(\mathfrak{g})(\mathfrak{g}_{ij}-c_{\psi[u](x_0)}\delta_{ij})|+ F^{nn}(\mathfrak{g}) (\mathfrak{g}_{nn}-c_{\psi[u](x_0)}) 
\\
\geq \,& 
-C'\sum_{i=1}^n f_i (\lambda)+ \theta(\mathfrak{g}_{nn}-c_{\psi[u](x_0)})  \sum_{i=1}^n f_i(\lambda).  \nonumber
\end{aligned}
\end{equation}
This gives $\mathfrak{g}_{nn}(x_0) \leq C.$
%\begin{equation} \begin{aligned}
%\mathfrak{g}_{nn}(x_0) \leq C. \nonumber
%\end{aligned} \end{equation}

%We complete the proof of boundary estimate.

%\begin{remark}
%In proof of interior and boundary estimates we do not use subsolutions.
%\end{remark}

\begin{remark}
\label{remark-10}
The cone $\widetilde{\Gamma}$, corresponding to $\Gamma$ constructed as in \eqref{map1}, is an open symmetric convex cone in 
$\mathbb{R}^n$ with vertex at origin. It should be notable that $\widetilde{\Gamma}$ contains the positive 
cone $\Gamma_n$ if $\varrho\leq 1$ and $\varrho\neq 0$, %(thus, it is of type 2 according to Corollary \ref{coro-type2}), 
while it may not contain $\Gamma_n$ for some $1<\varrho<\frac{1}{1-\kappa_\Gamma \vartheta_{\Gamma}}.$ 
Fortunately, from the proof of interior and boundary estimates above, we can check that this does not affect the proof. So the
 estimates obtained there still hold for
 \eqref{n-1-equation1} even if
$1<\varrho<\frac{1}{1-\kappa_\Gamma \vartheta_{\Gamma}}$. % $\varrho\neq 0$.
\end{remark}

%\section{Statement of interior estimates}
%\label{statement-estimates}

 %The proof will be given by combining with the results obtained in Sections \ref{section4} and \ref{interiori-estimate} below.

% \section{Proof of main results}
\section{Existence and asymptotic behavior of complete conformal metrics}
\label{proofofmainresult}

%In addition to \eqref{elliptic} and \eqref{concave}, 
We assume throughout this section that 
%$(M,g)$ is a compact Riemannian manifold of dimension $n\geq 3$ with smooth boundary, 
$\nu$ is the outer unit normal vector field along $\partial M$,
 $\psi\in C^\infty(\bar M), \psi>0 \mbox{ in } \bar M$.
 %and $f$ is homogeneous of degree one.
 % $$f|_{\partial\Gamma}\equiv0, \mbox{  } f|_{\Gamma}>0, \mbox{  } f(\vec{\bf 1})=1.$$ 
 %Using \eqref{concavity1},  $\sum_{i=1}^n f_i(\lambda)\lambda_i=f(\lambda)$, $f(\vec{\bf 1})=1$, we derive $f_i(\vec{\bf 1})=\frac{1}{n}$ for all $i$, and 
%\begin{equation}\label{key1-main}\begin{aligned}
%\,&   \sum_{i=1}^n \lambda_i \geq nf(\lambda),\,& \forall \lambda\in\Gamma.
%\end{aligned}\end{equation}  
%We observe that \eqref{key1-main}  is a key ingredient in the study of asymptotic behavior.
We denote %$U[u]$ by 
$$U[u]:=\Delta u g-\nabla^2 u +\frac{n-3}{2}|\nabla u|^2 g+du\otimes du.$$
%\begin{equation}\begin{aligned}
%U[u]:=\Delta u g-\nabla^2 u +\frac{n-3}{2}|\nabla u|^2 g+du\otimes du.  \nonumber
%\end{aligned}\end{equation}
From \eqref{conformal-Einstein-tensor}, 
$\mathfrak{g}[u]:=\frac{1}{n-2}G_{\tilde{g}}=\frac{1}{n-2}G_g+U[u] \mbox{ for }\tilde{g}=e^{2u}g.$
%\begin{equation}\begin{aligned}
%\mathfrak{g}[u]=\frac{1}{n-2}G_{\tilde{g}}=\frac{1}{n-2}G_g+U[u] \mbox{ for }\tilde{g}=e^{2u}g.
%\end{aligned}\end{equation}
Without loss of generality, we assume %the background metric
 $g$ is \textit{admissible}, i.e.,
 $\lambda(G_g)\in \Gamma$ in $\bar M$. 
%Let $\underline{\psi}=f(\lambda(G_g))$. 

\subsection{Proof of Theorem \ref{thm0-conformal}}
%\subsection{Proof of Theorems \ref{thm0-conformal} and \ref{thm2-conformal}, and asymptotic properties near boundary}
% and asymptotic properties of complete conformal metrics near boundary
In the  proof we apply standard method of approximation, in which the construction of local lower barriers follows an idea of \cite{Guan2008IMRN}. 
 First, we are going to study Dirichlet problem for equation \eqref{conformal-equation-degree1}:
 \begin{equation}
\label{dirichlet-equ1}
\begin{aligned}
\,& f(\lambda(\mathfrak{g}[{u}])=\frac{\psi}{n-2}e^{2{u}} \mbox{ in } M, 
\,& u=\varphi \mbox{ on } \partial M,
\end{aligned}
\end{equation}
 as above, %$\mathfrak{g}[u]=\frac{1}{n-2}G_{\tilde{g}}=\frac{1}{n-2}G_g+U[u] \mbox{ for }\tilde{g}=e^{2u}g.$
\begin{equation}\label{g-u}\begin{aligned}
\mathfrak{g}[u]=\frac{1}{n-2}G_{\tilde{g}}=\frac{1}{n-2}G_g+U[u] \mbox{ for }\tilde{g}=e^{2u}g.
\end{aligned}\end{equation}
% (possibly with infinity boundary value to assure the completeness of metric).
 %Firstly, we prove 
%We prove the following proposition from which  Theorem \ref{thm2-conformal} follows. %as a consequence.
\begin{proposition}
\label{lemma4-main}
Given $\varphi\in C^\infty(\partial M)$.
The Dirichlet problem \eqref{dirichlet-equ1} admits a unique smooth admissible solution, 
provided that there is an admissible function $\underline{u}\in C^2(\bar M)$.
\end{proposition}
  The conformal metrics with prescribed boundary metric are also obtained.
 
  \begin{theorem}
 \label{thm2-conformal}
 %Let $(M,g)$ be a $n$-dimensional compact Riemannian manifold with smooth boundary, 
 Let %$n\geq 3$,  
 $\psi\in C^\infty(\bar M)$, $\psi>0$ in $\bar M$. 
 Suppose %\eqref{elliptic},
  \eqref{concave}, \eqref{homogeneous-1},    \eqref{homogeneous-1-buchong2},
\eqref{admissible-metric1} and $\Gamma\neq\Gamma_n$ hold.
 %$(M,g)$ supposes a $C^2$-\textit{admissible conformal} metric satisfying \eqref{admissible-metric1}. 
 For a Riemannian metric $h$ on $\partial M$ which is conformal to $g|_{\partial M}$, there is a smooth conformal 
 metric $\tilde{g}$ with prescribed boundary condition $\tilde{g}|_{\partial M}=h$ to satisfy  \eqref{conformal-equ0}. Moreover, $\lambda_{\tilde{g}}(G_{\tilde{g}})\in \Gamma \mbox{ in } \bar M$.
 
 \end{theorem}

%Theorem \ref{thm2-conformal} is deduced from Proposition \ref{lemma4-main}.
%Comparing with Theorem \ref{thm1-dirichlet},  one only needs  in Proposition \ref{lemma4-main} admissible functions but without restriction to boundary value.
% The proof of this lemma is similar to Theorem 5.3 and Corollary 5.4 of \cite{Guan2008IMRN}.

% \begin{remark}
%If the cutoff function used in the proof of interior estimates are chosen to be %$\zeta\equiv1$, $\eta\equiv1$, then
 We let $\eta\equiv1$ in proof of interior estimates,  there is a positive constant $C$ depending on $|u|_{C^0(\bar M)}$ %$\sup_{\partial M}|\nabla u|$ 
such that
\begin{equation}
\label{globalc2c1}
\begin{aligned}
\sup_M |\nabla u|^2+\sup_M |\nabla^2 u| \leq C(1+\sup_{\partial M}|\nabla u|^2+\sup_{\partial M} |\nabla^2 u| ).
\end{aligned}
\end{equation}
%\end{remark}
  
  The boundary estimates for second derivatives have been already obtained in previous section.
 It remains to  prove
 \begin{lemma}
 \label{c0-bdyc1}
  For any $C^2$-admissible solution $u$ to Dirichlet problem \eqref{dirichlet-equ1}, we have
  \begin{equation}
\begin{aligned}
 \sup_M |u|\leq C_0', \nonumber
 \end{aligned}
\end{equation}
 where $C_0'$ is a uniformly positive constant.
 \end{lemma}
 
 \begin{proof}
 Without loss of generality, we assume $g$ is admissible, i.e. $\lambda_g(G_g)\in\Gamma$.
 If $u(x_0)=\min_{\bar M}u$ for some $x_0\in M$ then $\frac{\psi(x_0)}{n-2}e^{2u(x_0)}\geq f(\lambda(\frac{1}{n-2}G_g))$.
 We get $u(x_0)\geq -C'$, and $\inf_M u\geq -\min\{C',\min_{\partial M}\varphi\}$. Similarly, $\sup_M u\leq C''.$
 \end{proof}

 \begin{lemma}
 \label{lemma3-main1}
 Let $u$ be a $C^2$-admissible solution to Dirichlet problem \eqref{dirichlet-equ1},
  there is a uniformly positive constant $C$ such that $\nabla_\nu u|_{\partial M}\leq C.$
% \begin{equation}\begin{aligned}
%\nabla_\nu u|_{\partial M}\leq C. \nonumber
%\end{aligned}\end{equation}
\end{lemma}

\begin{lemma}
 \label{bdyc1}
  For any $C^2$-admissible solution $u$ to Dirichlet problem \eqref{dirichlet-equ1}, we have
  \begin{equation}
\begin{aligned}
 -C_1'\leq  \nabla_{\nu}u|_{\partial M}. \nonumber
 \end{aligned}
\end{equation}
 where $C_1'$ is a uniformly positive constant.
 \end{lemma}
 
The proof of Lemmas \ref{lemma3-main1} and \ref{bdyc1} is standard by combining the maximum principle with existence of local lower and upper barriers near boundary.
% (see e.g. \cite{GT1983}). The construction here is inspired by  \cite{Guan2008IMRN}.

 Let $\mathrm{d}(x)$ be as defined above %in  Theorem \ref{thm1-asymptotic} %\eqref{distance1-def} 
 the distance function to boundary. Then we know $\mathrm{d}(x)$ is smooth near the boundary (whenever $\partial M$ is smooth), i.e., for small $\delta>0$, $\mathrm{d}(x)$ is smooth in $$\Omega_\delta:=\{x\in M: 0<\mathrm{d}(x)<\delta\}.$$

Let $\underline{h}$ be of the form $\underline{h}=\underline{h}(\mathrm{d})$. By direct computation one obtains

\begin{lemma}
\label{lemma1-main}

$U[\underline{h}]= (\underline{h}'' +\frac{(n-3)\underline{h}'^2}{2})|\nabla \mathrm{d}|^2 g+(\underline{h}'^2-\underline{h}'')d\mathrm{d}\otimes d\mathrm{d}+\underline{h}'(\Delta \mathrm{d}g-\nabla^2 \mathrm{d}).$

\end{lemma}

%We hope to construct a (local) barrier $\underline{h}$ on $\bar\Omega_\delta$ satisfying, for $0<\delta\ll1$,
%\begin{enumerate}
%\item $\underline{h}'^2\geq \underline{h}'',$ $\underline{h}''+\frac{n-3}{2}\underline{h}'^2>c_0\underline{h}'^2$ in $\Omega_\delta$ for some $c_0>0$.
%\item $\underline{h}'^2   e^{-2\underline{h}}$ can be arbitrary large in $\Omega_\delta$ if $\delta$ is sufficiently small.
%\item  In $\Omega_\delta$, $\underline{h}'^2   e^{-2\underline{h}}\gg1$ and $\underline{h}'^2\gg1$.
%\item $\underline{h}(\delta)\leq \inf_M(u-\varphi)$. %for $0<\delta\ll1$.
%\end{enumerate}

 %It is interesting to find $\underline{h}$ such that $\underline{h}''$ and $\underline{h}'^2$ are comparable.
  Following \cite{Guan2008IMRN}, we set
\begin{equation}
\begin{aligned}
h_k(\mathrm{d})=  \log \frac{k\delta^2}{k\mathrm{d}+\delta^2}.   \nonumber
\end{aligned}
\end{equation}
% The $h_k$ satisfies the above requirements, as %
One can check that
  for any $k\geq1$,
\begin{equation}
\label{check-1}
\begin{aligned}
h'_k(\mathrm{d}) =\,& -\frac{ k}{k\mathrm{d}+\delta^2}, \mbox{  }
 h''_k(\mathrm{d})=\frac{  k^2}{(k\mathrm{d}+\delta^2)^2}, 
 \\ h_k'^2(\mathrm{d})= \,& h_k''(\mathrm{d}), \mbox{  } h_k'^2 e^{-2h_k}= \delta^{-4}, \mbox{  }
h_k(\delta)\leq \log\delta.  %\nonumber
\end{aligned}
\end{equation}

\begin{lemma}
\label{lemma-k5}
Let %$v=\mathrm{d}$ in Lemma \ref{lemma1-main},
 $\varphi\in C^2(\bar M)$ and $k>0$ fixed.  Then  for  $0<\delta\ll1$
\begin{equation}
\label{keyeq2-main}
\begin{aligned}
\,& U[h_k]\geq   \frac{(n-1) k^2}{8(k\mathrm{d}+\delta^2)^2}g,  %\mbox{ in } \Omega_\delta,
%\end{aligned}\end{equation}
%\begin{equation} \label{keyeq3-main}\begin{aligned}
\,& U[h_k+\varphi]\geq  \frac{ (n-1) k^2}{16(k\mathrm{d}+\delta^2)^2}g,  \mbox{ in } \Omega_\delta,
\end{aligned}
\end{equation}
%in $\Omega_\delta$.
\end{lemma}
\begin{proof}

Since $-h_k'=\frac{ k}{k\mathrm{d}+\delta^2} \geq 
\frac{k}{k\delta+\delta^2}$ in $\Omega_\delta$, so $$-h_k'\gg1\mbox{ whenever } 0<\delta\ll1;$$
 moreover $|\nabla \mathrm{d}|\geq \frac{1}{\sqrt{2}}$ on such  $\Omega_\delta$.
Notice $\Delta\mathrm{d}g-\nabla^2\mathrm{d}$ is bounded in $\bar M$,
we  can get the first inequaliry of \eqref{keyeq2-main}.
%Next, we prove \eqref{keyeq3-main}.
Next, using
\begin{equation}
\begin{aligned}
U[h_k+\varphi]=U[h_k]+U[\varphi]+h_k'(d\mathrm{d}\otimes d\varphi+d\varphi\otimes d\mathrm{d}+(n-3)g(\nabla\mathrm{d},\nabla\varphi) g), \nonumber
\end{aligned}
\end{equation}
 we have the second one of \eqref{keyeq2-main} by the same argument.

%The \eqref{keyeq4-main} can be found in \cite{Guan2008IMRN}.
\end{proof}

%By Lemma \ref{lemma3.4}, if $\lambda(G_g)\in \Gamma$, then in $\Omega_\delta$
For $0<\delta\ll1$ satisfying $\frac{G_g}{n-2}\leq \frac{(n-1) k^2}{32(k\mathrm{d}+\delta^2)^2}g$, we have by \eqref{g-u}
%\begin{equation} \begin{aligned}
%f(\lambda(\mathfrak{g}[h_k]))\geq F(  \frac{(n-1) k^2}{16(k\mathrm{d}+\delta^2)}g)= \frac{(n-1)k^2}{16(k\mathrm{d}+\delta^2)}\nonumber \end{aligned} \end{equation}
\begin{equation}
\begin{aligned}
f(\lambda(\mathfrak{g}[h_k+\varphi]))\geq F(\frac{(n-1) k^2}{32(k\mathrm{d}+\delta^2)^2}g)=  \frac{(n-1) k^2}{32(k\mathrm{d}+\delta^2)^2}.  \nonumber
\end{aligned}
\end{equation}
Observe as in \eqref{check-1} that 
$ \frac{k^2}{(k\mathrm{d}+\delta^2)^2}e^{-2h_k}=  \delta^{-4}\rightarrow+\infty \mbox{ as } \delta\rightarrow 0^+.$
So
%\begin{equation}\begin{aligned}
%f(\lambda(\mathfrak{g}[h_k))\geq \frac{\psi}{n-2} e^{2h_k} \mbox{ in } \Omega_\delta,  \nonumber
%\end{aligned}\end{equation}
\begin{equation}
\label{key4321}
\begin{aligned}
f(\lambda(\mathfrak{g}[h_k+\varphi]))\geq \frac{\psi}{n-2} e^{2(h_k+\varphi)} \mbox{ in } \Omega_\delta, \mbox{ for } 0<\delta\ll1.
\end{aligned}
\end{equation}

%Next we complete the proof of Lemma \ref{lemma3-main1}.
\begin{proof}
[Proof of Lemma \ref{lemma3-main1}]
Let  $k=1$ $({h}_1(0)=0)$ and $$\underline{w}=h_1+\varphi =\log\frac{\delta^2}{\mathrm{d}+\delta^2}+\varphi.$$
%Clearly,  $\underline{w}=\varphi$ on $\partial M$.
Near the boundary for small $\delta$
\begin{equation}
\label{keyeq4-main}
\begin{aligned}
\underline{w}=\log\frac{\delta}{1+\delta}+ \varphi \leq u \mbox{ on } \{\mathrm{d}=\delta\}.
\end{aligned}
\end{equation}
Such $\delta$ exists according to Lemma \ref{c0-bdyc1}. 
 Together with \eqref{key4321}, \eqref{keyeq4-main}, and  $\underline{w}=\varphi$ on $\partial M$, 
 we apply maximum principle to the equation on $\Omega_\delta$ to obtain
\begin{equation}
\label{keyeq8-main}
\begin{aligned}
 u\geq \underline{w}   \mbox{ in } \Omega_\delta. \nonumber
 \end{aligned}
\end{equation}
Couple with  $u=\underline{w}=\varphi$ on $\partial M$, we get $\nabla_\nu (u- \underline{w})\leq0$ on $\partial M$. 
%Since $u=\underline{w}=\varphi$ on $\partial M$.
%which completes the proof. %by Hopf lemma.
\end{proof}

 \begin{proof}
 [Proof of Lemma \ref{bdyc1}]
 It suffices to construct (local) upper barrier near boundary. 
 %Let ${w}^+$ be the solution to 
%  \begin{equation}  \label{supersolution-2} \begin{aligned}
% \,& \mathrm{tr}_g(\mathfrak{g}[{{w}^+}])=0 \mbox{ in } M,  \,& {w}^+=\varphi \mbox{ on } \partial M.
%\end{aligned}\end{equation}
%The solvability of \eqref{supersolution-2} can be found in \cite{GT1983}. Maximum principle yields $u\leq {w}^+$, and thus $ {w}^+$ is a upper barrier,  $\nabla_\nu(u-{w}^+)|_{\partial M}\geq 0$.
 %In the proof of the Lemma 5.2 in \cite{Guan2008IMRN}, Guan constructed the desired (global) upper barriers. 
 Here we give a  local upper barrier analogous to that used
  in the proof of Lemma \ref{lemma3-main1}. Near boundary, we let 
% $$\overline{w}=-\frac{1}{n-2}h_1(\mathrm{d})+\varphi=\varphi-\frac{1}{n-2}\log\frac{\delta^2}{\mathrm{d}+\delta^2}.$$
 $$\bar{w}=\bar{h}+\varphi$$
% for $\bar{h}=\bar{h}(\mathrm{d})$ to be picked later. 
% \begin{equation}\begin{aligned}
%\frac{1}{n-1}\mathrm{tr}_g(U[v])=\Delta v+\frac{n-2}{2}|\nabla v|^2 \mbox{ for } v\in C^2, \nonumber
%\end{aligned}\end{equation}
 %which gives $u\leq\bar w$ near boundary. 
%From Lemma \ref{c0-bdyc1} again,
%It suffices to construct $\bar h$ satisfying 
%\begin{enumerate}
%\item   $\bar h'^2$ and $-\bar h''$ shall be comparable.
%\item $\bar h''+\frac{n-2}{2}\bar h'^2<-c_1\bar h'^2$ in $\Omega_\delta$ for some $c_1>0$.
%\item $\bar h'^2\rightarrow+\infty$ as $x\rightarrow\partial M$.
%\item  $\bar h'^2\rightarrow+\infty$ in $\Omega_\delta$ as $\delta\rightarrow 0^+$.
%\item   $\bar{h}(0)=0$, $\bar h(\delta) \geq \sup_{M}(u-\varphi)$ for $0<\delta\ll1$.
%\end{enumerate}
%Such conditions are all fulfilled if 
where $\bar h(d) =\frac{1}{n-2}\log(1+\frac{\mathrm{d}}{\delta^2})$ or $\bar h(d) =  \frac{1}{n-2}\log(1+\frac{k \mathrm{d}}{\delta}) $ with $ k\geq e^{(n-2)\sup_M(u-\varphi)}-1$.
The construction of barriers of this type is standard which can  be found in textbooks (see e.g. \cite{GT1983}). 
%The (local) barrier $\bar h$ would be obtained if one could find $\bar h$ satisfying 
Straightforward computation gives
  \begin{equation}
\begin{aligned}
\frac{1}{n-1}\mathrm{tr}_g(U[\bar{h}])=\bar{h}' \Delta\mathrm{d} +(\bar h''+\frac{n-2}{2}\bar h'^2)|\nabla \mathrm{d}|^2, \nonumber
\end{aligned}
\end{equation}
 \begin{equation}
\begin{aligned}
U[w]=U[\bar{h}]+U[\varphi]+\bar{h}' [d\mathrm{d}\otimes d\varphi+d\varphi\otimes d\mathrm{d}+(n-3)g(\nabla\mathrm{d},\nabla\varphi) g]. \nonumber
\end{aligned}
\end{equation}
Thus
\begin{equation}
\begin{aligned}
\mathrm{tr}_g(\mathfrak{g}[\bar h+\varphi])\leq0 
\mbox{ in } \Omega_\delta, \mbox{ }
\bar h|_{\partial M}\geq0, \mbox{ }
(\bar h+\varphi-u)|_{\{\mathrm{d}=\delta\}}\geq0. \nonumber
\end{aligned}
\end{equation}
The proof is complete.
 \end{proof}

\begin{proof}[Proof of Proposition \ref{lemma4-main}]
%The uniqueness %of admissible solution
% follows from maximum principle.
Let $\underline{\psi}=f(\lambda(G_g))$.  For $t\in [0,1]$
 we consider 
\begin{equation}
\label{Dirichlet-t}
\begin{aligned}
\,& f(\lambda(\frac{1}{n-2}G_{e^{2u^t}g}) = \frac{1}{{n-2}}\left({t\psi}+(1-t)\underline{\psi} \right)e^{2u^t} \mbox{ in } M, \,&
 {u}^t=t\varphi \mbox{ on } \partial M.
\end{aligned}
\end{equation}
Set ${\bf S}=\{t\in [0,1]: \mbox{For $t$ the Dirichlet problem \eqref{Dirichlet-t} is classically solvable}\}.$

Clearly $0\in {\bf S}$ ($u^0\equiv0$). The openess is classical. By Proposition \ref{keykey}  the linearized operator $\mathcal{L}_t: C^{k+2,\alpha}(\bar M)\rightarrow C^{k,\alpha}(\bar M)$  is invertible ($t\in {\bf S}$), thus implying the openess of ${\bf S}$. 
Combining with Lemmas \ref{c0-bdyc1},  \ref{lemma3-main1}, \ref{bdyc1}, Theorem \ref{thm2-bdy}, and \eqref{morry-1}, \eqref{globalc2c1},
 we have
$$|u|_{C^2(\bar M)}\leq C, \mbox{ independent of $u$ and its derivatives}.$$
Using Evans-Krylov theorem \cite{Evans82,Krylov83} and Schauder theory (see \cite{GT1983}), ${\bf S}$ is closed.  So ${\bf S}=[0,1]$. 
%We complete the proof.
\end{proof}
 
%Theorem \ref{thm2-conformal} is deduced from Proposition \ref{lemma4-main}.

%Next  we complete the proof of Theorem \ref{thm0-conformal}.
%Comparing with Theorem \ref{thm1-dirichlet},  in Proposition \ref{lemma4-main} 
%one only needs admissible functions but without restriction to boundary value.
% The proof of this lemma is similar to Theorem 5.3 and Corollary 5.4 of \cite{Guan2008IMRN}.
 % \vspace{0.5mm}
The following theorem %due to Aviles-McOwen, 
% Aviles-McOwen's theorem mentioned above
plays an important role in the proof of Theorem \ref{thm0-conformal}.
\begin{theorem}[\cite{Aviles1988McOwen}]
\label{thm-AM}
Let $(M,g)$ be a compact Riemannian manifold of dimension $n\geq 3$ with smooth boundary. 
For any negative constant $-c$ $(c>0)$,
there exists on $M$ a smooth complete  metric $\tilde{g}=e^{2\tilde{u}}g$ of scalar curvature $-c$.
\end{theorem}

This theorem indicates that  there is a smooth function $\tilde{u}$ satisfying
\begin{equation}
\label{scalar-equ1}
\begin{aligned}
\,&\mathrm{tr}_g (\mathfrak{g}[\tilde{u}])=\mathrm{tr}_g(\frac{G_g}{n-2}+U[\tilde{u}])=\frac{c}{2}e^{2\tilde{u}} \mbox{ in } M, \,&
\tilde{u}=+\infty \mbox{ on } \partial M.
\end{aligned}
\end{equation}
It can be viewed as a supersolution of approximate Dirichlet problems, thereby providing, on each compact subset, a uniform upper bound for approximate solutions. 

 \begin{proof}
 [Proof of Theorem \ref{thm0-conformal}]
 The proof is standard once one established interior estimates.
 Let's consider approximate Dirichlet problems
 \begin{equation}
\label{dirichlet-equk}
\begin{aligned}
\,& f(\lambda(\mathfrak{g}[{u_k}])=\frac{\psi}{n-2}e^{2{u_k}} \mbox{ in } M, 
\,& u_k=\log k \mbox{ on } \partial M.
\end{aligned}
\end{equation}
The Dirichlet problems are solvability in class of smooth admissible
 functions according to Proposition \ref{lemma4-main}.
 
  Let $c=\frac{2n}{n-2}\inf_M \psi$ and  $\tilde{u}$ be the corresponding function obeying \eqref{scalar-equ1}. 
By \eqref{key1-main}, $\frac{n\psi}{n-2} e^{2u_k}\leq \mathrm{tr}_g(\mathfrak{g}[u_k])$.
The maximum principle yields
 \begin{equation}
 \label{uniform-c0}
\begin{aligned}
u_k\leq u_{k+1}\leq \tilde{u} \mbox{ in } M \mbox{ for } k\geq 1. % \nonumber
\end{aligned}
\end{equation}
%which shows that the function
 %Let $u_\infty(x):=\lim_{k\rightarrow+\infty}u_k(x)$ is well defined for all $x\in M$. 
 As a result, for any compact subset $K\subset\subset M$,
  $$|u_k|_{C^0(K)}\leq C_1(K) \mbox{ for all } k, \mbox{ for $C_1(K)$  being independent of $k$}.$$

Given any compact subset $K\subset\subset M$, we choose a compact subset $K_1\subset\subset M$ such that
$K\subset\subset K_1$.
By interior estimates established above %(Theorems \ref{thm1-c2} and \ref{thm1-c1}) 
and 
Evans-Krylov theorem, 
 $$|u_k|_{C^{2,\alpha}(K)}\leq C_2(K,K_1) \mbox{ for all } k, \mbox{ where $C_2(K,K_1)$ depends not on $k$}.$$
 %Combining with diagonal argument
Combining with  Schauder theory, let $k\rightarrow +\infty$, %Indeed, 
\begin{equation}
\label{u_infty}
\begin{aligned}
u_\infty(x)=\lim_{k\rightarrow+\infty}u_k(x) \mbox{ for all } x\in M,
\end{aligned}
\end{equation}
  $u_\infty$ is a smooth admissible solution to \eqref{conformal-equation-degree1} and
  \eqref{infty-boundary}.
 %\eqref{conformal-equ0}.
%\begin{equation}\label{dirichlet-equ-infty}\begin{aligned}
%\,& f(\lambda(\mathfrak{g}[{u_\infty}])=\frac{\psi}{n-2}e^{2{u_\infty}} \mbox{ in } M, \,& u_\infty=+\infty \mbox{ on } \partial M.
%\end{aligned}\end{equation}
Again, maximum principle yields $u_k\geq \log\frac{k\delta^2}{k\mathrm{d}+\delta^2}$ near boundary.
Thus $u_\infty +\log {\mathrm{d}}\geq -C_0$  for some $C_0$
%\begin{equation}\label{key1-asymptotic}
%u_\infty +\log {\mathrm{d}}\geq -C_0 \mbox{ for some } C_0
%\end{equation}
near the boundary. The metric $g_\infty=e^{2u_\infty}g$ is complete.

\end{proof}

\subsection{Asymptotic properties and uniqueness of complete conformal metrics}
\label{oo}

%\begin{remark}
It would be interesting to investigate the asymptotic behavior  when $x\rightarrow\partial M$.

 For conformal scalar curvature equation \eqref{scalar-equ1} (with prescribed scalar curvature $-c$) on smooth bounded domains in Euclidean spaces,
  Loewner-Nirenberg \cite{Loewner1974Nirenberg} 
proved that %near the boundary %$|(\mathrm{d}e^{\tilde{u}})^{\frac{n-2}{2}}-1|\leq Cd$.  
$\mathrm{d}(x)^2e^{2\tilde{u}(x)}\rightarrow \frac{n(n-1)}{c}$ as $x\rightarrow\partial M$;
%$\frac{2}{n-2}\log(1-C\mathrm{d})\leq \tilde{u}+\log \mathrm{d}\leq \frac{2}{n-2}\log(1+C\mathrm{d}).$
%$| \tilde{u}+\log \mathrm{d}|\leq C$
while on curved manifolds,
   in the proof of their Lemma 5.2 in \cite{Gursky-Streets-Warren2011},
Gursky-Streets-Warren constructed local supersolution of %the conformal scalar curvature equation 
\eqref{scalar-equ1}. From their construction one can see  %near the boundary 
$\tilde{u}$ obeys 
\begin{equation}
\label{asymptotic-rate1-002}
\begin{aligned}
\lim_{x\rightarrow\partial M}\mathrm{d}(x)^2e^{2\tilde{u}(x)}\leq \frac{n(n-1)}{c}.
\end{aligned}
\end{equation}
 Please refer to \cite{Loewner1974Nirenberg,Gursky-Streets-Warren2011} for more details. 
 %Similar to the proof of Proposition \ref{proposition-asymptotic} below, one can derive
% $\lim_{x\rightarrow\partial M}\mathrm{d}(x)^2e^{2\tilde{u}(x)}\geq \frac{n(n-1)}{c}.$
%\begin{equation}
%\label{asymptotic-rate1-002}
%\begin{aligned}
%\lim_{x\rightarrow\partial M}\mathrm{d}(x)^2e^{2\tilde{u}(x)}\geq \frac{n(n-1)}{c}.
%\end{aligned}
%\end{equation}
%As a consequence,  following the method of \cite{Loewner1974Nirenberg,Gursky-Streets-Warren2011} one obtains the %uniqueness of complete conformal metric  of constant scalar curvature $-c$.

For the complete conformal metrics satisfying \eqref{conformal-equ0},
 %obtained in  Theorem  \ref{thm0-conformal},
  we further prove the following asymptotic properties.
\begin{proposition}
\label{super123}
Let $\widetilde{g}_\infty=e^{2\widetilde{u}_\infty}g$
be a complete %conformal 
metric %claimed in  
satisfying \eqref{conformal-equ0}, then
%obtained in the statement of Theorem  \ref{thm0-conformal}. Then  
 \begin{equation} 
 \label{asymptotic-rate1}
 \begin{aligned}
\lim_{x\rightarrow\partial M} (\widetilde{u}_\infty(x)+\log \mathrm{d}(x))\leq \frac{1}{2}\log\frac{(n-1)(n-2)}{2 \inf_{\partial M}\psi}. 
 \end{aligned}
\end{equation}
\end{proposition}

\begin{proof}
 The proof uses Theorem \ref{thm-AM} and \eqref{asymptotic-rate1-002}.
Let $\Omega_{\delta,\delta'}=\{x\in M: 0<\delta'<\mathrm{d}(x)<\delta-\delta'\}.$
 If  $0<\delta'<\frac{\delta}{2}\ll1$, $\Omega_\delta$ and $\Omega_{\delta,\delta'}$ are both smooth.
 By Theorem \ref{thm-AM},  $\Omega_{\delta}$ admits a smooth complete conformal metric $g_\infty^{\delta}=e^{2w_\infty^{\delta}}g$ with constant scalar curvature $-1$.
Again, by Theorem \ref{thm-AM},  there is on $\Omega_{\delta,\delta'}$ a smooth complete conformal metric $g_\infty^{\delta,\delta'}=e^{2u_\infty^{\delta,\delta'}}g$ with constant scalar curvature $-1$, that is,
 $\mathrm{tr}_g (\mathfrak{g}[u_\infty^{\delta,\delta'}])=\frac{1}{2}e^{2u_\infty^{\delta,\delta'}},$  ${u_\infty^{\delta,\delta'}}|_{\partial \Omega_{\delta,\delta'}}=+\infty.$
% \begin{equation}  \begin{aligned}
%\,& \mathrm{tr}_g (\mathfrak{g}[u_\infty^{\delta,\delta'}])=\frac{1}{2}e^{2u_\infty^{\delta,\delta'}}, \,& {u_\infty^{\delta,\delta'}}|_{\partial \Omega_{\delta,\delta'}}=+\infty. \nonumber
% \end{aligned}\end{equation}
 For $\widetilde{u}_\infty$ solving \eqref{conformal-equ0}, $\mathrm{tr}_g(\mathfrak{g}[\widetilde{u}_\infty]) \geq {n\psi} 
 e^{2\widetilde{u}_\infty}$ 
according to  \eqref{key1-main}.
The maximum principle implies 
\begin{equation}  \begin{aligned}
\,& u_\infty^{\delta,\delta'} \geq \widetilde{u}_\infty+\frac{1}{2}\log \left(\frac{2n}{n-2}\inf_{\Omega_{\delta}}\psi\right),
\,& u_\infty^{\delta,\delta'}\geq w_\infty^{\delta}
\mbox{ in } \Omega_{\delta,\delta'}. \nonumber
 \end{aligned}\end{equation}
%For any compact subset $K\subset\subset\Omega_\delta$, we have $\delta_0>0$ such that $K\subset\subset\Omega^{\delta,\delta_0}$, and then
  For any $0<\delta'<\delta_0'<\frac{\delta}{2}$, according to the maximum principle, $u_\infty^{\delta,\delta'}\leq u_\infty^{\delta,\delta_0'}$ in $\Omega_{\delta,\delta_0'}$.
%Let $\mathrm{d}^{\delta,\delta'}(x)$ be the distance function from $x\in\Omega_{\delta,\delta'}$ to $\partial\Omega_{\delta,\delta'}$ with respect to $g$. Notice for $0<\delta'<\delta\ll1$, %$\mathrm{d}^{\delta,\delta'}(x)+\delta'=\mathrm{d}(x).$
% $\lim_{\delta'\rightarrow0}\mathrm{d}^{\delta,\delta'}(x)=\mathrm{d}(x).$
Let $u_\infty^{\delta}(x):=\lim_{\delta'\rightarrow0} u_\infty^{\delta,\delta'}(x)$, then %\eqref{f-equ1}. 
%By Theorem \ref{thm-AM},  $\Omega_{\delta}$ admits a smooth complete conformal metric $g_\infty^{\delta}=e^{2u_\infty^{\delta}}g$ with constant scalar curvature $-1$.
% First, we prove % in $\Omega_\delta$, 
  \begin{equation} 
 \label{f-equ1}
 \begin{aligned}
\,& u_\infty^{\delta} \geq w_\infty^{\delta}, \,&  u_\infty^{\delta} \geq  \widetilde{u}_\infty+ \frac{1}{2}\log \left(\frac{2n}{n-2}\inf_{\Omega_{\delta}}\psi \right)  \mbox{ in } \Omega_{\delta}.
  \end{aligned}
\end{equation}
%Here we use the uniqueness of complete conformal metric of scalar curvature of $-1$.
%also the proof of local $C^0$-estimate also uses the method in second proof of Theorem \ref{thm-c0-upper}).
Thus $e^{2u_\infty^\delta}g$ is a smooth complete conformal metric on $\Omega_{\delta}$ of scalar curvature $-1$. (The proof of interior estimates for conformally scalar curvature equation is standard and is the same as before, since
 it corresponds to $\sigma_1$ that is of fully uniform ellipticity).
  Furthermore, \eqref{asymptotic-rate1-002} gives
 $\lim_{x\rightarrow \partial M} (u_\infty^{\delta}(x)+\log\mathrm{d}(x))\leq \frac{1}{2}\log [n(n-1)].$
%$\lim_{x\rightarrow \partial\Omega^{\delta,\delta'}} (u_\infty^{\delta,\delta'}(x)+\log\mathrm{d}^{\delta,\delta'}(x))\leq \frac{1}{2}\log [n(n-1)].$
%For the solution $u_k$ of \eqref{dirichlet-equk}, $\mathrm{tr}_g(\mathfrak{g}[u_k])\geq \frac{n\psi}{n-2}e^{2u_k}$ 
Putting them together,
%$\lim_{x\rightarrow\partial M} (u_\infty(x)+\log \mathrm{d}(x))\leq \frac{1}{2}\log\frac{(n-1)(n-2)}{2( \inf_{\partial M}\psi-\epsilon)}.$
we get \eqref{asymptotic-rate1}. 
\end{proof}

The proof above indeed gives the following proposition:
 \begin{proposition}
\label{proposition-asymptotic-scalar}
%Let $u_\infty$ be the solution obtained in Theorem  \ref{thm0-conformal}.
Suppose the smooth conformal metric 
$\widetilde{g}=e^{2w}g$ has strictly negative scalar curvature $R_{\widetilde{g}}$ in $M$.
Assume $R_{\widetilde{g}}$ can be continuously extended to boundary, still denoted $R_{\widetilde{g}}$.
Then %$\lim_{x\rightarrow\partial M} \mathrm{d}(x)^2 e^{2w(x)}\leq \frac{n(n-1)}{ \inf_{\partial M}(-R_{\widetilde{g}}) }.$
$\lim_{x\rightarrow\partial M} (w(x)+\log \mathrm{d}(x))\leq \frac{1}{2}\log\frac{n(n-1)}{ \inf_{\partial M}(-R_{\widetilde{g}}) }.$

\end{proposition}

%By constructing local subsolutions, we prove
On the other hand, we can prove
 \begin{proposition}
\label{proposition-asymptotic}
 
Any admissible complete metric $\widetilde{g}_\infty=e^{2\widetilde{u}_\infty}g$  satisfying \eqref{conformal-equ0}  
obeys
\begin{equation}
\label{asymptotic-rate2}
\begin{aligned}
\lim_{x\rightarrow\partial M} (\widetilde{u}_\infty(x)+\log \mathrm{d}(x))\geq \frac{1}{2}\log\frac{(n-1)(n-2)}{2 \sup_{\partial M}\psi}. %\nonumber
\end{aligned}
\end{equation}
In addition if %$\inf_{\partial M}\psi=\sup_{\partial M}\psi=\Lambda$, 
$\psi|_{\partial M}\equiv\Lambda$, i.e. $\psi$ is constant when restricted to boundary, 
then
\begin{equation}
\label{asymptotic-rate3}
\begin{aligned}
\lim_{x\rightarrow\partial M} (\widetilde{u}_\infty(x)+\log \mathrm{d}(x))=\frac{1}{2}\log\frac{(n-1)(n-2)}{2 \Lambda}
\end{aligned}
\end{equation} 
and the admissible  complete conformal metric satisfying \eqref{conformal-equ0} is unique.
\end{proposition}

\begin{proof}
%We use \eqref{key1-main} and $|\nabla \mathrm{d}|_{\partial  M}=1$.
The proof is different from that used by \cite{Loewner1974Nirenberg,Gursky-Streets-Warren2011}.
 The key ingredients in this proof are \eqref{key1-main} and $|\nabla \mathrm{d}|_{\partial  M}=1$. 
Fix a constant $0<\epsilon<1$. %Fix $\epsilon$.
 We consider 
\begin{equation}
\label{001}
\left\{
\begin{aligned}
\,& f(\lambda(\mathfrak{g}[u_{k,\epsilon,\delta}]))=\frac{\psi}{n-2} e^{2u_{k,\epsilon,\delta}} \mbox{ in } M, \\
\,& u_{k,\epsilon,\delta}=\log {k} +\frac{1}{2}\log\frac{(1-\epsilon)^2(n-1)(n-2)}{2(\sup_{\partial M}\psi+\epsilon)}  \mbox{ on } \partial M,
\end{aligned}
\right.
\end{equation}
where $\delta$ is a positive constant to be determined later.
The solvability is obtained by Proposition \ref{lemma4-main}. 
The maximum principle implies
\begin{equation}
\label{b4b}
\begin{aligned}
\widetilde{u}_\infty\geq u_{k,\epsilon,\delta} \mbox{ in } M,
\end{aligned}
\end{equation}
\begin{equation}
\label{1114-3}
\begin{aligned}
\inf_M u_{k,\epsilon,\delta} \geq \min\left\{\frac{1}{2}\inf_M \log\frac{f(\lambda(G_g))}{\psi},  \frac{1}{2}\log\frac{k^2(1-\epsilon)^2(n-1)(n-2)}{2(\sup_{\partial M}\psi+\epsilon)} \right\}.
\end{aligned}
\end{equation}
Obviously, if $k\gg1$  then 
$\inf_M u_{k,\epsilon,\delta} \geq \frac{1}{2}\inf_M \log\frac{f(\lambda(G_g))}{\psi}.$
%\begin{equation}\begin{aligned}
%  \min\left\{\frac{1}{2}\inf_M \log\frac{f(\lambda(G_g))}{\psi}, 
% \frac{1}{2}\log\frac{k^2(1-\epsilon)^2(n-1)(n-2)}{2(\sup_{\partial M}\psi+\epsilon)} \right\} =  \frac{1}{2}\inf_M \log\frac{f(\lambda(G_g))}{\psi}. \nonumber
%\end{aligned}\end{equation}

The local subsolution is given by
 \begin{equation}
\label{1115-0}
\begin{aligned}
h_{k,\epsilon,\delta}(\mathrm{d})=\log\frac{k}{k\mathrm{d}+1}
+\frac{1}{2}\log\frac{(1-\epsilon)^2(n-1)(n-2)}{2(\sup_{\partial M}\psi+\epsilon)}+\frac{1}{\mathrm{d}+\delta}-\frac{1}{\delta}. \nonumber
\end{aligned}
\end{equation}
The $``\frac{1}{\mathrm{d}+\delta}-\frac{1}{\delta}"$ part %in \eqref{1115-0}
 is inspired by \cite{Gursky-Streets-Warren2011}.
Lemma \ref{lemma3.4} %couple with  Lemma 6.2 in \cite{CNS3}, 
yields $$F^{ij}(\mathfrak{g}[h_{k,\epsilon,\delta}])(G_g)_{ij}>0.$$ Together with the concavity of $f$, it further yields
\begin{equation}
\begin{aligned}
f(\lambda(\mathfrak{g}[h_{k,\epsilon,\delta}]))%=F(\frac{G_g}{n-2}+U[h_{k,\epsilon,\delta}])
=F(\mathfrak{g}[h_{k,\epsilon,\delta}])\geq F(U[h_{k,\epsilon,\delta}])
=f(\lambda(U[h_{k,\epsilon,\delta}])).  \nonumber
\end{aligned}
\end{equation}

Straightforward computation tells
\begin{equation}
\label{1114-1}
\begin{aligned}
\,& h_{k,\epsilon,\delta}'=-\frac{k}{k\mathrm{d}+1}-\frac{1}{(\mathrm{d}+\delta)^2},
 \,&h_{k,\epsilon,\delta}''=\frac{k^2}{(k\mathrm{d}+1)^2}+\frac{2}{(\mathrm{d}+\delta)^3}, \nonumber
 \end{aligned}
\end{equation}
\begin{equation}
\begin{aligned}
 h_{k,\epsilon,\delta}'^2
 =\frac{k^2}{(k\mathrm{d}+1)^2}+\frac{1}{(\mathrm{d}+\delta)^4}+\frac{2k}{(k\mathrm{d}+1)(\mathrm{d}+\delta)^2}, \nonumber
\end{aligned}
\end{equation}
$$ h_{k,\epsilon,\delta}''+\frac{n-3}{2} h_{k,\epsilon,\delta}'^2
\geq \frac{n-1}{2}\frac{k^2}{(k\mathrm{d}+1)^2}+\frac{2}{(d+\delta)^3}.$$
It is easy to see $ h_{k,\epsilon,\delta}'^2- h_{k,\epsilon,\delta}''\geq0$ provided $\mathrm{d}+\delta\leq\frac{1}{2}$ or $k\geq\frac{1}{\delta}$.
 Together with Lemma \ref{lemma1-main}, one has
\begin{equation}
\begin{aligned}
U[h_{k,\epsilon,\delta}]
%= \,& \left(h_{k,\epsilon,\delta}''+\frac{n-3}{2}h_{k,\epsilon,\delta}'^2\right)|\nabla \mathrm{d}|^2 g+h'(\Delta \mathrm{d} g -\nabla^2\mathrm{d}) \\\,& +(h_{k,\epsilon,\delta}'^2-h_{k,\epsilon,\delta}'')d\mathrm{d}\otimes d\mathrm{d} \\  
\geq\,&
\left(\frac{n-1}{2}\frac{k^2}{(k\mathrm{d}+1)^2}+\frac{2}{(d+\delta)^3} \right)|\nabla \mathrm{d}|^2 g
\\\,&
- \left(\frac{k}{k\mathrm{d}+1}+\frac{1}{(\mathrm{d}+\delta)^2}\right)(\Delta \mathrm{d} g -\nabla^2\mathrm{d}).  \nonumber
\end{aligned}
\end{equation}

Note that $|\nabla \mathrm{d}|^2=1$ on ${\partial M}$, 
$\Delta\mathrm{d}g-\nabla^2\mathrm{d}$ is bounded in $\bar M$,
 $\frac{k}{k\mathrm{d}+1}=\frac{1}{\mathrm{d}+\frac{1}{k}}\geq \frac{1}{\mathrm{d}+\delta}$ for $k\geq\frac{1}{\delta}$, 
and $\frac{1}{\mathrm{d}+\delta}\gg1$ in $\Omega_\delta$ 
if $0<\delta\ll1$.
For the $\epsilon$ fixed, there is a $\delta_\epsilon$ depending on $\epsilon$ and other known data %but not on $k$
 such that for 
$0<\delta<\delta_{\epsilon}$,
\begin{enumerate}
\item $\mathrm{d}$ is smooth, and $|\nabla \mathrm{d}|^2\geq 1-\epsilon$.
\item $\frac{2}{(d+\delta)^3} |\nabla \mathrm{d}|^2 g
-  \frac{1}{(\mathrm{d}+\delta)^2} (\Delta \mathrm{d} g -\nabla^2\mathrm{d})\geq0$.
\item $\frac{\epsilon(n-1)}{2}\frac{k^2}{(k\mathrm{d}+1)^2} |\nabla \mathrm{d}|^2 g-\frac{k}{k\mathrm{d}+1} (\Delta \mathrm{d} g -\nabla^2\mathrm{d})\geq0$.
%\item $U[h_{k,\epsilon,\delta}]\geq \frac{(n-1)(1-\epsilon)}{2} \frac{k^2}{(k\mathrm{d}+1)^2} |\nabla \mathrm{d}|^2 g$.
\item $\sup_{\partial M}\psi+\epsilon\geq \sup_{\Omega_{\delta}}{\psi}$,
%\end{enumerate}
%And on $\partial \Omega_\delta$
%\begin{enumerate}
\end{enumerate}
 in $\Omega_\delta$.
From now on we fix $0<\delta<\delta_\epsilon$ (for example $\delta=\frac{\delta_\epsilon}{2}$) and $k\geq\frac{1}{\delta}$.
Putting the discussion above together we have on $\Omega_\delta$
\begin{equation}
\label{buchong1}
\begin{aligned}
 f(\lambda(\mathfrak{g}[h_{k,\epsilon,\delta}]))\geq\,& F(U[h_{k,\epsilon,\delta}])\\
 \geq\,&
 F\left( \frac{(n-1)(1-\epsilon)}{2}\cdot\frac{k^2}{(k\mathrm{d}+1)^2}|\nabla \mathrm{d}|^2 g\right) 
% \\\geq\,& \frac{(n-1)(1-\epsilon)}{2}
%  \cdot \frac{2(\sup_{\partial M}\psi+\epsilon)}{(1-\epsilon)^2(n-1)(n-2)}e^{2h_{k,\epsilon,\delta}(\mathrm{d})}  \cdot (1-\epsilon) 
 \\\geq\,&  \frac{\sup_{\partial M}\psi+\epsilon}{n-2}e^{2h_{k,\epsilon,\delta}(\mathrm{d})}   \\
  \geq\,& \frac{\sup_{\Omega_{\delta}}\psi}{n-2} e^{2h_{k,\epsilon,\delta}(\mathrm{d})} \\
  \geq\,&  \frac{\psi}{n-2} e^{2h_{k,\epsilon,\delta}}.  % \nonumber
\end{aligned}
\end{equation}
Here we also use $\frac{k^2}{(k\mathrm{d}+1)^2} \geq\frac{2(\sup_{\partial M}\psi+\epsilon)}{(1-\epsilon)^2(n-1)(n-2)}e^{2h_{k,\epsilon,\delta}}.$
%\begin{equation}\label{1115-1}\begin{aligned}
%\frac{k^2}{(k\mathrm{d}+1)^2} \geq\frac{2(\sup_{\partial M}\psi+\epsilon)}{(1-\epsilon)^2(n-1)(n-2)}e^{2h_{k,\epsilon,\delta}}.\nonumber
%\end{aligned}\end{equation}
Note that on $\partial \Omega_\delta$, $u_{k,\epsilon,\delta}=h_{k,\epsilon,\delta}$ on $\partial M$, and
$u_{k,\epsilon,\delta}|_{\mathrm{d}=\delta} \geq h_{k,\epsilon,\delta}(\delta)$ (according to \eqref{1114-3}).
Here we also use $\frac{\log\frac{k}{k\delta+1}}{\frac{1}{2\delta}}\leq\frac{\log\frac{1}{\delta}}{\frac{1}{2\delta}}\rightarrow 0$ as $\delta\rightarrow0^+$.
  Then the maximum principle yields on $\Omega_{\delta}$ 
\begin{equation}
\begin{aligned}
u_{k,\epsilon,\delta}\geq h_{k,\epsilon,\delta}=\log\frac{k}{k\mathrm{d}+\delta}+\frac{1}{2}\log\frac{(1-\epsilon)^2(n-1)(n-2)}{2((\sup_{\partial M}\psi+\epsilon)}+\frac{1}{\mathrm{d}+\delta}-\frac{1}{\delta}.  \nonumber
\end{aligned}
\end{equation}
From \eqref{b4b}, let $k\rightarrow +\infty$, %$u_\infty\geq u_{k,\epsilon,\delta}$ for any $k>\frac{1}{\delta}$.
we know that on $\Omega_\delta$,
\begin{equation}
\begin{aligned}
  \widetilde{u}_\infty(x)+\log \mathrm{d}(x)\geq \frac{1}{2}\log\frac{(n-1)(n-2)(1-\epsilon)^2}{2(\sup_{\partial M}\psi+\epsilon)}
  +\frac{1}{\mathrm{d}+\delta}-\frac{1}{\delta}.  \nonumber
\end{aligned}
\end{equation}
Thus \eqref{asymptotic-rate2} follows.
%\begin{equation}\begin{aligned}
%\lim_{x\rightarrow\partial M}  (u_\infty(x)+\log \mathrm{d}(x))\geq \frac{1}{2}\log\frac{(n-1)(n-2)}{2\sup_{\partial M}\psi}  \nonumber
%\end{aligned}\end{equation}
The uniqueness of complete conformal metric in case of $\inf_{\partial M}\psi=\sup_{\partial M}\psi$ is a consequence of Proposition \ref{unique-prop} below.

%Finally, if $\inf_{\partial M}\psi=\sup_{\partial M}\psi$, \eqref{asymptotic-rate3} follows from \eqref{asymptotic-rate1} and  \eqref{asymptotic-rate2}, which yields the uniqueness of complete conformal metric. More precisely, let $e^{2u_\infty}g$ and $e^{2w_\infty}g$ be two complete conformal metrics satisfying \eqref{conformal-equ0}.  We first prove $u_\infty\geq w_\infty$. Assume $w_\infty>u_\infty$ at some point.  Since the asymptotic property \eqref{asymptotic-rate3} implies $\lim_{x\rightarrow\partial M}(u_\infty(x)-w_\infty(x))=0$, we know there is $x_0\in M$ at which $(w_\infty-u_\infty)(x_0)=\sup_M(w_\infty-u_\infty)>0$, while the maximum principle yields $w_\infty(x_0)\leq u_\infty(x_0)$.  This is a contradiction. So $u_\infty\geq w_\infty$ in $M$. Conversely, $w_\infty\geq u_\infty$ in $M$. Hence, $u_\infty\equiv w_\infty$.

\end{proof}

\begin{proposition}
\label{unique-prop}
Let $u_\infty$ be as in \eqref{u_infty} and $w_\infty$ be any admissible solution to  \eqref{conformal-equ0}, then 
$$u_\infty\leq w_\infty \leq u_\infty +\frac{1}{2}(\sup_{\partial M}\log\psi- \inf_{\partial M}\log\psi) \mbox{ in } M.$$

\end{proposition}

\begin{proof}
For $u_k$ solving \eqref{dirichlet-equk}, the maximum principle yields $w_\infty\geq u_k$ $(\forall k\geq1)$. Then $w_\infty\geq u_\infty$.

Next, 
we prove the second inequality.
%$w_\infty-u_\infty\leq\frac{1}{2}\sup_{\partial M}\log\psi-\frac{1}{2}\inf_{\partial M}\log\psi.$
   Note that
the asymptotic properties 
\eqref{asymptotic-rate1} and  \eqref{asymptotic-rate2} imply $\lim_{x\rightarrow\partial M}(u_\infty(x)-w_\infty(x))=\frac{1}{2}(\sup_{\partial M}\log\psi-\inf_{\partial M}\log\psi)$. If $w_\infty>u_\infty+\frac{1}{2}(\sup_{\partial M}\log\psi-\inf_{\partial M}\log\psi)$ at some point, then there is $x_0\in M$ at which $(w_\infty-u_\infty)(x_0)=\sup_M(w_\infty-u_\infty)>\frac{1}{2}(\sup_{\partial M}\log\psi-\inf_{\partial M}\log\psi)$, which is a contradiction since the maximum principle yields $w_\infty(x_0)\leq u_\infty(x_0)$.  
\end{proof}

%\begin{proof}
%The proof is similar to proof of Proposition \ref{proposition-asymptotic}. On the annulus $$A_{\delta}:=\left\{x\in M: \frac{\delta}{2}<\mathrm{d}(x)<\delta\right\}$$ we construct the supersolution
 % \begin{equation} \begin{aligned}
%h_{k,\epsilon,\delta}(\mathrm{d})=\log\frac{k}{k\mathrm{d}+1}+\frac{1}{2}\log\frac{(1+\epsilon)^2(n-1)(n-2)}{2(\inf_{\partial M}(-R_{\widetilde{g}})-\epsilon)}.
%\end{aligned}\end{equation}
%The proof is much more simple. We omit the details here.
%\end{proof}

   \subsection{Remark on the existence of \textit{admissible conformal} metrics}

 Let $[g]=\left\{e^{2v}g: v\in C^\infty(\bar M)\right\}$, and $\mathcal{C}(u)=\left\{x\in \bar M: du(x)=0\right\}$ denote the critical set of $u$.
% be the conformal class of $g$. 
We define %a conformal invariant
\begin{equation}
\label{A(M)-def}
\begin{aligned}
\mathcal{A}(\bar M,[g],\Gamma)=\left\{u\in C^2(\bar M):   \lambda_{\tilde g}(\Delta_{\tilde g} u \tilde{g}-\tilde{\nabla}^2 u) \in \Gamma  \mbox{ in  } %\tilde{\nabla} u =0
\mathcal{C}(u),  \mbox{ for some } \tilde{g}\in [g] \right\}.
\end{aligned}
\end{equation}
We can check that $\mathcal{A}(\bar M,[g],\Gamma)$ is independent of the choice of $\tilde{g}$ in %the conformal class 
 $[g]$.
So $\mathcal{A}(\bar M,[g],\Gamma)$ is \textit{conformally invariant}.
%\vspace{1mm}
%In this subsection we
%In the following proposition, we show that  \eqref{admissible-metric1} 
%in Theorems \ref{thm0-conformal} and \ref{thm2-conformal} 
%automatically holds if $\mathcal{A}(\bar M,[g],\Gamma)\neq\emptyset.$

\begin{proposition}
\label{lemma5-main}
Let $\mathcal{A}(\bar M,[g],\Gamma)$ be as above.
%For $n\geq 3$.
If $(M,g)$ supposes a $C^2$-function in $\mathcal{A}(\bar M,[g],\Gamma)$, then there exists an \textit{admissible conformal} metric %satisfying \eqref{admissible-metric1} 
on $\bar M$.
\end{proposition}

\begin{proof}
For $v\in \mathcal{A}(\bar M,[g],\Gamma)$ (assume $v\geq0$), %otherwise replaced by $v-\inf_M v$), 
we set $ \underline{u}=e^{Av},$ $\tilde{g}=e^{2\underline{u}}g,$
% \begin{equation} \begin{aligned} \,&\underline{u}=e^{Av}, \,& \tilde{g}=e^{2\underline{u}}g, \nonumber\end{aligned}\end{equation}
then \begin{equation}
\begin{aligned}
U[\underline{u}]=\,&A^2e^{Av}(1+\frac{n-3}{2}e^{Av})|\nabla v|^2 g
+Ae^{Av}(\Delta v g-\nabla^2v) \\\,&
+A^2e^{Av}(e^{Av}-1)dv\otimes dv \\ \nonumber
\geq \,& Ae^{Av} \left\{ A|\nabla v|^2 g+(\Delta v g-\nabla^2v) \right\}. 
\end{aligned}
\end{equation}

% As above we denote the critical set of $v$ by %$$\mathcal{C}(v):=\{x\in \bar M: \nabla v(x)=0\}.$$
% \begin{equation}\begin{aligned}
% \mathcal{C}(v):=\left\{x\in \bar M: \nabla v(x)=0\right\}. \nonumber
% \end{aligned}\end{equation}
%Obviously,
 %$$\mathcal{C}(v)=\left\{x\in \bar M: dv(x)=0\right\}.$$
 For $x\in \mathcal{C}(v)$, there is $r_x>0$ %($r_x$ is small)  
 such that 
$$\lambda(\Delta v g-\nabla^2 v)\in \Gamma \mbox{ in } \mathrm{V}_x:=\{y\in \bar M: \mathrm{dist}(x,y)<r_x\}.$$
Fix $r_x$ as above. Set $U=\cup_{x\in \mathcal{C}(v)} \{y\in \bar M: \mathrm{dist}(x,y)<\frac{r_x}{2}\}.$ Then $U$ is an open subset. The
subset $\bar M\setminus {U}$ as well as the closure of $U$, denoted by $\overline{U}$, are both compact. 
From above we see $\lambda(\Delta v g-\nabla^2 v)\in \Gamma$ in $\overline{U}$ and there is a uniformly positive constant $a_0$ such that $|\nabla v|^2\geq a_0$ in $\bar M\setminus {U}$.

If $x\in \bar M\setminus U$, then $U[\underline{u}] \geq
\frac{1}{2}a_0A^2e^{Av} g \mbox{ for } A\gg1. $

For $x\in \overline{U}$,  $\lambda(\frac{1}{n-2}G_{\tilde{g}})=\lambda(\frac{1}{n-2}G_g+U[\underline{u}])\in \Gamma$
in $\overline{U}$ when $A\gg1$, since  $U[\underline{u}]\geq Ae^{Av}(\Delta v g-\nabla^2v)$ and 
$\lambda(\Delta v g-\nabla^2 v)\in \Gamma \mbox{ in } \overline{U}$.

%The proof is complete.
\end{proof}
 
 \begin{proof}
[Proof of Theorem \ref{thm1-conformal}]
%It suffices to find admissible conformal metrics.
%Clearly, %$\mathcal{A}(\bar \Omega, g_\Omega,\Gamma)\neq \emptyset$. 
Admissible conformal metrics exist, since there are many smooth functions $v$ on $\bar\Omega$ such that $\nabla  v\neq 0$ everywhere.
%When $\Omega$ is a smooth, bounded domain in $\mathbb{R}^k$, $1\leq k\leq n$. 
%The $v$ can be chosen as a smooth strongly convex function on $\Omega$.  
%According to Proposition \ref{lemma5-main}  admissible conformal metrics exist.
 
\end{proof}

 \begin{remark}
 \label{remark1-nonzero}
 If $\bar M$ admits a $C^1$-gradient field without zero, then $\mathcal{A}(\bar M,[g],\Gamma)\neq \emptyset$ for each conformal class $[g]$ on $\bar M$.
 Furthermore, if a manifold $Y$ supposes a function $v$ with $d v\neq 0$ everywhere, then so does  $X\times Y$.
 % for each smooth manifold $X$.
 \end{remark}

 \section{Conformal deformation of general modified Schouten tensors}
 \label{general-existence}

% In this section,  for %modified Schouten tensors $S_g^\tau$ with 
% $\tau>1+(n-2)(1-\kappa_\Gamma\vartheta_{\Gamma})$, %$ i.e., $\tau$ satisfies \eqref{assumption-tau}, 
%we are going to find conformal metrics $\tilde{g}=e^{2u}g$ to solve \eqref{conformal-equ0-general}.
%  \begin{equation} \label{assumption-tau}\begin{aligned}\tau>1+(n-2)(1-\kappa_\Gamma\vartheta_{\Gamma}).
%\end{aligned}\end{equation}
 We assume throughout this section that
 %$(M,g)$  is a compact Riemannian manifold of dimension $n\geq 3$ with smooth boundary, 
 $\psi\in C^\infty(\bar M)$, $\psi>0 \mbox{ in } \bar M$, and $f$ satisfies %\eqref{elliptic}, 
 \eqref{concave}, \eqref{homogeneous-1},  \eqref{homogeneous-1-buchong2}.

% \subsection{Some computations}

Let $h=h(\mathrm{d})$, then 
 \begin{equation}
 \label{1212121}
\begin{aligned}
\,& \frac{\tau-1}{n-2}\Delta h g -\nabla^2 h +\frac{\tau-2}{2}|\nabla h|^2 g+dh\otimes dh
%\\=\,&  \frac{\tau-1}{n-2}(h'_k\Delta \mathrm{d}+h''_k|\nabla\mathrm{d}|^2)g-(h'_k\nabla^2 \mathrm{d}+h''_kd\mathrm{d}\otimes d\mathrm{d})\\ \,&+\frac{\tau-2}{2}h'^2|\nabla\mathrm{d}|^2g+h'^2 d\mathrm{d}\otimes d\mathrm{d}
\\=\,&
\left(\frac{(\tau-1)h''}{n-2}+\frac{(\tau-2)h'^2}{2}\right) |\nabla\mathrm{d}|^2g
+h' \left(  \frac{\tau-1}{n-2}\Delta \mathrm{d}g-\nabla^2 \mathrm{d} \right)
\\\,&
 +(h'^2-h'') d\mathrm{d}\otimes d\mathrm{d}. % \nonumber
\end{aligned}
\end{equation}
 Recall that under the conformal change $\tilde{g}=e^{2u}g$,  $S^\tau_{\tilde{g}}$ obeys the following formula
 $$S_{\tilde{g}}^\tau=S_g^\tau + \frac{\tau-1}{n-2}\Delta u g-\nabla^2u+\frac{\tau-2}{2}|\nabla u|^2 g+du\otimes du.$$
 %Next, we give some computations. Let %$\beta=1$,
Let  $h_k=  \log\frac{k\delta^2}{k\mathrm{d}+\delta^2}$.  
The straightforward computation gives $h_k''=h_k'^2$ and
\begin{equation}
\begin{aligned}
\,& \frac{\tau-1}{n-2}\Delta h_k g -\nabla^2 h_k +\frac{\tau-2}{2}|\nabla h_k|^2 g+dh_k\otimes dh_k 
%\\=\,&  \frac{\tau-1}{n-2}(h'_k\Delta \mathrm{d}+h''_k|\nabla\mathrm{d}|^2)g
%-(h'_k\nabla^2 \mathrm{d}+h''_kd\mathrm{d}\otimes d\mathrm{d})\\ \,&
%+\frac{\tau-2}{2}h'^2|\nabla\mathrm{d}|^2g+h'^2 d\mathrm{d}\otimes d\mathrm{d}
%\\=\,&
%\left(\frac{(\tau-1)h''_k}{n-2}+\frac{(\tau-2)h_k'^2}{2}\right) |\nabla\mathrm{d}|^2g
%+h'_k \left(  \frac{\tau-1}{n-2}\Delta \mathrm{d}g-\nabla^2 \mathrm{d} \right)
%\\\,& -(h''-h'^2) d\mathrm{d}\otimes d\mathrm{d}
\\ =\,&
%\left(\frac{\tau+1-n}{n-2}+\frac{\tau}{2}\right)
\frac{n\tau+2-2n}{2(n-2)} (h_k')^2 |\nabla\mathrm{d}|^2g
+h'_k \left(  \frac{\tau-1}{n-2}\Delta \mathrm{d}g-\nabla^2 \mathrm{d} \right). \nonumber
\end{aligned}
\end{equation}
It is easy to see $0<\vartheta_{\Gamma}\leq \frac{1}{n}$ and $0\leq\kappa_\Gamma\leq n-1$.
%($\vartheta_{\Gamma}=\frac{1}{n}$ and $\kappa_\Gamma=n-1$ cannot occur %simultaneity 
% simultaneously; otherwise $f_i(\lambda)=\frac{1}{n}\sum_{j=1}^n f_j(\lambda)$ in $\Gamma$, $\forall i$).  
As a result, \eqref{assumption-tau} yields
\begin{equation}
\label{positive-sign1}
\begin{aligned}
%\frac{\tau+1-n}{n-2}+\frac{\tau}{2}=\frac{n\tau+2-2n}{2(n-2)}>0,
n\tau+2-2n>0.
\end{aligned}
\end{equation}
%which verifies the positive sign of coefficient of %quadratic gradient term. 
% $(h_k')^2|\nabla\mathrm{d}|^2 g$.

\begin{remark}
\label{remark-important1}
The importance of assumption \eqref{assumption-tau}:
\begin{enumerate}
\item  It  assures  \eqref{assumption-4} 
(corresponding to $0<\varrho<\frac{1}{1-\kappa_\Gamma\vartheta_{\Gamma}}$ in \eqref{n-1-equation1}).
\item It implies the positive sign of
 coefficient of %quadratic gradient term. 
 $(h_k')^2|\nabla\mathrm{d}|^2 g$ as shown in \eqref{positive-sign1}.
\end{enumerate}
\end{remark}
Following the outline of proof in Section \ref{proofofmainresult}, we can prove existence result asserted 
Theorem \ref{thm0-general}.
%For any complete conformal metric satisfying \eqref{conformal-equ0-general}, % in Theorem \ref{thm0-general}, 
%we have 

For asymptotic property, we construct 
the local subsolution as follows:
 \begin{equation}
%\label{1115-0-2}
\begin{aligned}
h_{k,\epsilon,\delta}(\mathrm{d})=\log\frac{k}{k\mathrm{d}+1}
+\frac{1}{2}\log\frac{(1-\epsilon)^2(n\tau+2-2n)}{2(\sup_{\partial M}\psi+\epsilon)}+
 \frac{1}{\mathrm{d}+\delta}- \frac{1}{\delta}. \nonumber
\end{aligned}
\end{equation}
%where $0<\beta\leq 1$ is determined later. 
Computation shows
 \begin{equation}
\begin{aligned}
\,&\frac{\tau-1}{n-2}h_{k,\epsilon,\delta}''+\frac{\tau-2}{2}h_{k,\epsilon,\delta}'^2
= \frac{n\tau+2-2n}{2(n-2)}\frac{k^2}{(k\mathrm{d}+1)^2}
\\\,&
+ \frac{2(\tau-1)}{(n-2)(\mathrm{d}+\delta)^3}
%\\\,&
+\frac{\tau-2}{2}\left[\frac{1}{(\mathrm{d}+\delta)^4}+\frac{2k}{(\mathrm{d}+\delta)^2(k\mathrm{d}+1)}\right]. \nonumber
\end{aligned}
\end{equation}
Therefore, if $\tau\geq 2$ then $$\frac{\tau-1}{n-2}h_{k,\epsilon,\delta}''+\frac{\tau-2}{2}h_{k,\epsilon,\delta}'^2
\geq  \frac{n\tau+2-2n}{2(n-2)}\frac{k^2}{(k\mathrm{d}+1)^2}.$$
Since $\frac{k}{k\mathrm{d}+1}\geq \frac{1}{\mathrm{d}+\delta}$ if $k\geq\frac{1}{\delta}$, we have
 $h_{k,\epsilon,\delta}'^2-h_{k,\epsilon,\delta}''= \frac{2k }{(\mathrm{d}+\delta)(k\mathrm{d}+1)} \geq 0.$
%\begin{equation} \begin{aligned}
%h_{k,\epsilon,\delta}'^2-h_{k,\epsilon,\delta}''= \frac{2k }{(\mathrm{d}+\delta)(k\mathrm{d}+1)} \geq 0. \nonumber
%\end{aligned}\end{equation}

As in \eqref{buchong1}, for $0<\delta\ll1$ and $k\geq \frac{1}{\delta}$, we derive by \eqref{1212121} that
\begin{equation} \begin{aligned}
f(\lambda(S_{e^{2h_{k,\epsilon,\delta}}g}^\tau)\geq \frac{\psi}{n-2}e^{2h_{k,\epsilon,\delta}} \mbox{ in } \Omega_\delta. \nonumber
\end{aligned}\end{equation}
Thus we obtain the following asymptotic property.
 \begin{proposition}
\label{proposition-asymptotic-general}
%Let $u_\infty$ be the solution obtained in Theorem  \ref{thm0-conformal}.
Let $\tau>1+(n-2)(1-\kappa_\Gamma\vartheta_{\Gamma})$ and $\tau\geq 2$.
Let $\widetilde{g}_\infty=e^{2\widetilde{u}_\infty}g$ be a complete conformal metric satisfying \eqref{conformal-equ0-general}, 
$\lambda_{\widetilde{g}_\infty}(S^\tau_{\widetilde{g}_\infty})\in \Gamma$ in $M$,
% claimed in Theorem  \ref{thm0-general}, 
then 
\begin{equation}
\label{asymptotic-rate2-general}
\begin{aligned}
 \frac{1}{2}\log\frac{n\tau +2-2n}{2 \sup_{\partial M}\psi }\leq\lim_{x\rightarrow\partial M} (\widetilde{u}_\infty(x)+\log \mathrm{d}(x))\leq
  \frac{1}{2}\log\frac{n\tau +2-2n}{2 \inf_{\partial M}\psi }. \nonumber
\end{aligned}
\end{equation}
Furthermore, if $\psi$ is constant when restricted to boundary, i.e. %$\inf_{\partial M}\psi=\sup_{\partial M}\psi=\Lambda$
 $\psi|_{\partial M}\equiv \Lambda$, then
\begin{equation}
\label{asymptotic-rate3-general}
\begin{aligned}
\lim_{x\rightarrow\partial M} (\widetilde{u}_\infty(x)+\log \mathrm{d}(x))=\frac{1}{2}\log\frac{n\tau+2-2n}{2 \Lambda} \nonumber
\end{aligned}
\end{equation} 
and the %(admissible) 
 complete conformal metric  %satisfying  \eqref{conformal-equ0-general}
  is unique.
\end{proposition}

Similarly, we define as in \eqref{A(M)-def} a conformal  invariance
\begin{equation}
\begin{aligned}
\mathcal{A}^\tau(\bar M, [g],\Gamma)=\left\{u\in C^2(\bar M): \lambda_{\tilde{g}}(\frac{\tau-1}{n-2}\Delta_{\tilde{g}} u \tilde{g}-\tilde{\nabla}^2 u) \in \Gamma \mbox{ on }   \mathcal{C}(u) \mbox{ for } \tilde{g}\in [g] \right\}  \nonumber
\end{aligned}
\end{equation}
as above $\mathcal{C}(u)$ is the critical set of $u$. For $\tau\geq2$ satisfying \eqref{assumption-tau} the conformal metric satisfying \eqref{admissible-metric23} exists if
 $\mathcal{A}^\tau(\bar M, [g],\Gamma)\neq\emptyset$.
%Let $v\in \mathcal{A}^\tau(\bar M, [g],\Gamma)$. Assume $v<0$. Let $h=-\epsilon v$,  $\underline{u}_k=  \log\frac{k}{kh+1}$, then 
%\begin{equation}\begin{aligned}
%\,& \frac{\tau-1}{n-2}\Delta \underline{u}_k g -\nabla^2  \underline{u}_k +\frac{\tau-2}{2}|\nabla  \underline{u}_k|^2 g+d \underline{u}_k\otimes d \underline{u}_k \\ =\,&
%\frac{n\tau+2-2n}{2(n-2)} \frac{k^2}{(kh+1)^2} |\nabla h|^2g- \frac{k}{kh+1}\left(  \frac{\tau-1}{n-2}\Delta h g-\nabla^2 h \right). \end{aligned}\end{equation}
%there is an $u\in C^2(\bar M)$ satisfying for any critical point $x_0$ with  $\nabla u(x_0)=0$,
%$\lambda(\frac{\tau-1}{n-2}\Delta u g-\nabla^2 u)(x_0)\in \Gamma$. 
As a consequence, we obtain Theorem \ref{coro73}.
Furthermore, we obtain existence result for conformal metrics with prescribed boundary metric.

  \begin{theorem}
  Suppose %\eqref{elliptic}, 
  \eqref{concave}, \eqref{homogeneous-1}, \eqref{homogeneous-1-buchong2}, \eqref{assumption-tau}, \eqref{admissible-metric23} hold.
 For a Riemannian metric $h$ on $\partial M$ which is conformal to $g|_{\partial M}$, there is a smooth conformal 
 metric $\tilde{g}$, satisfying  \eqref{conformal-equ0-general} and $\lambda_{\tilde{g}}(S^\tau_{\tilde{g}})\in \Gamma \mbox{ in } \bar M$, with prescribed boundary condition $\tilde{g}|_{\partial M}=h$. %Moreover, $\lambda_{\tilde{g}}(S^\tau_{\tilde{g}})\in \Gamma \mbox{ in } \bar M$.
 
 \end{theorem}

% \begin{remark}
%   Theorems \ref{thm0-conformal} and \ref{thm2-conformal} reveal the existence of conformal metrics that are determined by Einstein tensor connects closely with critical points of gradient fields.
 
% \end{remark}
 
  \begin{remark}
  With assumption  \eqref{assumption-tau} holds,  the condition  $\tau\geq 2$ is  satisfied  if either 
  $0\leq\kappa_\Gamma\leq n-3$, or $n=4$ and $\kappa_\Gamma= 2$.
%\begin{enumerate}
%\item $0\leq\kappa_\Gamma\leq n-3$.
%\item $n=4$ and $\kappa_\Gamma= 2$.
%\end{enumerate}
 
 \end{remark}

%$\mathcal{\widetilde{A}}(\bar M)\neq\emptyset$,
% where 
%\begin{equation}\begin{aligned}
%\mathcal{\widetilde{A}}(\bar M):=\left\{u\in C^2(\bar M): \lambda(\frac{\tau-1}{n-2}\Delta u g-\nabla^2 u)\in \Gamma  \mbox{ at $x_0$ where } \nabla u(x_0)=0 \right\}.  \nonumber
%\end{aligned}\end{equation} 

\section{Hessian and Weingarten equations}
\label{section3}

 %\vspace{0.5mm}

In this section we give another application of partial uniform ellipticity. 
 Corollary \ref{prop1-operator},
  % This interesting proposition, 
   couple with  \eqref{sumfi}, %assertion % is a consequence of Theorem \ref{yuan-k+1} and 
deduces that if $f$ satisfies %\eqref{elliptic},
 \eqref{concave} and \eqref{addistruc}, then for any $\lambda\in \partial \Gamma^\sigma$, there is an $\theta>0$ depending on $\sigma$ so that
 \begin{equation}
\label{key2-yuan}
\begin{aligned}
f_i(\lambda)\geq \theta +\theta \sum_{j=1}^n f_j(\lambda)  \mbox{ if } \lambda_i\leq 0.
\end{aligned}
\end{equation}
%in very general context.
Combining with the (interior and global) gradient estimates obtained previously in  \cite{LiYY1991,Urbas2002},
 we briefly discuss Hessian  and Weingarten equations.

\subsection{Hessian equations}

  %fully nonlinear elliptic equations 
Let's consider Hessian equations on compact Riemannian manifolds with or without boundary
 \begin{equation}
\label{hessianequ1-riemann}
\begin{aligned}
f(\lambda(\nabla^2 u+A))=\psi,
\end{aligned}
\end{equation}
 where $\psi\in C^\infty(\bar M)$ %($\bar M=M$ if $\partial M=\emptyset$)
  and
 $A$ is a smooth symmetric $(0,2)$-type tensor.
 
 \vspace{0.5mm}
 % \begin{equation}
%  \label{sum-infty}
%\begin{aligned}
%\lim_{|\lambda|\rightarrow +\infty}\sum_{i=1}^n f_i(\lambda)=+\infty,
%\end{aligned}
%\end{equation}
The gradient estimate is %a key ingredient left open for %derive \textit{a priori estimates} and
 important for the study of Hessian equations.
However, it is rather hard %to derive  gradient estimate 
on curved manifolds.
In \cite{LiYY1990} the gradient estimate was obtained under  assumptions that  $\lim_{|\lambda|\rightarrow +\infty}\sum_{i=1}^n f_i(\lambda)=+\infty$ and  the Riemannian manifold admits nonnegative sectional curvature, later was extended by \cite{Urbas2002} %to Hessian equations 
 with replacing such two restrictions by \eqref{key2-yuan}.
 %that is fulfilled for $f$ satisfying \eqref{elliptic}, \eqref{concave}, \eqref{addistruc}. 
 % As an application of
%Combining Proposition \ref{prop1-operator} and \eqref{sumfi}, 
Thus one can follow \cite{Urbas2002} to 
 obtain the gradient bound for solutions to  \eqref{hessianequ1-riemann}
 with existence of 
  $\mathcal{C}$-subsolution satisfying
 \begin{equation}
\label{existenceofsubsolution2}
\begin{aligned}
\lim_{t\rightarrow +\infty}f(\lambda(\nabla^2\underline{u}+A)+te_i)>\psi, 
\mbox{ in } \bar M \mbox{ for each } i=1,\cdots, n,  \nonumber
\end{aligned}
\end{equation}
where  %=\sqrt{-1}\partial\overline{\partial}\underline{u}+\chi(z,\partial\underline{u},\overline{\partial} \underline{u})$,
  $e_i$ is the $i$-$\mathrm{th}$   standard basis vector  (see \cite{Gabor}).
 %The proof also uses Proposition 5 in \cite{Gabor}. %(analogous to Lemma \ref{guan2014} below).
 %Also one can verify that an \textit{admissible} subsolution satisfying \eqref{subsolution-1} below is indeed a $\mathcal{C}$-subsolution.
\begin{proposition}
 \label{thm1-gradient}
 Let $u$ be an admissible solution $u\in C^3(M)\cap C^1(\bar M)$ to equation \eqref{hessianequ1-riemann} with nondegenerate 
 right-hand side $\inf_M \psi>\sup_{\partial\Gamma}f$.
In addition to \eqref{concave},  \eqref{elliptic} and \eqref{addistruc}, we assume  
  there exists a $C^2$-smooth $\mathcal{C}$-subsolution $\underline{u}$. %to the equation.
 Then 
 \begin{equation}
\begin{aligned}
\sup_{M}|\nabla u|\leq C(1+\sup_{\partial M}|\nabla u|), \nonumber
\end{aligned}
\end{equation}
where $C$ depends on $|\psi|_{C^1(\bar M)}$, $|\underline{u}|_{C^2(\bar M)}$ and other known data. Moreover, $C$ is independent of $(\delta_{\psi,f})^{-1}$, where $\delta_{\psi,f}=\inf_M \psi-\sup_{\partial\Gamma}f$.
 \end{proposition}

%Indeed, the proof % used in this paper also %extends Theorem \ref{thm1-gradient}
%also applies to more general Hessian type equations \eqref{hessianequ2-riemann}
% with certain concave/convex, asymptotic growth conditions on $\nabla u$ %same as that 
% imposed in \cite{Guan2015Jiao}. % (we omit it here).
%furthermore, they also obtain interior second order estimate for equation  \eqref{hessianequ2-riemann}  under certain assumptions (on $A(x,z,p)$ and $\psi(x,z,p)$) including the   MTW  condition on   $A(x,z,p)$    to derive interior regularity for optimal transports  \cite{MaTrudingerWang2005}.
%In next section the MTW condition is removed for certain Hessian type equations.
 
% \vspace{0.5mm}
For Dirichlet problem, with a subsolution assumption
 \begin{equation}
  \label{subsolution-1}
  \begin{aligned} 
  f(\lambda(\nabla^2\underline{u}+A))\geq\psi, \mbox{  }
   \lambda(\nabla^2 \underline{u}+A)\in\Gamma \mbox{ in } \bar M,  \mbox{  } 
 \underline{u}=\varphi \mbox{ on } \partial M, \nonumber
  \end{aligned} 
  \end{equation}
  the gradient estimate was proved in new version of \cite{Guan14} for Hessian equations  
  %of form \eqref{hessianequ1-riemann} 
   with a different method; while if the background space is a closed Riemannian manifold,
Proposition \ref{thm1-gradient}
 was stated in \cite{Gabor} where the proof is based on %the quantitative second  estimate
\begin{equation}
\label{quantitative-second-estimate}
\begin{aligned}
\sup_{M}\Delta u \leq C(1+\sup_{M}|\nabla u|^2)
\end{aligned}
\end{equation}
which extends Hou-Ma-Wu's \cite{HouMaWu2010} second estimate for complex $k$-Hessian equations on closed 
K\"ahler manifolds.
The estimate \eqref{quantitative-second-estimate} was extended by the author \cite{yuan2017} to compact Riemannian manifolds with concave boundary
by establishing a quantitative boundary estimate of the form
\begin{equation}
\label{quantitative-boundary-estimate}
\begin{aligned}
\sup_{\partial M}\Delta u \leq C(1+\sup_{M}|\nabla u|^2).
\end{aligned}
\end{equation}
%(This holds for equations on Hermitian manifolds with real analytic Levi flat/holomorphically flat boundary).
Indeed, by using a strategy further developed   in \cite{yuan2019},
 such a quantitative boundary estimate 
% \eqref{quantitative-boundary-estimate} 
 can be  extended to more general case that the
principal curvatures of $\partial M$, $\kappa_1, \cdots,\kappa_{n-1}$, satisfy
 \begin{equation}
\label{bdry-assum1}
\begin{aligned}
(-\kappa_1,\cdots,-\kappa_{n-1})\in \overline{\Gamma}^\infty_{\mathbb{R}^{n-1}}\mbox{ in } \partial M,
\end{aligned}
\end{equation}
where $\overline{\Gamma}^{\infty}_{\mathbb{R}^{n-1}}$ is the closure of
$\Gamma^{\infty}_{\mathbb{R}^{n-1}}$. 
 (In \cite{yuan2017,yuan2019} the author dealt primarily with %fully nonlinear elliptic equations
 Dirichlet problem on Hermitian manifolds, in which \eqref{quantitative-boundary-estimate} 
 was established for complex fully nonlinear equations. 
 In fact the gradient estimate in complex variables is much more hard to handle). 
 % Also,  please see \cite{yuan2020,yuan2019-3} for follow-up work).
%$\Gamma_{\infty}=\{(\lambda_1,\cdots,\lambda_{n-1}): (\lambda_1,\cdots,\lambda_{n-1},t )\in \Gamma \mbox{ for some } t>0\}.$
%\vspace{0.5mm}
It is a remarkable fact that, besides with proposing a new and effective
 approach to gradient estimate,  the quantitative boundary estimate %obtained in \cite{yuan2017,yuan2019} 
indeed applies to the Dirichlet problem for  equation \eqref{hessianequ1-riemann}  %on Riemannian manifolds
 with degenerate right-hand side $\inf_M\psi=\sup_{\partial\Gamma} f$.
 %  (see Theorem \ref{thm1-solvability-Hessian} below).
  To solve degenerate equations,   the \textit{admissible} subsolution shall be replaced by a strictly \textit{admissible} subsolution with $f(\lambda(\nabla^2\underline{u}+A))>\psi$ in $\bar M$,  $\underline{u}=\varphi$ on $\partial M$.
 % \begin{equation} \begin{aligned}
%\,&  f(\lambda(\nabla^2\underline{u}+A))>\psi \mbox{ in } \bar M,  \,& \underline{u}=\varphi \mbox{ on } \partial M. \nonumber
 % \end{aligned}\end{equation}
  %For more 
  %literature, please see \cite{yuan2020,yuan2019-3}.
%\begin{equation} \label{degenerate1} \begin{aligned}
%\inf_M\psi=\sup_{\partial\Gamma} f.
%\end{aligned}\end{equation}
%and   imposes some new regularity assumptions on boundary and boundary data as well.
%(For degenerate equations we shall assume in addition that $f\in C^\infty(\Gamma)\cap C(\overline{\Gamma})$, where $\overline{\Gamma}$ denotes the closure of $\Gamma$).

\vspace{0.5mm}
  As a result, %combining with \cite{Urbas2002,Guan12a} and \cite{yuan2019} (degenerate),
   combining with the results in \cite{Urbas2002} and \cite{Guan12a} for nondegenerate equations, one can solve the Dirichlet problem. 
   In addition, for degenerate equations, the arguments of \cite{yuan2019} is also needed.
   % to prove 
   %\eqref{quantitative-boundary-estimate} %a quantitative boundary estimate
   % that is independent of $(\delta_{\psi,f})^{-1}$.
      % with assuming %the existence of
  %  \textit{admissible} subsolutions. 
    %For nondegenerate equations,  we remove additional assumptions on nonnegative sectional curvature (see \cite{LiYY1990}) and concave boundary (see  \cite{yuan2017}). 

 \begin{theorem}
 \label{thm1-solvability-Hessian}
Let $(M,g)$ be a Riemannian manifold with smooth boundary.
Let $\psi$, $\varphi$ be smooth functions. %$A$ be a smooth symmetric $(0,2)$-type tensor.
In addition to \eqref{concave}, \eqref{elliptic} and \eqref{addistruc}, we assume %that 
$\inf_M \psi>\sup_{\partial\Gamma}f$ and there is a $C^{3,1}$-smooth admissible subsolution $\underline{u}$.
% satisfying \eqref{subsolution-1}.
Then there is a unique smooth admissible function $u$ with $u|_{\partial M}=\varphi$ to solve
%the Hessian
% equation 
\eqref{hessianequ1-riemann}.

% \begin{remark} \label{remark1-degenerate}
Furthermore, if  $\partial M$ satisfies \eqref{bdry-assum1} and $\underline{u}$ 
 is a strict subsolution, then the equation with degenerate right-hand side  %$\inf_M\psi=\sup_{\partial\Gamma} f$  
 admits a weak solution $u$ with $u|_{\partial M}=\varphi$, $u\in C^{1,1}(\bar M)$, 
 $\Delta u\in L^\infty(\bar M)$ and $\lambda(\nabla^2u+A)\in \overline{\Gamma}$.
% \end{remark}

 \end{theorem}

\subsection{Weingarten equation}
Let $S_u:=\{(x',x_{n+1})\subset\mathbb{R}^{n+1}:   x_{n+1}=u(x'),   x'=(x_1,\cdots,x_n)\}$ be the graph of $u$,
 $\nu=(\nu_1,\cdots,\nu_{n+1})$ be the downward normal vector of $S_u$. 
 We denote $\kappa=(\kappa_1,\cdots,\kappa_n)$ the vector of 
principal curvatures of the graph $S_u$. 
The Weingarten equation is given by 
\begin{equation}
\label{weingartenequ1}
\begin{aligned}
f(\kappa)=\psi(x,u,\nu)
\end{aligned}
\end{equation}
where  $\psi$ is in  class of $C^1$
and satisfies %that %there are
for some positive constants $A_0$, $A_1$, $A_2$, %such that 
\begin{equation}
\label{weingartenequ-condition1}
\begin{aligned}
\psi_u\geq 0, \mbox{  } |\psi|\leq A_0, \mbox{  }  |\nabla\psi|\leq A_1, \mbox{  }  \psi\geq  A_2>\sup_{\partial\Gamma}f.
\end{aligned}
\end{equation}
%In addition, $f$ satisfies 

Interior gradient estimate for admissible solutions of Weingarten equations
 with $f=(\sigma_k)^{\frac{1}{k}}$  has been obtained by Korevaar \cite{Korevaar1987},
while it was generalized by Li \cite{LiYY1991} to more general Weingarten equations satisfying \eqref{key2-yuan} and
\begin{equation}
\label{weingartenequ-condition2}
\begin{aligned}
 \liminf_{\lambda\rightarrow0}f(\lambda)>-\infty.
\end{aligned}
\end{equation}
More results concerning Weingarten equations with assumption \eqref{key2-yuan} can be found \cite{ShengUrbasWang-Duke,Trudinger90}.

Since we confirm \eqref{key2-yuan} in general context, 
%when $f$ satisfies \eqref{elliptic}, \eqref{concave} and \eqref{addistruc},
% according to Corollary  \ref{prop1-operator},
%Together with Corollary  \ref{prop1-operator}, 
%Therefore, 
Li's interior estimate gives 
\begin{proposition}
Suppose \eqref{concave}, \eqref{elliptic}, \eqref{addistruc}, \eqref{weingartenequ-condition1} and \eqref{weingartenequ-condition2} hold. Let $\overline{B_1(0)}=\{x'\in \mathbb{R}^n: |x'|\leq 1\}$, and
$u\in C^3(\overline{B_1(0)})$ be a  solution to \eqref{weingartenequ1} with 
  $\kappa\in \Gamma$ in  graph $S_u$. In addition $u|_{\overline{B_1(0)}}<0$, $u(0)=-A_3$. Then
there is a positive constant $C$ depending only on $f$, $n$, $A_0$, $A_1$, $A_2$ and $A_3$,  such that
$|\nabla u(0)|\leq C.$
\end{proposition}

\bigskip

%\begin{appendix}
%\section{A characterization of cones being of type 2}
%\end{appendix}

%\bigskip

%{\bf Acknowledgement.}

%\subsection*{Acknowledgement}
% The author was supported by %the National Natural Science Foundation of China, Grant No.
% NSFC Grant 11801587.

\small
\bibliographystyle{plain}

\end{document}